\documentclass[11pt,a4paper,reqno]{amsart}

\usepackage[pdftitle={Green's function and the strong maximum principle}, pdfauthor={Luigi Orsina and Augusto C. Ponce}]{hyperref}
\usepackage{tikz}

\usepackage{amsfonts}

\usepackage[T1]{fontenc}

\usepackage{enumerate}

\usepackage[abbrev,backrefs]{amsrefs}
\usepackage{mathtools}

\usepackage[utf8]{inputenc}

\usepackage{graphicx}
\usepackage{amssymb}
\usepackage{esint}
\makeatletter
\newdimen\rh@wd
\newdimen\rh@hta
\newdimen\rh@htb
\newbox\rh@box
\def\rh@measure#1{\setbox\rh@box=\hbox{$#1$}\rh@wd=\wd\rh@box \rh@hta=\ht\rh@box}

\def\widecheck#1{\rh@measure{#1}%
  \setbox\rh@box=\hbox{$\widehat{\vrule height \rh@hta width\z@ \kern\rh@wd}$}%
  \rh@htb=\ht\rh@box \advance\rh@htb\rh@hta \advance\rh@htb\p@
  \ooalign{$\vrule height \ht\rh@box width\z@ #1$\cr
           \raise\rh@htb\hbox{\scalebox{1}[-1]{\box\rh@box}}\cr}}
\makeatother

\usepackage{palatino}

\usepackage{color}
\usepackage{todonotes}

\usepackage[nameinlink]{cleveref}

\newcommand{\abs}[1]{\mathopen\lvert#1\mathclose\rvert}

\newcommand{\Bigabs}[1]{\Bigl\lvert#1\Bigr\rvert}
\newcommand{\biggabs}[1]{\biggl\lvert#1\biggr\rvert}
\newcommand{\norm}[1]{\mathopen\lVert#1\mathclose\rVert}
\newcommand{\bignorm}[1]{\mathopen\big\lVert#1\mathclose\big\rVert}

\newcommand{\N}{{\mathbb N}}
\newcommand{\R}{{\mathbb R}}

\newcommand{\cH}{\mathcal{H}}
\newcommand{\cM}{\mathcal{M}}

\DeclareMathOperator{\diam}{diam}

\DeclareMathOperator{\sgn}{sgn}

\DeclareMathOperator{\capt}{cap}

\newcommand{\dif}{\,\mathrm{d}}
\newcommand{\quasi}[1]{\widehat{#1}}

\newcommand{\lc}{_\mathrm{c}}
\newcommand{\ld}{_\mathrm{d}}
\newcommand{\la}{_\mathrm{a}}
\newcommand{\ls}{_\mathrm{s}}
\newcommand{\loc}{_\mathrm{loc}}

\newcommand{\meas}[1]{\left| #1 \right|}

\theoremstyle{plain}
\newtheorem{proposition}{Proposition}[section]
\newtheorem{lemma}[proposition]{Lemma}
\newtheorem{theorem}[proposition]{Theorem}
\newtheorem{corollary}[proposition]{Corollary}
\newtheorem{definition}{Definition}[section]

\theoremstyle{definition}

\newtheorem{remark}[proposition]{Remark}

\theoremstyle{remark}
\newtheorem{example}{Example}[section]

\newtheorem*{Claim}{Claim}

\numberwithin{equation}{section}

\title[Green's function and the strong maximum principle]{On the nonexistence of Green's function\\ and failure of the strong maximum principle}

\author{Luigi Orsina}
\address{
Luigi Orsina\hfill\break\indent
``Sapienza'' Universit\`a di Roma\hfill\break\indent
Dipartimento di Matematica \hfill\break\indent
P.le A.~Moro 2\hfill\break\indent
00185 Roma, Italy}
\email{orsina@mat.uniroma1.it}

\author{Augusto C. Ponce}
\address{
Augusto C. Ponce\hfill\break\indent
 Université catholique de Louvain\hfill\break\indent
 Institut de Recherche en Mathématique et Physique\hfill\break\indent
 Chemin du cyclotron 2, L7.01.02\hfill\break\indent
1348 Louvain-la-Neuve, Belgium}
\email{Augusto.Ponce@uclouvain.be}

\begin{document}

%~ \date{\today}

\begin{abstract}
Given any Borel function \(V : \Omega \to [0, +\infty]\) on a smooth bounded domain \(\Omega \subset \R^{N}\),{}
we establish that the strong maximum principle for the Schrödinger operator \(-\Delta + V\) in \(\Omega\) holds in each Sobolev-connected component of \(\Omega \setminus Z\), where \(Z \subset \Omega\) is the set of points which cannot carry a Green's function for \(- \Delta + V\).{}
More generally, we show that the equation \(- \Delta u + V u = \mu\) has a distributional solution in \(W_{0}^{1, 1}(\Omega)\) for a nonnegative finite Borel measure \(\mu\) if and only if \(\mu(Z) = 0\). 
\end{abstract}

\subjclass[2010]{Primary: 35J10, 35B05, 35B50; Secondary: 31B15, 31B35, 31C15}

\keywords{Schr\"{o}dinger operator, strong maximum principle, measure datum, singular potential}

\maketitle
%\tableofcontents

\section{Introduction and main results}

Let \(\Omega \subset \R^{N}\) be a smooth bounded connected open set and let \(V : \Omega \to [0, +\infty]\) be a Borel function.
The \emph{weak} maximum principle  for the Schr\"odinger operator 
\(-\Delta + V\) ensures that if
\(w \in W_{0}^{1, 2}(\Omega) \cap L^{\infty}(\Omega)\) is a function such that \(Vw \in L^{1}(\Omega)\) and
\begin{equation}
	\label{eqEquationSupersolution}
-\Delta w + Vw = f
\quad \text{in the sense of distributions in \(\Omega\),}
\end{equation}
where \(f \in L^{\infty}(\Omega)\) is nonnegative, then \(w\) must be nonnegative in \(\Omega\).
In this paper, we are interested in the mechanism that guarantees the validity or the 
failure of the \emph{strong} maximum principle for 
\(-\Delta + V\) and whether there is some unifying property that holds regardless of the potential \(V\).
More precisely, for a fixed potential \(V\), we want to understand whether the alternative holds:
\begin{equation}
\label{eqAlternative}
		\text{either}
		\quad
		\text{\(w > 0\) \ in \(\Omega\) \quad{}
		or 
		\quad{}
		\(w \equiv 0\) \ in \(\Omega\),}
\end{equation}
and, when it fails, to identify the location of the zero-set \(\{w = 0\}\) in \(\Omega\) and decide whether there is some form of the strong maximum principle that survives in each component of \(\Omega \setminus \{w = 0\}\). 

To clarify the pointwise meaning of \(w\) in \(\Omega\), we reformulate \eqref{eqAlternative} using the precise representative \(\quasi{w}\).{}
We recall that, by the Lebesgue differentiation theorem, \(\quasi{w} = w\) almost everywhere in \(\Omega\).
Since in our case \(w\) is almost everywhere the difference between a continuous and a bounded superharmonic
 function, every \(x \in \Omega\) is a Lebesgue point of \(w\).{}
Hence, the precise representative can be computed pointwise using the limit 
\[{}
\quasi{w}(x)
= \lim_{r \to 0}{\fint_{B_{r}(x)}{w}},
\]
where \(\fint_{B_{r}(x)} \vcentcolon= \frac{1}{\abs{B_{r}}} \int_{B_{r}(x)}\) denotes the average integral over the ball.

\vspace{3pt}

Observe that \eqref{eqAlternative} classically holds for potentials \(V\) in \(L^{\infty}(\Omega)\) and, more generally, in the Lorentz space \(L^{\frac{N}{2}, 1}(\Omega)\) or in the Kato class \(K(\Omega)\)\,; see \citelist{\cite{Murata:1986} \cite{Zhao:1986} \cite{Devyver:2018}*{Section~3}} or \Cref{exampleLorentz} below.{}
Such a conclusion is no longer true when \(V\) merely belongs to \(L^{p}(\Omega)\) for some \(p \le N/2\)\,:{}

\begin{example}\label{eqExample0}
Given \(1 \le p \le N/2\) and any compact set \(K \subset \Omega\) with finite  \(\cH^{N - 2p}\) Hausdorff measure, we construct in  \cite{Orsina_Ponce:2016}*{Section~6} a nonnegative potential \(V \in L^{p}(\Omega)\), depending on \(K\), such that every nontrivial solution \(w\) associated to the Schrödinger operator \(- \Delta + V\) with nonnegative bounded datum \(f\) satisfies
\begin{equation}
	\label{eqExample0Z}
\{\quasi{w} = 0\} = K.
\end{equation}
\end{example}

We prove in \cite{Orsina_Ponce:2016} that the \(W^{2, p}\)~capacity is the correct way of quantifying the smallness of \(\{\quasi{w} = 0\}\) when one considers the full class of \(L^{p}\) potentials \(V\) for \(p > 1\)\,; the counterpart for \(p = 1\) involves the \(W^{1, 2}\)~capacity, as first identified by Ancona~\cite{Ancona:1979}.{}
By looking at a specific potential \(V\) one may have a zero-set of dimension strictly smaller than \(N-2p\)\,:

\begin{example}\label{eqExample1}
Given \(a \in \Omega\) and \(\alpha \in \R\), let
\begin{equation*}
	V(x) = \frac{1}{\abs{x - a}^{\alpha}}~.
\end{equation*}
When \(\alpha \ge 2\), every nontrivial solution \(w\) with nonnegative bounded datum \(f\) satisfies 
\begin{equation}
	\label{eqExample1Z}
\{\quasi{w} = 0\} = \{a\}.
\end{equation}
To see why \(\quasi{w}(a) = 0\), one relies on the fact that
\[{}
\int_{B_{r}(a)}{Vw}
= o(r^{N-2})
\quad \text{as \(r \to 0\),}
\]
which follows from a scaling argument in the equation \eqref{eqEquationSupersolution} by means of test functions of the form \(\varphi(\frac{x - a}{r})\).
For \(\alpha \ge N\), one may argue differently by observing that \(V\) is not summable in any neighborhood of \(a\) but \(Vw \in L^{1}(\Omega)\).
\end{example}

These examples are particular cases of the general principle implied by our \Cref{theoremMaximumPrinciplePrecise} below that \emph{the zero-set is independent of the solution} whenever \(V \in L^{1}(\Omega)\) or, more generally, when the set where \(V\) fails to be locally summable has \(\cH^{N-1}\) Hausdorff measure zero; see also \Cref{theoremMaximumPrinciplePreciseImproved}.{}
When the singular set of \(V\) is large enough, a splitting of the domain in connected components of \emph{analytic} type may occur:

\begin{example}\label{eqExample2}
In the unit ball \(\Omega = B_{1}(0)\), take	
\begin{equation*}
	V(x)
= \frac{1}{\abs{x_{1} - a}^{\alpha}} + \frac{1}{\abs{x_{1} - b}^{\beta}}\,,
\end{equation*}
where \(-1 < a < b < 1\) and \(x_{1}\) denotes the first component of \(x = (x_{1}, \dots, x_{N})\).{}
Here, the \emph{strength of the singularity modifies the geometric configuration} of the zero-set, even inside the range \(\alpha \ge 1\) and \(\beta \ge 1\) where \(V \not\in L^{1}(B_{1}(0))\)\,:
\begin{enumerate}[(a)]
\item For \(\alpha \ge 2\) and \(1 \le \beta < 2\), every nontrivial solution \(w\) satisfies
	\begin{equation}
		\label{eqExample2Za}
		\{\quasi{w} = 0\} 
		= \{x_{1} \ge a\} \cap B_{1}(0),
	\end{equation}
	and in particular vanishes on a non-empty open set.
\item{}
\label{item-301}
 For a stronger singularity with \(\alpha \ge 2\) and \(\beta \ge 2\), the zero-set of \(w\) depends on \(f\).{}
The reason is that the Dirichlet problem splits in three independent regions inside \(B_{1}(0)\), identified by the conditions 
\[{}
x_{1} < a,
\quad 
a < x_{1} < b
\quad \text{and} \quad{}
x_{1} > b.{}
\]
	In particular, the choice \(f \equiv 1\) yields a smaller zero-set, namely
	\begin{equation}
		\label{eqExample2Zb}
	\{\quasi{w} = 0\} 
	= \bigl( \{x_{1} = a\} \cup \{x_{1} = b\}\bigr) \cap B_{1}(0).
	\end{equation}
\end{enumerate}
These assertions can be established using \cite{Orsina_Ponce:2008}*{Section~9} and are related to the failure of the Hopf boundary lemma.
\end{example}

The previous example illustrates in \eqref{item-301} the fact that strong singularities may be used to confine physical particles in prescribed regions; see \cite{Diaz:2015,Diaz:2017}.
Potentials \(V\) which are \(+\infty\) in some large parts of \(\Omega\) are also of interest and model the presence of impurities or coolers in the domain; see \cite{Rauch_Taylor:1975}.
This is intended to prescribe regions where solutions must vanish:

\begin{example}\label{eqExample3}
	Take 
\begin{equation*}
	V(x)
= \frac{1}{d(x, \overline{\omega})^{\alpha}}\,,
\end{equation*}
where \(\omega \Subset \Omega\) is a smooth open set and \(d(x, \overline{\omega})\) denotes the distance from \(x\) to \(\overline{\omega}\).{}
The strong maximum principle depends on the exponent \(\alpha\)\,:
\begin{enumerate}[(a)]
	\item When \(1 \le \alpha < 2\), there is only the trivial solution \(w \equiv 0\) in \(\Omega\), as an application of the Hopf lemma.
	\item When \(\alpha \ge 2\), nontrivial supersolutions do exist since the Hopf lemma fails pointwise on \(\partial\omega\), see~\cite{Orsina_Ponce:2017}*{Proposition~2.7}, and they all satisfy
\begin{equation*}
\{\quasi{w} = 0\} 
= \overline\omega.
\end{equation*}
\end{enumerate}
\end{example}

To understand the unifying idea behind the strong maximum principle for an arbitrary Borel function \(V : \Omega \to [0, +\infty]\), we first select the subset of points in \(\Omega\) where distributional solutions of the Schrödinger equation must vanish:

\begin{definition}
	Given a Borel function \(V : \Omega \to [0, +\infty]\), the \emph{universal zero-set} \(Z\) associated to \(- \Delta + V\) is the set of points \(x \in \Omega\) such that
	\[{}
	\quasi{w}(x) = 0
	\] 
	for every solution \(w \in W_{0}^{1, 2}(\Omega) \cap L^{\infty}(\Omega)\) of~\eqref{eqEquationSupersolution} for some nonnegative \(f \in L^{\infty}(\Omega)\).
\end{definition}

The universal zero-set depends on \(V\), but to simplify the notation we do not explicit such a dependence.
In our \Cref{eqExample0,eqExample1,eqExample2}, the sets \(Z\) are given by \eqref{eqExample0Z} to \eqref{eqExample2Zb}.
In the latter example, \(\Omega \setminus Z\) has three connected components;
a variant of this case using singularities on infinitely many hyperplanes \(\{x_{1} = a_{i}\}\) with exponents \(\alpha_{i} \ge 2\) yields a set \(\Omega \setminus Z\) with an infinite number of components.
Finally, for \(V\) as in \Cref{eqExample3} one has \(Z = \Omega\) when \(1 \le \alpha < 2\) and \(Z = \overline\omega\) when \(\alpha \ge 2\).{}

We prove later on that \(Z\) is, topologically speaking, a \emph{Sobolev-closed} set in the sense that there exists a nonnegative function \(\xi \in W_{0}^{1, 2}(\Omega)\) such that every \(x \in \Omega\) is a Lebesgue point of \(\xi\) and
\begin{equation}
	\label{eqDefinitionQuasiClosed}
	Z = \bigl\{x \in \Omega : \quasi{\xi}(x) = 0 \bigr\}.
\end{equation}
For example, the solution of \eqref{eqEquationSupersolution} with the characteristic function \(f = \chi_{\Omega \setminus Z}\) satisfies \eqref{eqDefinitionQuasiClosed}, although it is not clear for the moment why this is true nor even why such a solution exists.
These facts are a consequence of \Cref{remarkSetA} and \Cref{corollaryBounded}, respectively. 

We then identify all possible zero-sets of supersolutions of the 
Schrödinger operator \(-\Delta + V\) using the Sobolev-connected components of \(\Omega \setminus Z\).{}
Our main result below provides one with a quantization property for the 
strong maximum principle,
where the relevant singularities of the potential \(V\) for \(-\Delta + V\) are encoded in the universal zero-set \(Z\)\,:

\begin{theorem}
	\label{theoremMaximumPrincipleFull}
	For every Borel function \(V : \Omega \to [0, +\infty]\), the Sobolev-open set \(\Omega \setminus Z\) can be uniquely decomposed as a finite or countably infinite union of disjoint Sobolev-connected-open sets \((D_{j})_{j \in J}\) and any solution \(w \in W_{0}^{1, 2}(\Omega) \cap L^{\infty}(\Omega)\) of the Schrödinger equation~\eqref{eqEquationSupersolution} for nonnegative \(f \in L^{\infty}(\Omega)\) satisfies, in each component \(D_{j}\)\,,
	\begin{center}
		either 
		\quad
		 \(\quasi{w} > 0\) \ in \(D_{j}\) \quad{}
		or 
		\quad{}
		\(\quasi{w} \equiv 0\) \ in \(D_{j}\).
	\end{center}
\end{theorem}

The concepts of \emph{Sobolev-open} and \emph{Sobolev-connected} sets are directly inspired from their classical topological counterparts; see \Cref{definitionQuasiOpen,definitionQuasiConnected}.
In \Cref{remarkQuasiOpen} below, we relate Sobolev-open sets with other notions of open sets that have been extensively investigated in Potential theory.

As we explain in \Cref{sectionProofs}, the Dirichlet problem in \(\Omega\) uncouples to the various Sobolev-connected components of \(\Omega \setminus Z\) and each possibility provided by \Cref{theoremMaximumPrincipleFull} can effectively happen in \(D_{j}\) without interaction with the other parts of \(\Omega \setminus Z\).{}
In particular, for every subset of indices \(L \subset J\), there exists a solution with 
\begin{center}
\(\quasi{w} > 0\) \ in \(\bigcup\limits_{j \in L}{D_{j}}\) 
\quad and \quad
\(\quasi{w} = 0\) \ otherwise.
\end{center}

In dimension \(N = 1\), solutions are continuous and the picture that comes from \Cref{theoremMaximumPrincipleFull} is rather simple when \(Z \ne \emptyset\)\,:
the universal zero-set \(Z\) is relatively closed in \(\Omega\) for the Euclidean topology and then \(\Omega \setminus Z\) is a finite or countable union of disjoint open intervals \(D_{j} = (a_{j}, b_{j})\) where 
\[{}
V \in L^{1}\loc(D_{j})
\quad \text{and} \quad{}
\int_{D_{j}}{V(x) d(x, \partial D_{j}) \dif x}
= +\infty.
\]
Indeed, at an endpoint \(c_{j} = a_{j}\) or \(b_{j}\) inside \(\Omega\), the Hopf lemma in \(D_{j}\) must fail at \(c_{j}\), which is the case if and only if
\[{}
\int_{c_{j}}^{\frac{a_{j} + b_{j}}{2}}{V(x) (x - c_{j}) \dif x}
= +\infty.
\]
One thus recovers \cite{Bertsch_Smarrazzo_Tesei:2015}*{Theorem~2.1}  by Bertsch, Smarrazzo and Tesei; see also \cite{Orsina_Ponce:2017}.

In dimension \(N \ge 2\), one deduces that \(\Omega \setminus Z\) has only one Sobolev-connected component for small \(Z\) using the Intermediate value theorem for Sobolev functions by Van~Schaftingen and Willem~\cite{VanSchaftingen_Willem:2008}:

\begin{corollary}
	\label{theoremMaximumPrinciplePrecise}
	If\/ \(\cH^{N-1}(Z) = 0\), then \(\Omega \setminus Z\) is Sobolev-connected.
	Hence, the zero-set of any solution \(w \in W_{0}^{1, 2}(\Omega) \cap L^{\infty}(\Omega)\) of the Schrödinger equation \eqref{eqEquationSupersolution} with nonnegative \(f \in L^{\infty}(\Omega)\) and \(\int_{\Omega}{f} > 0\) does not depend on \(w\), and 
	\[{}
		\quasi{w}(x) = 0
		\quad \text{if and only if} \quad
		x \in Z.
	\]
\end{corollary}

The proof of \Cref{theoremMaximumPrincipleFull} relies on the fact that the universal zero-set \(Z\) is the set of points where the Schrödinger operator \(-\Delta + V\) is unable to have a Green's function in the sense of distributions.
For example, in the spirit of the seminal work of Bénilan and Brezis~\cite{Benilan_Brezis:2004} one verifies that when \(V\) is the potential in \Cref{eqExample1} with exponent \(\alpha \ge 2\) the equation
\[{}
- \Delta u + V u = \delta_{a}
\]
involving a Dirac mass \(\delta_{a}\) does not have a distributional solution in \(\Omega\), see \cite{Ponce_Wilmet:2017}*{Section~9}, and as we have observed in this case, \(Z = \{a\}\).
More generally, we establish that

\begin{theorem}
	\label{theoremGreen}
Let \(V : \Omega \to [0, +\infty]\) be a Borel function.
Given \(x \in \Omega\), there  exists 
\(G_{x} \in W_{0}^{1, 1}(\Omega) \cap L^{1}(\Omega; V \dif x)\) such that
\[{}
- \Delta G_{x} + VG_{x} = \delta_{x}
\quad \text{in the sense of distributions in \(\Omega\)}
\]
if and only if 
\[{}
x \not\in Z.
\] 
Moreover, one has Green's representation formula
\begin{equation*}
	\quasi{w}(x)
= \int_{\Omega}{G_{x} f}
\quad \text{at each \(x \in \Omega \setminus Z\),}
\end{equation*} 
for every function \(w \in W_{0}^{1, 2}(\Omega) \cap L^{\infty}(\Omega)\) that satisfies \eqref{eqEquationSupersolution} with \(f \in L^{\infty}(\Omega)\).
\end{theorem}

In dimension \(N \ge 3\) and at a point \(x \in \Omega\) where the Newtonian potential
	\begin{equation*}
	\mathcal{N}V : z \in \R^{N} \longmapsto \int_{\Omega}{\frac{V(y)}{\abs{z - y}^{N - 2}} \dif y}
	\end{equation*}
	is finite, the Green's function \(G_{x}\) exists and thus \(x \not\in Z\).{}
	The reason is that the fundamental solution of the Laplacian yields a supersolution for \(-\Delta + V\) with Dirac mass \(\delta_{x}\) and then one can apply the method of sub- and supersolutions from \Cref{sectionMethodSS} below.
	Here are some consequences of this observation:
	
	\begin{example}
		\label{exampleLorentz}
		If \(V \in L^{\frac{N}{2}, 1}(\Omega)\), then by \((L^{\frac{N}{2}, 1}\), \(L^{\frac{N}{N-2}, \infty})\) duality in Lorentz spaces the Newtonian potential \(\mathcal{N}V\) is a bounded function in \(\Omega\).{}
		Thus, 
		\[{}
		Z = \emptyset{}
		\]
		and the classical alternative \eqref{eqAlternative} is satisfied.
	\end{example}

	\begin{example}
		If \(V \in L^{1}(\Omega)\), then  \(\mathcal{N}V\) satisfies the Poisson equation
	\[{}
	- \Delta (\mathcal{N}V)
	= \gamma_{N} V
	\quad \text{in the sense of distributions in \(\Omega\),}
	\]
	where \(\gamma_{N} > 0\).{}
	From classical Potential theory, we have in particular that \(\mathcal{N}V\) can only be infinite on a set of \(W^{1, 2}\) capacity zero.
	Hence, 
	\[{}
	\capt_{W^{1, 2}}{(Z)} = 0
	\]
	and \Cref{theoremMaximumPrinciplePrecise} applies since in this case the Hausdorff dimension of \(Z\) is at most \(N-2\).
	While Ancona's maximum principle from \cite{Ancona:1979} already asserts that \(\{\quasi{w} = 0\}\) has \(W^{1, 2}\)~capacity zero
	for every nontrivial solution of \eqref{eqEquationSupersolution} with nonnegative \(f\), we now have the stronger new property that \(\{\quasi{w} = 0\}\) is actually independent of the solution.
	\end{example}
	
	\begin{example}
	\label{exampleSingularityLp}
		Assume that \(V \in L^{p}(\Omega)\) for some \(1 < p \le N/2\), which is an intermediate case between the two previous examples.
		We now have \(\Delta(\mathcal{N}V) \in L^{p}(\Omega)\) and then, by singular-integral estimates, \(\mathcal{N}V \in W^{2, p}\loc(\Omega)\).{}
		As the exceptional set of \(W^{2, p}\) functions has \(W^{2, p}\)~capacity zero, we deduce that
		\[{}
		\capt_{W^{2, p}}{(Z)} = 0,
		\]
		which combined with \Cref{theoremMaximumPrinciplePrecise} above implies Theorem~1 from our previous work~\cite{Orsina_Ponce:2016}.
	\end{example}

The universal zero-set \(Z\) identifies not only the Dirac masses, but in fact all nonnegative finite Borel measures \(\mu\) for which the Dirichlet problem
\begin{equation}
	\label{eqDirichletProblem}
	\left\{
	\begin{alignedat}{2}
		- \Delta u + Vu & = \mu	&& \quad \text{in \(\Omega\),}\\
		u & = 0	&&	\quad \text{on \(\partial\Omega\),}
	\end{alignedat}
	\right.
\end{equation}
has a \emph{distributional solution}, where by a solution we mean a function \(u \in W_{0}^{1, 1}(\Omega) \cap L^{1}(\Omega; V \dif x)\) which verifies the equation in the sense of distributions in \(\Omega\).{}
Observe that the Green's function \(G_{x}\) arises as a special case of this setting with \(\mu = \delta_{x}\).
We ask that \(u\) belong to the Sobolev space \(W_{0}^{1, 1}(\Omega)\) to encode the zero boundary value of \(u\).{}
An equivalent formulation, without relying on Sobolev spaces, consists of using test functions in the larger class \(C_{0}^{\infty}(\overline{\Omega})\) of smooth functions in \(\overline\Omega\) that vanish on \(\partial\Omega\), not necessarily with compact support in \(\Omega\)\,; see \Cref{sectionMethodSS}.

Our next theorem fully characterizes the nonnegative finite measures for which \eqref{eqDirichletProblem} has a solution:

\begin{theorem}
	\label{theoremGoodMeasuresCharacterization}
	For every Borel function \(V : \Omega \to [0, +\infty]\), the Dirichlet problem \eqref{eqDirichletProblem} has a distributional solution with a nonnegative finite Borel measure \(\mu\) in \(\Omega\) if and only if 
	\[{}
	\mu(Z) = 0.
	\]
\end{theorem}

Observe in particular that \eqref{eqDirichletProblem} has a distributional solution with \(\mu = \chi_{\Omega \setminus Z} \dif x\), since
\[{}
\mu(Z){}
= \int_{Z}{\chi_{\Omega \setminus Z} \dif x} = 0.
\]
When \(Z\) is negligible with respect to the Lebesgue measure, we also deduce the existence of a distributional solution with \(\mu = f \dif x\) for every \(f \in L^{1}(\Omega)\).{}
Then, for \(f \in L^{\infty}(\Omega)\), such a solution belongs to \(W_{0}^{1, 2}(\Omega) \cap L^{\infty}(\Omega)\) and the representation formula in \Cref{theoremGreen} holds.
Although \Cref{theoremMaximumPrincipleFull} only applies to bounded data, the tools we use can be adapted to get its counterpart for general solutions of \eqref{eqDirichletProblem} with nonnegative measures; see \Cref{theoremMaximumPrincipleFullMeasure}.

We rely in this paper on the notion of duality solution of  \eqref{eqDirichletProblem} by Malusa and Orsina~\cite{Malusa_Orsina:1996}, which was inspired from the fundamental work of Littman, Stampacchia and Weinberger~\cite{Littman_Stampacchia_Weinberger:1963}.
In contrast with distributional solutions, duality solutions exist for any finite measure regardless of the potential \(V\).
One reason is that they typically require less test functions, just enough to ensure uniqueness.
In \Cref{sectionVariational,sectionDualityAsDistributions}, we compare both concepts.

A defect of the duality formulation is that the same function can solve the Schrödinger equation for different measures.
It may happen that \(u \equiv 0\) is the duality solution associated to the Dirac mass \(\delta_{x}\) when \(x \in Z\)\,; see \Cref{sectionGreen}.{}
Duality solutions are nevertheless a convenient tool to apply Perron's method and find distributional solutions of \eqref{eqDirichletProblem}.
Such an approach is pursued in \Cref{sectionExistenceBounded}, where we first prove \Cref{theoremGoodMeasuresCharacterization} for \(\mu = \chi_{\Omega \setminus Z} \dif x\).
This is used in \Cref{sectionOrthogonality} to establish an orthogonality principle between the sets \(Z\) and \(\Omega \setminus Z\) which is later applied in \Cref{sectionProofThm1} to prove the existence of distributional solutions of \eqref{eqDirichletProblem} in full generality.

In \Cref{sectionComparison}, we develop another fundamental tool: 
A comparison principle which relates a solution of \eqref{eqDirichletProblem} with nonnegative \emph{measure} datum to another one with nonnegative \emph{bounded} datum.
Namely, we prove that
\begin{equation}
	\label{eqIntroductionComparison}
u \ge w
\quad \text{almost everywhere in \(\Omega\),}
\end{equation}
where \(w \in W_{0}^{1, 2}(\Omega) \cap L^{\infty}(\Omega)\) is the solution of \eqref{eqEquationSupersolution} with right-hand side \(f = H(u)\) for some fixed bounded nondecreasing continuous function \(H\) that is positive on \((0, +\infty)\).{}
Such a function \(H\) can be chosen of the form 
\[{}
H(t) = \epsilon \min{\{t^{\alpha}, 1\}},
\]
for any \(\alpha > 1\) and any \(\epsilon > 0\) small enough, independently of \(u\) and \(V\).{}
Estimate \eqref{eqIntroductionComparison} relates the zero-sets of \(u\) and \(w\), and works as a replacement of the Harnack inequality, which is false for singular potentials \(V\).{}

In \Cref{sectionProofThm1}, we prove \Cref{theoremGoodMeasuresCharacterization,theoremGreen}, where the comparison principle \eqref{eqIntroductionComparison} is used to prove that \(\mu(Z) = 0\) is necessary for the existence of a distributional solution of \eqref{eqDirichletProblem}.
We apply again \eqref{eqIntroductionComparison} in \Cref{sectionGreen} to show that \(\Omega \setminus Z\) is a disjoint union of superlevel sets of  Green's functions of \(-\Delta + V\).

The topological properties of the Sobolev-components of \(\Omega \setminus Z\) are investigated in \Cref{sectionQuasiTopology,sectionSobolevConnected}.
We show for example that they are Sobolev-connected using a variant of Poincaré's balayage method on Sobolev-open sets.
We then prove \Cref{theoremMaximumPrincipleFull} and \Cref{theoremMaximumPrinciplePrecise} in \Cref{sectionProofs} using the decomposition of \(\Omega \setminus Z\) and Green's representation formula.
In \Cref{sectionCounterpart}, we present a weaker version of this formula for solutions of \eqref{eqDirichletProblem}, which entitles us to adapt the proof of \Cref{theoremMaximumPrincipleFull} and get its counterpart for general nonnegative measures.

%%%%%%%%%%%%%%%%%%%%%%%%%%%%%%%%%%%%%%%%%%%%%%%%%%%%%%%%%%%%%%%%%%%%%%
%%%%%%%%%%%%%%%%%%%%%%%%%%%%%%%%%%%%%%%%%%%%%%%%%%%%%%%%%%%%%%%%%%%%%%
%%%%%%%%%%%%%%%%%%%%%%%%%%%%%%%%%%%%%%%%%%%%%%%%%%%%%%%%%%%%%%%%%%%%%%

\section{Method of sub- and supersolutions}
\label{sectionMethodSS}

We denote by \(\cM(\Omega)\) the vector space of finite Borel measures in \(\Omega\), which we equip with the total variation norm
\[{}
\norm{\mu}_{\cM(\Omega)}
\vcentcolon= |\mu|(\Omega) 
= \int_{\Omega}{\dif\abs{\mu}}.
\]
We recall that a function \(u \in L^{1}(\Omega)\) satisfies the equation
\[
-\Delta u + Vu = \mu \quad \text{in the sense of distributions in \(\Omega\)}
\]
for some \(\mu \in \cM(\Omega)\) whenever one has \(u \in L^{1}(\Omega; V \dif x)\) and
\[
\int_{\Omega}{u\, (-\Delta\varphi + V\varphi)} = \int_{\Omega}{\varphi\dif\mu} \quad \text{for every \(\varphi \in C^{\infty}_{c}(\Omega)\).}
\]

We prove in this section the following form of the method of sub- and supersolutions for distributional solutions of the Dirichlet problem \eqref{eqDirichletProblem} involving the Schrödinger operator:

\begin{proposition}
	\label{propositionDistributionalExistence}
	If \eqref{eqDirichletProblem} has a distributional solution with a nonnegative measure  \(\mu \in \cM(\Omega)\), then \eqref{eqDirichletProblem} also has a distributional solution for every datum \(\nu \in \cM(\Omega)\) such that \(\abs{\nu} \le \mu\).
\end{proposition}

Although this statement is already proved in \cite{Ponce:2016}, we present a different argument based on the truncation of the potential.
This will be the occasion for us to recall several properties of solutions involving measures that are used throughout the paper.

In view of the linearity of the equation, a natural approach would be to rely on a duality argument based on the estimate
\begin{equation}
\label{eqDualityEstimate}
\biggabs{\int_{\Omega}{\varphi \dif\mu}}
\le C \norm{-\Delta\varphi + V\varphi}_{L^{p}(\Omega)}
\quad \text{for every \(\varphi \in C_{c}^{\infty}(\Omega)\),}
\end{equation}
where \(p > N/2\), that follows from the Sobolev imbedding of solutions of the Schrödinger equation with measure data; see \eqref{eqEstimateSobolev} below.
However, since \(V\) is merely a Borel function, such an estimate is useless as the right-hand side may be infinite for various choices of \(\varphi \in C_{c}^{\infty}(\Omega)\).

We begin instead by proving that a solution of \eqref{eqDirichletProblem} can be obtained as the limit of solutions of \eqref{eqDirichletProblem} involving the operator \(- \Delta + T_{k}(V)\), with the bounded potential \(T_{k}(V)\).
Here, \(T_{k} : \R \to \R\) denotes the truncation at levels \(\pm k\)\,:{}
\[{}
T_{k}(s)
\vcentcolon=
\begin{cases}
	- k	& \text{if \(s < - k\),}\\
	s	& \text{if \(-k \le s \le k\),}\\
	k	& \text{if \(s > k\).}
\end{cases}
\]

We recall that, for nonnegative bounded potentials, \eqref{eqDirichletProblem} has a solution \(u\) for every \(\mu \in \cM(\Omega)\)\,; see \cite{Stampacchia:1965}. 
Independently of the fact that \(V\) is bounded or not, 
we have the absorption estimate
\begin{equation}
	\label{eqEstimateAbsorption}
	\norm{Vu}_{L^{1}(\Omega)}
	\le \norm{\mu}_{\cM(\Omega)},
\end{equation}
which can be obtained using as test function a suitable approximation of \(\sgn{u}\)\,; see \cite{Brezis_Marcus_Ponce:2007}*{Proposition~4.B.3} or \cite{Ponce:2016}*{Proposition~21.5}.
Moreover, a solution of \eqref{eqDirichletProblem} belongs to \(W_{0}^{1, q}(\Omega)\) for every \(1 \le q < \frac{N}{N-1}\) and satisfies
\begin{equation}
	\label{eqEstimateSobolev}
	\norm{u}_{W^{1, q}(\Omega)}
	\le C \norm{\mu}_{\cM(\Omega)},
\end{equation}
for some constant \(C > 0\) depending on \(q\) and \(\Omega\), but not on the potential \(V\).{}

To see why \(C\) in \eqref{eqEstimateSobolev} can be chosen independently of \(V\), one observes that, by the equation satisfied by \(u\) and by the absorption estimate \eqref{eqEstimateAbsorption},
\[{}
\norm{\Delta u}_{\cM(\Omega)}
\le \norm{Vu}_{L^{1}(\Omega)} + \norm{\mu}_{\cM(\Omega)}
\le 2 \norm{\mu}_{\cM(\Omega)}.
\] 
Since \eqref{eqEstimateSobolev} is true for \(-\Delta\), see \cite{Littman_Stampacchia_Weinberger:1963}*{Lemma~7.3} or \cite{Ponce:2016}*{Proposition~5.1}, we get
\[{}
\norm{u}_{W^{1, q}(\Omega)}
	\le C' \norm{\Delta u}_{\cM(\Omega)}
	\le 2 C' \norm{\mu}_{\cM(\Omega)}.
\]
This argument thus shows that the estimate \eqref{eqEstimateSobolev} for \(-\Delta\) implies its counterpart for every \(-\Delta + V\) with a nonnegative \(V\).

The approximation scheme that is used in the proof of \Cref{propositionDistributionalExistence} is given by the following

\begin{lemma}
	\label{lemmaDualitySolutionsDistributionalBis}
	Let \(\mu \in \cM(\Omega)\) be a nonnegative measure and, for every \(k \in \N\), let \(u_{k} \in W_{0}^{1, 1}(\Omega)\) be such that
	\[{}
	- \Delta u_{k} + T_{k}(V) u_{k} = \mu{}
	\quad \text{in the sense of distributions in \(\Omega\).}
	\]
	If \eqref{eqDirichletProblem} has a distributional solution \(u\), then \((u_{k})_{k \in \N}\) converges to \(u\) and \((T_{k}(V)u_{k})_{k \in \N}\) converges to \(Vu\), both in \(L^{1}(\Omega)\).
\end{lemma}

\begin{proof}[Proof of \Cref{lemmaDualitySolutionsDistributionalBis}]
	We first observe that \(u\) also satisfies the Dirichlet problem with the operator \(-\Delta + T_{k}(V)\) and datum \(\mu - (V - T_{k}(V))u\).{}
	Thus, subtracting the equations satisfied by \(u_{k}\) and \(u\) we find
	\[{}
	- \Delta (u_{k} - u) + T_{k}(V) (u_{k} - u) = (V - T_{k}(V)) u
	\]
	in the sense of distributions in \(\Omega\). 
	Using the absorption estimate for the operator \(-\Delta + T_{k}(V)\), we get
	\begin{equation}
		\label{eqContractionTruncation}
	\norm{T_{k}(V)(u_{k} - u)}_{L^{1}(\Omega)}
	\le \norm{(V - T_{k}(V))u}_{L^{1}(\Omega)}.
	\end{equation}
	Since the constant in \eqref{eqEstimateSobolev} does not depend on the potential, we also have
	\begin{equation}
		\label{eqL1Truncation}
	\norm{u_{k} - u}_{L^{1}(\Omega)}
	\le C \norm{(V - T_{k}(V))u}_{L^{1}(\Omega)}.
	\end{equation}
	Observing that \(Vu \in L^{1}(\Omega)\),
	\[{}
	\lim_{k \to \infty}{\norm{(V - T_{k}(V))u}_{L^{1}(\Omega)}}
	= 0.
	\]
	Hence, the conclusion follows from \eqref{eqContractionTruncation} and \eqref{eqL1Truncation}.
\end{proof}

To prove a weak maximum principle for \eqref{eqDirichletProblem}, it is convenient to reformulate the definition of distributional solution as:
\(u \in L^{1}(\Omega)\) is such that \(u \in L^{1}(\Omega; V \dif x)\) and
\[{}
- \Delta u + Vu = \mu{}
\quad \text{in the sense of \((C_{0}^{\infty}(\overline{\Omega}))'\)\,,}
\]
that is,
\begin{equation}
\label{eqIntegralC0Infty}
\int_{\Omega}{u \, (-\Delta\psi + V\psi)}
= \int_{\Omega}{\psi \dif\mu}
\quad \text{for every \(\psi \in C_{0}^{\infty}(\overline\Omega)\).}
\end{equation}
The fact that we can use a larger class of smooth test functions comes from the assumption that \(u \in W_{0}^{1, 1}(\Omega)\), which encodes the zero boundary value of \(u\)\,; see \cite{Ponce:2016}*{Proposition~6.3}.

\begin{lemma}
	\label{lemmaWeakMaximumPrinciple}
	Let \(u\) be the distributional solution of \eqref{eqDirichletProblem} with \(\mu \in \cM(\Omega)\).{}
	If \(\mu \le 0\) in \(\Omega\), then \(u \le 0\) almost everywhere in \(\Omega\).
\end{lemma}

\begin{proof}[Proof of \Cref{lemmaWeakMaximumPrinciple}]
	By assumption on \(\mu\),
	\[{}
	- \Delta u + Vu \le 0{}
	\quad \text{in the sense of \((C_{0}^{\infty}(\overline{\Omega}))'\).}
	\]
	Applying the formulation of Kato's inequality up to the boundary from \cite{Brezis_Marcus_Ponce:2007}*{Proposition~4.B.5}, see also \cite{Ponce:2016}*{Lemma~20.8}, we have
	\[{}
	- \Delta u^{+} + \chi_{\{u > 0\}} Vu \le 0{}
	\quad \text{in the sense of \((C_{0}^{\infty}(\overline{\Omega}))'\).}
	\]
	Thus, for every \emph{nonnegative} \(\psi \in C_{0}^{\infty}(\overline\Omega)\),{}
	\[{}
	- \int_{\Omega}{u^{+} \Delta\psi} 
	\le - \int_{\{u > 0\}}{V u \psi}
	\le 0.
	\]
	Taking any such a \(\psi\) with \(-\Delta\psi > 0\) in \(\Omega\), we deduce that \(u^{+} = 0\) almost everywhere in \(\Omega\).
\end{proof}

\begin{proof}[Proof of \Cref{propositionDistributionalExistence}]
	For every \(k \in \N\), let \(u_{k}\) be as in \Cref{lemmaDualitySolutionsDistributionalBis} and let \(v_{k}\) be also a solution of \eqref{eqDirichletProblem} for \(-\Delta + T_{k}(V)\), but with datum \(\nu\).{}
	Observe that \(v_{k}\) exists in this case since \(T_{k}(V)\) is bounded.
	Thus,
	\[{}
	\int_{\Omega}{v_{k} \, (- \Delta \varphi + T_{k}(V) \varphi)}
	= \int_{\Omega}{\varphi \dif\nu} \quad \text{for every \(\varphi \in C_{c}^{\infty}(\Omega)\).}
	\]
	By \eqref{eqEstimateSobolev}, the sequence \((v_{k})_{k \in \N}\) is bounded in \(W_{0}^{1, q}(\Omega)\) for \(1 \le q < \frac{N}{N-1}\).{}
	Hence, there exists a subsequence \((v_{k_{j}})_{j \in \N}\) which converges in \(L^{1}(\Omega)\) and almost everywhere to some function \(v \in W_{0}^{1, 1}(\Omega)\).
	Since \(\abs{\nu} \le \mu\) and \(T_{k}(V)\) is nonnegative, by linearity of the equation and the weak maximum principle above we have 
	\[{}
	\abs{v_{k}} \le u_{k}
	\quad \text{almost everywhere in \(\Omega\).}
	\]
	By \Cref{lemmaDualitySolutionsDistributionalBis}, the sequence \((T_{k}(V)u_{k})_{k \in \N}\) converges to \(Vu\) in \(L^{1}(\Omega)\).{}
	Thus, by the Dominated convergence theorem the sequence \((T_{k_{j}}(V)v_{k_{j}})_{j \in \N}\) converges to \(Vv\) in \(L^{1}(\Omega)\).{}
	Therefore, letting \(k = k_{j} \to \infty\) in the integral identity above, we deduce that \(v\) satisfies the equation involving \(-\Delta + V\) with datum \(\nu\).
\end{proof}

%%%%%%%%%%%%%%%%%%%%%%%%%%%%%%%%%%%%%%%%%%%%%%%%%%%%%%%%%%%%%%%%%%%%%%
%%%%%%%%%%%%%%%%%%%%%%%%%%%%%%%%%%%%%%%%%%%%%%%%%%%%%%%%%%%%%%%%%%%%%%
%%%%%%%%%%%%%%%%%%%%%%%%%%%%%%%%%%%%%%%%%%%%%%%%%%%%%%%%%%%%%%%%%%%%%%

\section{Distributional solutions are duality solutions}
\label{sectionVariational}

Given \(f \in L^{2}(\Omega)\), we denote by \(\zeta_{f}\) the unique minimizer of the energy functional
	\begin{equation}
	\label{eqEnergy}
	E(z) = \frac{1}{2} \int_{\Omega}{(\abs{\nabla z}^{2} + Vz^{2})} - \int_{\Omega}{f z}
	\quad \text{in \(W_{0}^{1, 2}(\Omega) \cap L^{2}(\Omega; V \dif x)\).}
	\end{equation}
	Thus, \(\zeta_{f}\) is the (variational) solution of the Euler-Lagrange equation
	\begin{equation}
		\label{eqEulerLagrange}
	\int_{\Omega}{\bigl(\nabla \zeta_{f} \cdot \nabla z + V\zeta_{f} \, z \bigr)} 
	= \int_{\Omega}{f z}
	\quad \text{for every \(z \in W_{0}^{1, 2}(\Omega) \cap L^{2}(\Omega; V \dif x)\).} 
	\end{equation}
	Under the additional assumption that \(f \in L^{\infty}(\Omega)\), we have \(\zeta_{f} \in L^{\infty}(\Omega)\) and
	\[{}
	\norm{\zeta_{f}}_{L^{\infty}(\Omega)}
	\le \norm{f}_{L^{\infty}(\Omega)} \norm{\theta}_{L^{\infty}(\Omega)},
	\]
	where \(\theta\) is the classical solution of
	\begin{equation}
	\label{eqDirichletTheta}
	\left\{
	\begin{alignedat}{2}
		- \Delta\theta & = 1 && \quad \text{in \(\Omega\),}\\
		\theta & = 0 &&\quad  \text{on \(\partial\Omega\).}
	\end{alignedat}
	\right.
	\end{equation}
	
	We recall that \(x \in \Omega\) is a Lebesgue point of a function \(v \in L^{1}(\Omega)\) whenever there exists \(c \in \R\) such that
	\[{}
	\lim_{r \to 0}{\fint_{B_{r}(x)}{\abs{v - c}}}
	 = 0.
	\]
	The precise representative of \(v\) at \(x\) is then defined as \(\quasi{v}(x) \vcentcolon= c\).
	For \(f \in L^{\infty}(\Omega)\), the precise representative \(\quasi{\zeta_{f}}\) is well-defined everywhere in \(\Omega\).
	To see why this is true, by linearity of the equation it suffices to consider the case where \(f\) is nonnegative.
	One then shows that \(\zeta_{f} \in L^{1}(\Omega; V \dif x)\) and
	\begin{equation}
		\label{eqInequalityZeta}
	-\Delta\zeta_{f} + V\zeta_{f} \le f
	\quad \text{in the sense of distributions in \(\Omega\),} 
	\end{equation}
	which implies that \(\zeta_{f}\) is almost everywhere the difference between a continuous and a bounded superharmonic function, and then every \(x \in \Omega\) is a Lebesgue point of \(\zeta_{f}\) as claimed; see \cite{Ponce:2016}*{Lemma~8.10}.
	Inequality in \eqref{eqInequalityZeta} comes from an application of Fatou's lemma; 
	see \cite{Malusa_Orsina:1996} or \cite{Orsina_Ponce:2017}*{Proposition~8.1}.
	When \(V\) is bounded, one can apply the Dominated convergence theorem instead to get equality;	see \cite{Orsina_Ponce:2016}*{Proposition~3.1}.

	These functions \(\quasi{\zeta_{f}}\) can be used as test functions for distributional solutions of \eqref{eqDirichletProblem} and are more adapted to a duality argument in the spirit of estimate~\eqref{eqDualityEstimate}.

\begin{proposition}
	\label{propositionDistributionImpliesDuality}
	If \(u\) is a distributional solution of \eqref{eqDirichletProblem} for some \(\mu \in \cM(\Omega)\), then
	\[{}
	\int_{\Omega}{uf}
	= \int_{\Omega}{\quasi{\zeta_{f}} \dif\mu}
	\quad \text{for every \(f \in L^{\infty}(\Omega)\).}
	\]	
\end{proposition}

This result is proved in \cite{Malusa_Orsina:1996} using an approximation of \(\mu\) of the form \(\rho_{k} * \mu\) where \((\rho_{k})_{k \in \N}\) is a sequence of mollifiers in \(C_{c}^{\infty}(\R^{N})\)\,; see also \Cref{remarkDualityApproximation} below.{}
For the convenience of the reader, we present an alternative approximation based on the truncation of the potential \(V\), without changing the measure \(\mu\).

\begin{lemma}
	\label{lemmaVariationalTruncationPointwiseConvergence}
	Given \(f \in L^{\infty}(\Omega)\) and \(k \in \N\), let \(\zeta_{f, k}\) be the minimizer of
	\[{}
	E_{k}(z) 
	= \frac{1}{2} \int_{\Omega}{(\abs{\nabla z}^{2} + T_{k}(V) z^{2})} - \int_{\Omega}{f z}	
	\quad \text{in \(W_{0}^{1, 2}(\Omega) \cap L^{2}(\Omega; V \dif x)\).}
	\]
	Then, 
	\[{}
	\lim_{k \to \infty}{\quasi{\zeta_{f, k}}(x)}
	= \quasi{\zeta_{f}}(x)
	\quad \text{for every \(x \in \Omega\).}
	\]
\end{lemma}

As the function \(\zeta_{f, k}\) satisfies
\[{}
-\Delta\zeta_{f, k} + T_{k}(V)\zeta_{f, k}
= f
\quad \text{in the sense of distributions in \(\Omega\),}
\]
we have that \(\Delta\zeta_{f, k} \in L^{\infty}(\Omega)\) and then \(\quasi{\zeta_{f, k}}\) is continuous (and even \(C^{1}\)) in \(\Omega\).
The proof of the lemma relies on the property that for a uniformly bounded and nondecreasing sequence \((v_{k})_{k \in \N}\) of nonnegative superharmonic functions  converging almost everywhere to \(v\), the sequence of precise representatives \((\quasi{v_{k}})_{k \in \N}\) converges \emph{everywhere} to \(\quasi{v}\)\,; see \cite[Lemma~4.12]{Malusa_Orsina:1996} or \cite{Ponce:2016}*{Exercise~8.4}.

\begin{proof}[Proof of \Cref{lemmaVariationalTruncationPointwiseConvergence}]
	We first prove that the sequence \((\zeta_{f, k})_{k \in \N}\) converges weakly in \(W_{0}^{1, 2}(\Omega)\) to \(\zeta_{f}\).{}
	We begin by observing that
	\begin{equation}
	\label{eqVariationalEstimates}
		E_{k}(\zeta_{f, k})
	\le E_{k}(\zeta_{f})
	\le E(\zeta_{f})
	\quad \text{for every \(k \in \N\).}
	\end{equation}
	This implies that \(({\zeta_{f, k}})_{k \in \N}\) is bounded in \(W_{0}^{1, 2}(\Omega)\).{}
	Thus, there exists a subsequence \(({\zeta_{f, k_{j}}})_{j \in \N}\) which converges weakly in \(W_{0}^{1, 2}(\Omega)\) and almost everywhere in \(\Omega\) to some function \(z\).{}
	In particular, by Fatou's lemma,
	\[{}
	\int_{\Omega}{Vz^{2}}
	\le \liminf_{j \to \infty}{\int_{\Omega}{T_{k_{j}}(V) \zeta_{f, k_{j}}^{2}}}. 
	\]
	Taking \(k = k_{j}\) in \eqref{eqVariationalEstimates} and letting \(j \to \infty\), we get
	\[{}
	E(z) \le E(\zeta_{f}).
	\]
	Since \(\zeta_{f}\) is the unique minimizer of the functional \(E\), we deduce that \(z = \zeta_{f}\) almost everywhere in \(\Omega\).{}
	By uniqueness of the limit, the entire sequence \(({\zeta_{f, k}})_{k \in \N}\) converges weakly to \(\zeta_{f}\).

	By linearity of the Euler-Lagrange equation, we may proceed with the proof of the lemma assuming that \(f\) is nonnegative.
	In this case, by the weak maximum principle the sequence \(({\zeta_{f, k}})_{k \in \N}\) is non-increasing in \(\Omega\) and then, by the first part of the proof, converges almost everywhere to \(\zeta_{f}\).{}
	Let \(v_{k}\) and \(w\) be such that 
	\[{}
	\left\{
	\begin{alignedat}{2}
		- \Delta v_{k} & = T_{k}(V)\zeta_{f, k}	&& \quad \text{in \(\Omega\),}\\
		v_{k} & = 0	&&	\quad \text{on \(\partial\Omega\),}
	\end{alignedat}
	\right.
	\quad{}
	\text{and}
	\quad{}
	\left\{
	\begin{alignedat}{2}
		- \Delta w & = f	&& \quad \text{in \(\Omega\),}\\
		w & = 0	&&	\quad \text{on \(\partial\Omega\).}
	\end{alignedat}
	\right.		
	\]
	We then have \(v_{k} = w - \zeta_{f, k}\) almost everywhere in \(\Omega\), which implies that
	\[{}
	\quasi{v_{k}}(x)
	= \quasi{w}(x) - \quasi{\zeta_{f, k}}(x)
	\quad \text{for every \(x \in \Omega\).}
	\]
	Observe that \((\quasi{v_{k}})_{k \in \N}\) is a uniformly bounded and nondecreasing sequence of nonnegative superharmonic functions.
	Thus, its pointwise limit \(v\) coincides with the precise representative \(\quasi{v}\) in \(\Omega\).{}
	We then get
	\begin{equation}
	\label{eqLimitPreciseRepresentative1}
	\quasi{v}(x)
	= v(x)
	= \quasi{w}(x) - \lim_{k \to \infty}{\quasi{\zeta_{f, k}}(x)}
	\quad \text{for every \(x \in \Omega\).}
	\end{equation}
	Since \(v = w - \zeta_{f}\) almost everywhere in \(\Omega\), we also have
	\begin{equation}
	\label{eqLimitPreciseRepresentative2}
	\quasi{v}(x)
	= \quasi{w}(x) - \quasi{\zeta_{f}}(x)
	\quad \text{for every \(x \in \Omega\).}
	\end{equation}	
	The conclusion follows from comparison between \eqref{eqLimitPreciseRepresentative1} and \eqref{eqLimitPreciseRepresentative2} and the boundedness of \(w\).
\end{proof}

\begin{proof}[Proof of \Cref{propositionDistributionImpliesDuality}]
	Let us first assume that \(V\) is bounded.
	In this case,  for every \(f \in L^{\infty}(\Omega)\), \(\quasi{\zeta_{f}}\) is continuous, \(\Delta\zeta_{f}\) is bounded and satisfies
	\[{}
	- \Delta\zeta_{f} = f - V \zeta_{f}
	\quad \text{in the sense of distributions in \(\Omega\).}
	\]
	One can thus approximate \(\quasi{\zeta_{f}}\) uniformly by a sequence \((\quasi{\zeta_{f_{k}}})_{k \in \N}\) in \(C_{0}^{\infty}(\overline{\Omega})\) such that \((f_{k})_{k \in \N}\) is bounded in \(L^{\infty}(\Omega)\) and converges almost everywhere to \(f\).{}
	To construct such an example, one can take \(g_{k} = \rho_{k} * (f - V\zeta_{f})\), where \((\rho_{k})_{k \in \N}\) is a sequence of mollifiers, and \(v_{k} \in C_{0}^{\infty}(\overline\Omega)\) as the classical solution of
	\[{}
	\left\{
	\begin{alignedat}{2}
		- \Delta v_{k} & = g_{k}	&& \quad \text{in \(\Omega\),}\\
		v_{k} & = 0	&& \quad \text{on \(\partial\Omega\).}
	\end{alignedat}
	\right.
	\]
	We then have the desired approximation of \(\quasi{\zeta_{f}}\) by observing that \(v_{k} = \quasi{\zeta_{f_{k}}}\) with \(f_{k} = g_{k} + Vv_{k}\).
	For every \(k \in \N\), we get from \eqref{eqIntegralC0Infty} that
	\[{}
	\int_{\Omega}{u f_{k}}
	= \int_{\Omega}{u \, (-\Delta v_{k} + V v_{k})}
	= \int_{\Omega}{v_{k} \dif\mu}
	= \int_{\Omega}{\quasi{\zeta_{f_{k}}} \dif\mu}
	\]	
	and the conclusion for \(V\) bounded follows as \(k \to \infty\).
	
	We now assume that \(V\) is merely a Borel function and denote by \(u_{k}\) the distributional solution of the Dirichlet problem associated to  \(- \Delta + T_{k}(V)\) and datum \(\mu\). 
	By \Cref{lemmaDualitySolutionsDistributionalBis}, \((u_{k})_{k \in \N}\) converges to \(u\) in \(L^{1}(\Omega)\).{}
	On the other hand, from the first part of this proof and using the notation of \Cref{lemmaVariationalTruncationPointwiseConvergence},
	\[{}
	\int_{\Omega}{u_{k} f}
	= \int_{\Omega}{\quasi{\zeta_{f, k}} \dif\mu}
	\quad \text{for every \(f \in L^{\infty}(\Omega)\).}
	\]
	By uniform boundedness and pointwise convergence of \((\quasi{\zeta_{f, k}})_{k \in \N}\), the proposition follows.
\end{proof}

The concept of duality solution of the Dirichlet problem \eqref{eqDirichletProblem} is a useful tool in establishing the connection between the failure of the strong maximum principle and the nonexistence of distributional solutions of \eqref{eqDirichletProblem}.
We recall its definition from \cite{Malusa_Orsina:1996}:

\begin{definition}
	\label{definitionDuality}
	Given \(\mu \in \cM(\Omega)\), we say that \(u \in L^{1}(\Omega)\) is a \emph{duality solution} of \eqref{eqDirichletProblem} whenever 
	\[{}
	\int_{\Omega}{u f}
	= \int_{\Omega}{\quasi{\zeta_{f}} \dif\mu}
	\quad \text{for every \(f \in L^{\infty}(\Omega)\).}
	\]
\end{definition}

Uniqueness of the duality solution is a straightforward consequence of the fact that \(u \equiv 0\) is the only solution with \(\mu = 0\).{}
More generally, the weak maximum principle also holds in the duality setting.
While \Cref{propositionDistributionImpliesDuality} states that distributional solutions (whenever they exist) are duality solutions, the latter exist for any given finite measure:

\begin{proposition}
	\label{propositionDualityExistence}
	The Dirichlet problem \eqref{eqDirichletProblem} has a unique duality solution for every \(\mu \in \cM(\Omega)\). 
\end{proposition} 

This proposition is proved in \cite{Malusa_Orsina:1996}*{Theorem~5.6}, in the spirit of \cite{Littman_Stampacchia_Weinberger:1963}. 
The proof is based on Stampacchia's estimate:
\[{}
\norm{\zeta_{f}}_{L^{\infty}(\Omega)}
\le C \norm{f}_{(W_{0}^{1, q}(\Omega))'}
\quad \text{for every \(f \in L^{\infty}(\Omega)\),}
\]
where \(q < \frac{N}{N-1}\) and \(C > 0\) depends on \(q\) and \(\Omega\),
which implies that any duality solution belongs to \(W_{0}^{1, q}(\Omega)\) and satisfies
\begin{equation}
\label{eqFederer}
\norm{u}_{W^{1, q}(\Omega)}
\le C \norm{\mu}_{\cM(\Omega)}.
\end{equation}

\begin{remark}
	\label{remarkDistributionalSolutionsVariational}
From the Euler-Lagrange equation \eqref{eqEulerLagrange}, minimizers of \(E\) are also duality solutions.
Indeed, for any \(h \in L^{2}(\Omega)\), one can apply \eqref{eqEulerLagrange} with \(z = \zeta_{h}\) to get
\[{}
\int_{\Omega}{\zeta_{h} f}
= \int_{\Omega}{\zeta_{f} h}
\quad \text{for every \(f \in L^{\infty}(\Omega)\).}
\]
Since \(\zeta_{f} = \quasi{\zeta_{f}}\) almost everywhere in \(\Omega\), \(\zeta_{h}\) is then a duality solution of \eqref{eqDirichletProblem} with \(\mu = h \dif x\).
Assuming that \eqref{eqDirichletProblem} has a distributional solution \(w\) with \(\mu = h \dif x\), then \(w\) is also a duality solution by \Cref{propositionDistributionImpliesDuality}.
Hence, by uniqueness, one has \(w = \zeta_{h}\).
\end{remark}

Denoting by \(G_{x}\) the duality solution associated to the Dirac mass \(\mu = \delta_{x}\) at any point \(x \in \Omega\), we have the following representation formula:
\begin{equation}
	\label{eqRepresentationGreen}
\quasi{\zeta_{f}}(x)
= \int_{\Omega}{G_{x} f}
\quad \text{for every \(f \in L^{\infty}(\Omega)\).}
\end{equation}
By \Cref{remarkDistributionalSolutionsVariational}, this formula also applies to distributional solutions with bounded data.
It thus seems that we have already fulfilled part of our goals we set in the introduction, more specifically in \Cref{theoremGreen}.
However, we do not know whether \(G_{x}\) is a distributional solution of \eqref{eqDirichletProblem} with \(\mu = \delta_{x}\)\,, and we still have to prove that this is the case if and only if \(x \in \Omega \setminus Z\).{}
In addition, we would like to identify all nonnegative functions \(f \in L^{\infty}(\Omega)\) such that \(\zeta_{f}\) is indeed a distributional solution with datum \(f\).{}
We eventually prove that this is true if and only if \(\int_{Z}{f} = 0\).{}

	To conclude this section we explain why duality solutions enjoy good approximation properties in the sense that reasonable approximation schemes of the measure \(\mu\) yield sequences of solutions that converge to the duality solution \(u\) associated to \(\mu\).

\begin{proposition}
	\label{remarkDualityApproximation}
	Let \(N \ge 2\) and let \((\rho_{k})_{k \in \N}\) be a sequence of mollifiers of the form \(\rho_{k}(x) = \frac{1}{r_{k}^{N}} \varphi(\frac{x}{r_{k}})\) for a fixed \(\varphi \in C_{c}^{\infty}(\R^{N})\) and a sequence \((r_{k})_{k \in \N}\) of positive numbers converging to zero.{}
	Given \(\mu \in \cM(\Omega)\), the sequence \((\zeta_{\rho_{k} * \mu})_{k \in \N}\) converges  in \(L^{p}(\Omega)\) to the duality solution of \eqref{eqDirichletProblem} for every \(1 \le p < \frac{N}{N-2}\).
\end{proposition}

\begin{proof}
	The assumption on \((\rho_{k})_{k \in \N}\) implies that, for every \(f \in L^{\infty}(\Omega)\) and \(x \in \Omega\),
	\[{}
	\lim_{k \to \infty}{\widecheck{\rho_{k}} * \zeta_{f}(x)}
	= \widehat{\zeta_{f}}(x),
	\]
	where \(\widecheck{\rho_{k}}(y) = \rho_{k}(-y)\).{}
	Then, by Fubini's theorem and the Dominated convergence theorem,
	\[{}
	\int_{\Omega}{\zeta_{f} \, \rho_{k} * \mu}
	= \int_{\Omega}{\widecheck{\rho_{k}} * \zeta_{f}  \dif\mu}
	\to \int_{\Omega}{\quasi{\zeta_{f}} \dif\mu}.
	\]
	Since \(\zeta_{\rho_{k} * \mu}\) is the duality solution with datum \(\rho_{k} * \mu\),{}
	the latter convergence can be rewritten as
	\[{}
	\lim_{k \to \infty}{\int_{\Omega}{\zeta_{\rho_{k} * \mu} \, f}}
	= \int_{\Omega}{u f},
	\] 
	where \(u\) is the duality solution with datum \(\mu\)
	As a result, \((\zeta_{\rho_{k} * \mu})_{k \in \N}\) converges to \(u\) with respect to the \(L^{\infty}(\Omega)\)-weak\(^{*}\) topology.
	By boundedness of \((\zeta_{\rho_{k} * \mu})_{k \in \N}\) in \(W_{0}^{1, q}(\Omega)\) for every \(1 \le q < \frac{N}{N-1}\)\,, see \eqref{eqFederer}, we have the strong convergence to \(u\) in \(L^{p}\) spaces.
\end{proof}

%%%%%%%%%%%%%%%%%%%%%%%%%%%%%%%%%%%%%%%%%%%%%%%%%%%%%%%%%%%%%%%%%%%%%%
%%%%%%%%%%%%%%%%%%%%%%%%%%%%%%%%%%%%%%%%%%%%%%%%%%%%%%%%%%%%%%%%%%%%%%
%%%%%%%%%%%%%%%%%%%%%%%%%%%%%%%%%%%%%%%%%%%%%%%%%%%%%%%%%%%%%%%%%%%%%%

\section{Duality solutions as distributional solutions}
\label{sectionDualityAsDistributions}

We henceforth denote by \(S\) the subset of \(\Omega\) defined as the zero-set of the \emph{torsion function} \(\zeta_{1}\)\,, namely
\begin{equation}
	\label{eqSetS1}
	S = 
\bigl\{ x \in \Omega : \quasi{\zeta_{1}}(x) = 0 \bigr\}.
\end{equation}
We recall that \(\zeta_{1}\) is the minimizer of the energy functional \(E\) with constant \(f \equiv 1\), and so \(S\) is a Sobolev-closed set and depends on the potential \(V\).
By the weak maximum principle for variational solutions, for every  \(f \in L^{\infty}(\Omega)\) the torsion function dominates \(\zeta_{f}\) in the sense that
\begin{equation}
\label{eqComparisonTorsion}
|\zeta_{f}|
\le \norm{f}_{L^{\infty}(\Omega)} \zeta_{1}
\quad \text{almost everywhere in \(\Omega\).}
\end{equation}
The same estimate is then satisfied by the precise representatives, this time at every point in \(\Omega\), and we deduce that
\begin{equation}
	\label{eqSetS2}
S = 
\bigl\{ x \in \Omega : \quasi{\zeta_{f}}(x) = 0 \ \text{for every \(f \in L^{\infty}(\Omega)\)} \bigr\}.
\end{equation}

By \Cref{remarkDistributionalSolutionsVariational}, this characterization of \(S\) involves more functions than in the definition of the universal zero-set \(Z\).{}
Therefore,
\[{}
S \subset Z.
\]
For example, when \(V\) is bounded, the notions of distributional and duality solution coincide and the strong maximum principle holds everywhere in \(\Omega\).
We thus have in this case 
\[{}
S = Z = \emptyset.{}
\]
For unbounded potentials \(V\), the inclusion can be strict: 

\begin{example}
	\label{exampleBallSingular1-2}
The Dirichlet problem
	\begin{equation*}
	\left\{
	\begin{alignedat}{2}
		- \Delta u + \frac{1}{\abs{x_{1}}^{\alpha}} \, u & = \mu	&& \quad \text{in \(B_{1}(0)\),}\\
		u & = 0	&&	\quad \text{on \(\partial B_{1}(0)\),}
	\end{alignedat}
	\right.{}
	\end{equation*}
has no distributional solution with \(\mu\) nonnegative, \(\mu \ne 0\), for any exponent \(1 \le \alpha < 2\)\,; see \cite{Orsina_Ponce:2008}*{Theorem~9.1}.{}
In this case, \(\zeta_{1}\) solves two independent Dirichlet problems, one on each side of the hyperplane \(\{x_{1} = 0\}\), and then
\[{}
S = \{x_{1} = 0\} \cap B_{1}(0)
\quad \text{and} \quad Z = \Omega.
\]
For \(\alpha \ge 2\), the singularity of \(V\) is even stronger and, nevertheless, one has the equality
\[{}
S = Z = \{x_{1} = 0\} \cap B_{1}(0),
\]
since \(\zeta_{1}\) now satisfies \eqref{eqEquationSupersolution} with \(f \equiv 1\).
\end{example}

We prove in this section that duality solutions can be seen as distributional solutions of the Dirichlet problem, but  when \(S \ne \emptyset\) they need not solve the equation in the sense of distributions  with the same datum \(\mu\).

\begin{proposition}
	\label{propositionSolutionDualityasDistribution}
	If \(u \in L^{1}(\Omega)\) is a duality solution of \eqref{eqDirichletProblem} for some nonnegative measure \(\mu \in \cM(\Omega)\), then \(u \in W_{0}^{1, 1}(\Omega) \cap L^{1}(\Omega; V \dif x)\) and
	\[{}
	-\Delta u + Vu = \mu\lfloor_{\Omega\setminus S}{} - \lambda 
	\quad \text{in the sense of distributions in \(\Omega\),}
	\]
	where \(\lambda \in \cM(\Omega)\) is nonnegative, diffuse with respect to the \(W^{1, 2}\) capacity and carried by \(S\), that is, 
	\[{}
	\lambda(\Omega \setminus S) = 0.
	\]
\end{proposition}

Here, \(\mu\lfloor_{A}\) denotes the contraction of \(\mu\) with respect to a Borel set \(A\), defined by 
\[{}
\mu\lfloor_{A}(B) = \mu(B \cap A).
\]
By a diffuse measure we mean that \(\lambda(B) = 0\) for every Borel subset \(B \subset \Omega\) having \(W^{1, 2}\)~capacity zero.
We finally recall that the \(W^{1, 2}\)~capacity of a compact subset \(K \subset \R^{N}\) is defined as
\[{}
\capt_{W^{1, 2}}{(K)}
= \inf{\Bigl\{\norm{\varphi}_{W^{1, 2}(\R^{N})} 
: \varphi \in C_{c}^{\infty}(\R^{N}),\ \varphi \ge 0\ \text{in \(\R^{N}\)}\ \text{and}\ \varphi > 1\ \text{on \(K\)}\Bigr\}}.
\]
It is then extended to open sets by inner regularity and then to arbitrary sets by outer regularity.

By \Cref{propositionSolutionDualityasDistribution}, a duality solution thus fails from being a distributional one for the same nonnegative datum \(\mu\) for two possible reasons: 
The existence of some nontrivial mass carried by \(\mu\) on \(S\) or the appearance of a nonpositive measure carried by \(S\).{}
This latter phenomenon always happens in the case of \Cref{exampleBallSingular1-2} when \(1 \le \alpha < 2\) since there are simply no distributional supersolutions, other than the trivial one.
One also shows that the measure \(\lambda\) is always singular with respect to the Lebesgue measure; see \Cref{remarkSingular}.

We begin with the following approximation procedure, where in contrast with \Cref{lemmaDualitySolutionsDistributionalBis} we do not assume that \eqref{eqDirichletProblem} has a distributional solution.

\begin{lemma}
	\label{lemmaDualitySolutionsDistributional}
	Let \(\mu \in \cM(\Omega)\) be a nonnegative measure and, for every \(k \in \N\), let \(u_{k} \in W_{0}^{1, 1}(\Omega)\) be such that
	\[{}
	- \Delta u_{k} + T_{k}(V) u_{k} = \mu{}
	\quad \text{in the sense of distributions in \(\Omega\).}
	\]
	Then, the sequence \((u_{k})_{k \in \N}\) converges in \(L^{1}(\Omega)\) to the duality solution \(u\) of \eqref{eqDirichletProblem}.
	Moreover, \(u \in W_{0}^{1, 1}(\Omega) \cap L^{1}(\Omega; V \dif x)\) and there exists a nonnegative measure \(\lambda \in \cM(\Omega)\) such that
	\[{}
	-\Delta u + Vu = \mu - \lambda 
	\quad \text{in the sense of distributions in \(\Omega\).}
	\]
\end{lemma}

\begin{proof}[Proof of \Cref{lemmaDualitySolutionsDistributional}]
	By the weak maximum principle, the sequence \((u_{k})_{k \in \N}\) is non-increasing and nonnegative, hence it converges in \(L^{1}(\Omega)\) to some function \(u\).{}
	Using the notation of \Cref{lemmaVariationalTruncationPointwiseConvergence}, we have
	\[{}
	\int_{\Omega}{u_{k} f}
	= \int_{\Omega}{\quasi{\zeta_{f, k}} \dif\mu}
	\quad \text{for every \(f \in L^{\infty}(\Omega)\).}
	\]
	The sequence \((\quasi{\zeta_{f, k}})_{k \in \N}\) is uniformly bounded and, by \Cref{lemmaVariationalTruncationPointwiseConvergence}, converges pointwise to \(\quasi{\zeta_{f}}\).{}
	By the Dominated convergence theorem, we thus have
	\[{}
	\int_{\Omega}{u f}
	= \int_{\Omega}{\quasi{\zeta_{f}} \dif\mu},
	\]
	so that \(u\) is the duality solution of \eqref{eqDirichletProblem} and then belongs to \(W^{1, q}_{0}(\Omega)\) for every \(q < \frac{N}{N-1}\).
	
	Since the sequence \((T_{k}(V) u_{k})_{k \in \N}\) is bounded in \(L^{1}(\Omega)\) and converges pointwise to \(Vu\), by Fatou's lemma we have \(Vu \in L^{1}(\Omega)\).{}
	For every  \(\varphi \in C_{c}^{\infty}(\Omega)\) we also have
	\[{}
	\biggabs{\int_{\Omega}{u_{k} \Delta\varphi} + \int_{\Omega}{\varphi \dif \mu}}
	= \biggabs{\int_{\Omega}{T_{k}(V)u_{k}\varphi}}
	\le C \norm{\varphi}_{L^{\infty}(\Omega)},
	\]
	for some constant independent of \(k\).
	Letting \(k \to \infty\), we deduce from the Riesz representation theorem that there exists \(\nu \in \cM(\Omega)\) such that
	\begin{equation}
		\label{eqDuality1}
	\int_{\Omega}{\varphi \dif\nu}
	= \int_{\Omega}{u \Delta\varphi} + \int_{\Omega}{\varphi \dif \mu}.
	\end{equation}
	By Fatou's lemma, for nonnegative test functions \(\varphi\) we also have
	\begin{equation}
		\label{eqDuality2}
	\begin{split}
	\int_{\Omega}{Vu \varphi}
	& \le \lim_{k \to \infty}{\int_{\Omega}{T_{k}(V)u_{k} \varphi}}\\
	& = \lim_{k \to \infty}{\int_{\Omega}{u_{k} \Delta\varphi} + \int_{\Omega}{\varphi \dif \mu}}
	= \int_{\Omega}{u \Delta\varphi} + \int_{\Omega}{\varphi \dif \mu}.
	\end{split}
	\end{equation}
	Combining \eqref{eqDuality1} and \eqref{eqDuality2}, we deduce that \(Vu \dif x \le \nu\) in the sense of distributions in \(\Omega\).{}
	By the regularity of finite Borel measures such an inequality also holds in the sense of measures, that is,
	\[{}
	\int_{A}{Vu \dif x}
	\le \nu(A){}
	\quad \text{for every Borel set \(A \subset \Omega\)\,;}
	\]
	see e.g.~\cite{Ponce:2016}*{Proposition~6.12}. 
	The conclusion is then satisfied by the finite measure \(\lambda = \nu - Vu \dif x\).
\end{proof}

\begin{proof}[Proof of \Cref{propositionSolutionDualityasDistribution}]
	Let \(u\) be the solution of \eqref{eqDirichletProblem} with datum \(\mu\).
	By the characterization \eqref{eqSetS2} of \(S\) we have \(\quasi{\zeta_{f}} = 0\) on \(S\) for every \(f \in L^{\infty}(\Omega)\),
	which implies that
	\begin{equation}
		\label{eqSolutionDualityasDistribution1}
			\int_{\Omega}{\quasi{\zeta_{f}} \dif\mu\lfloor_{\Omega \setminus S}}
	= \int_{\Omega}{\quasi{\zeta_{f}} \dif\mu}
	= \int_{\Omega}{u f}.
	\end{equation}
	Hence \(u\) is also a duality solution with datum \(\mu\lfloor_{\Omega \setminus S}\)\,.
	By \Cref{lemmaDualitySolutionsDistributional} applied to the measure \(\mu\lfloor_{\Omega \setminus S}\)\,, there exists a nonnegative measure \(\lambda\) such that \(u\) is a distributional solution of \eqref{eqDirichletProblem} with datum \(\mu\lfloor_{\Omega \setminus S}{} - \lambda\).{}
	Since by \Cref{propositionDistributionImpliesDuality} a distributional solution is a duality solution with the same datum, for every \(f \in L^{\infty}(\Omega)\) we thus have
	\begin{equation}
		\label{eqSolutionDualityasDistribution2}
	\int_{\Omega}{u f}
	= \int_{\Omega}{\quasi{\zeta_{f}} \dif(\mu\lfloor_{\Omega \setminus S}{} - \lambda)}.
	\end{equation}
	Then, by comparision between \eqref{eqSolutionDualityasDistribution1} and \eqref{eqSolutionDualityasDistribution2}, 
	\[{}
	\int_{\Omega}{\quasi{\zeta_{f}} \dif\lambda}
	= 0
	\quad \text{for every \(f \in L^{\infty}(\Omega)\).}
	\]
	Apply this identity with \(f \equiv 1\).
	Since \(\quasi{\zeta_{1}} > 0\) on \(\Omega \setminus S\), by nonnegativity of \(\lambda\) it follows that \(\lambda(\Omega \setminus S) = 0\).{}
	
	To prove that \(\lambda\) is diffuse, we first recall that \(\lambda\) can be uniquely decomposed as a sum of measures, \(\lambda = \lambda\ld + \lambda\lc\), where \(\lambda\ld\) is the diffuse part with respect to the \(W^{1, 2}\)~capacity and \(\lambda\lc\) is concentrated on a set of \(W^{1, 2}\)~capacity zero.
	This is analogous to the classical Lebesgue decomposition theorem with respect to a given measure.
	Although in our case it involves a capacity, the proofs are similar; see~\cite{Ponce:2016}*{Proposition~14.12}.
	
	We thus have to check that \(\lambda\lc = 0\).
	For this purpose, we rely on the inverse maximum principle which asserts that, by nonnegativity of \(u\), the concentrated part of \(\Delta u\) satisfies \((\Delta u)\lc \le 0\) in \(\Omega\)\,; see~\cite{Dupaigne_Ponce:2004}*{Theorem~3} or \cite{Ponce:2016}*{Proposition~6.13}.
	Next, the equation
	\[{}
	-\Delta u + Vu = \mu\lfloor_{\Omega \setminus S}{} - \lambda
	\]
	holds in the sense of distributions, whence also in the sense of measures in \(\Omega\)\,; see \cite{Ponce:2016}*{Proposition~6.12}.
	More precisely, for every Borel set \(A \subset \Omega\),{}
	\[{}
	\int_{A}(-\Delta u + Vu) = 
	\mu(A \setminus S) - \lambda(A).
	\]
	Restricting such an identity to subsets of \(S\) we get
	\(\Delta u = Vu + \lambda\) in \(S\).{}
	Then, as the Lebesgue measure is diffuse with respect to the \(W^{1, 2}\)~capacity, 
	\[{}
	(\Delta u)\lc = \lambda\lc 
	\quad \text{in \(S\).}
	\]
	It thus follows from the inverse maximum principle that \(\lambda\lc \le 0\) in \(S\).{}
	Since \(\lambda = 0\) in \(\Omega\setminus S\), we conclude that \(\lambda\lc \le 0\) in \(\Omega\). 
	By nonnegativity of \(\lambda\), we must have \(\lambda\lc = 0\), which means that \(\lambda\) is diffuse.
\end{proof}

\begin{remark}
	\label{remarkSingular}
	The nonnegative measure \(\lambda\) given by \Cref{propositionSolutionDualityasDistribution} is singular with respect to the Lebesgue measure.
	Indeed, on the one hand, we claim that
	\begin{equation}
		\label{eq-1300}
		u = 0
		\quad \text{almost everywhere in \(S\).}
	\end{equation}
	To this end, observe that \(G_{x} = 0\) for every \(x \in S\) by an application of the representation formula \eqref{eqRepresentationGreen} with \(f \equiv 1\).{}
	As the counterpart of the representation formula is satisfied by \(u\) almost everywhere in \(\Omega\), see  \Cref{lemmaRepresentationMeasure}, we thus have \eqref{eq-1300}.
	On the other hand, by the Lebesgue decomposition theorem, we can decompose the measure \(\Delta u\) as a sum \(\Delta u = (\Delta u)\la + (\Delta u)\ls\)\,, where \((\Delta u)\la\) is absolutely continuous with respect to the Lebesgue measure and \((\Delta u)\ls\) is singular.
	According to a result by Ambrosio, Ponce and Rodiac~\cite{Ambrosio_Ponce_Rodiac:2019}*{Theorem~1.1},
	\[{}
	(\Delta u)\la = 0	
	\quad \text{in \(\{u = c\}\)}
	\]
	for every \(c \in \R\) and in particular in the level set  \(\{u = 0\}\).
	In our case, \(S\) is contained in \(\{u = 0\}\), except for a set of Lebesgue measure zero, and the equation satisfied by \(u\) gives 
	\[{}
	(\Delta u)\la 
	= Vu \dif x - (\mu\la)\lfloor_{\Omega \setminus S}{} + \lambda\la{}
	\quad \text{in \(\Omega\).} 
	\]
	Restricting this identity to \(S\), we thus have
	\[{}
	\lambda\la 
	= (\Delta u)\la{}
	= 0
	\quad \text{in \(S\).} 
	\]
	As \(\lambda = 0\) in \(\Omega \setminus S\), we conclude that \(\lambda\la = 0\) in \(\Omega\).
\end{remark}

The precise pointwise identification of the zero-set \(S\) can be obtained using the Wiener test by Dal~Maso and Mosco~\cite{DalMaso_Mosco:1986,DalMaso_Mosco:1987}, which involves a capacity explicitly defined in terms of the potential \(V\).{}
If one is simply willing to get a rough location of \(S\), up to sets of \(W^{1, 2}\)~capacity zero, then a more elementary approach is to look for nontrivial elements of \(W_{0}^{1, 2}(\Omega) \cap L^{2}(\Omega; V \dif x)\)\,:

\begin{proposition}
	\label{propositionSingularSetSize}
	For every nonnegative function \(v \in W_{0}^{1, 2}(\Omega) \cap L^{2}(\Omega; V \dif x)\), we have
	\[{}
	\capt_{W^{1, 2}}{(\{\quasi{v} > 0\} \cap S)}
	= 0.
	\]
\end{proposition}

In other words, there exists a set \(R \subset \Omega\), possibly depending on \(v\), with \(W^{1, 2}\)~capacity zero and such that
\[{}
S \subset \{\quasi{v} = 0\} \cup R.
\]
Observe that \(R\) is always negligible with respect to the Lebesgue measure.
Therefore, \(S\) is negligible whenever there exists \emph{some} \(v \in W_{0}^{1, 2}(\Omega) \cap L^{2}(\Omega; V \dif x)\) such that \(v > 0\) almost everywhere in \(\Omega\).

\begin{proof}[Proof of \Cref{propositionSingularSetSize}]
	Assume by contradiction that the capacity is positive, and take a compact subset \(K \subset \{\quasi{v} > 0\} \cap S\) with positive \(W^{1, 2}\)~capacity.{}
	Let \(\nu\) be a finite positive Borel measure supported in \(K\) such that \(\nu \in (W_{0}^{1, 2}(\Omega))'\).{}
	The action of \(\nu\) as a continuous linear functional in \(W_{0}^{1, 2}(\Omega)\) is simply an integration with respect to \(\nu\)\,:
	\begin{alignat}{2}
		\tag*{}
	\nu[\varphi]
	& = \int_{\Omega}{\varphi \dif\nu}
	& \quad &\text{for every \(\varphi \in C_{c}^{\infty}(\Omega)\)}\\
	\intertext{and then, by density,}
	\label{eqActionW12}
	\nu[z]
	& = \int_{\Omega}{\quasi{z} \dif\nu}
	&\quad &\text{for every \(z \in W_{0}^{1, 2}(\Omega) \cap L^{\infty}(\Omega)\).}
	\end{alignat}
	Since \(\nu\) is supported in \(K\) and \(\quasi{T_{1}(v)} = T_{1}(\quasi{v}) > 0\) in \(K\), one then deduces that \(\nu[T_{1}(v)] >  0\). 
	Hence, the minimum of the energy functional 
	\[{}
	E(z) 
	= \frac{1}{2} \int_{\Omega}{(\abs{\nabla z}^{2} + Vz^{2})} - \nu[z]
	\]
	is negative in \(W_{0}^{1, 2}(\Omega) \cap L^{2}(\Omega; V \dif x)\), as the function \(s \in \R \mapsto E(s \, T_{1}(v))\) is decreasing in a neighborhood of \(s = 0\).{}
	
	We now denote by \(w\) the minimizer of \(E\).
	Using \eqref{eqActionW12}, one verifies that \(w\) is the duality solution of \eqref{eqDirichletProblem} with datum \(\nu\).{}
	Since \(\nu\) is supported in \(K \subset S\) and \(\quasi{\zeta_{f}} = 0\) in \(S\)  for every \(f \in L^{\infty}(\Omega)\), we then get
	\[{}
	\int_{\Omega}{w f}
	= \int_{\Omega}{\quasi{\zeta_{f}} \dif \nu}
	= 0
	\quad \text{for every \(f \in L^{\infty}(\Omega)\).}
	\]
	Hence, \(w = 0\) almost everywhere in \(\Omega\), which contradicts the fact that \(E(w) < 0\).
\end{proof}

%%%%%%%%%%%%%%%%%%%%%%%%%%%%%%%%%%%%%%%%%%%%%%%%%%%%%%%%%%%%%%%%%%%%%%
%%%%%%%%%%%%%%%%%%%%%%%%%%%%%%%%%%%%%%%%%%%%%%%%%%%%%%%%%%%%%%%%%%%%%%
%%%%%%%%%%%%%%%%%%%%%%%%%%%%%%%%%%%%%%%%%%%%%%%%%%%%%%%%%%%%%%%%%%%%%%

\section{Existence of a distributional solution with datum $\chi_{\Omega \setminus Z}$}
\label{sectionExistenceBounded}

As a preliminary step towards the proof of \Cref{theoremGoodMeasuresCharacterization}, we show in this section that

\begin{proposition}
	\label{corollaryBounded}
	The set \(\Omega \setminus Z\) is such that \(\zeta_{\chi_{\Omega \setminus Z}}\) satisfies \eqref{eqEquationSupersolution} with \(f = \chi_{\Omega \setminus Z}\).
\end{proposition}

The existence of a largest Borel set with such a property, without identification with \(\Omega \setminus Z\) and up to negligible sets, is straightforward:

\begin{lemma}
	\label{propositionZCharacteristic}
	There exists a Borel set \(A \subset \Omega\) such that 
	\(\zeta_{\chi_{A}}\) satisfies \eqref{eqEquationSupersolution} with \(f = \chi_{A}\) and
	\[{}
	Z = \bigl\{ x \in \Omega : \quasi{\zeta_{\chi_{A}}}(x) = 0 \bigr\}.
	\] 
	In particular, \(Z\) is a Sobolev-closed set.
\end{lemma}

\begin{proof}[Proof of \Cref{propositionZCharacteristic}]
	Let 
	\[{}
	\alpha = 
	\sup{
	\left\{ \abs{B} 
	\left|
	 \begin{aligned}
	 & B \subset \Omega\ \text{is a Borel set and}\\
	 & \text{\eqref{eqDirichletProblem} has a distributional solution with \(\mu = \chi_{B} \dif x\)}
	 \end{aligned}
	\right.
	\right\}}.
	\]
	We first prove that the supremum is achieved by some Borel set \(A \subset \Omega\).{}
	To this end, take a maximizing sequence of Borel sets \((B_{k})_{k \in \N}\).{}
	We observe that \(A_{n} \vcentcolon= \bigcup\limits_{k = 0}^{n}{B_{k}}\) satisfies
	\[{}
	0 \le \chi_{A_{n}}
	\le \sum_{k = 0}^{n}{\chi_{B_{k}}}.
	\] 
	By linearity of the equation, there exists a distributional solution with \(\sum\limits_{k = 0}^{n}{\chi_{B_{k}}} \dif x\).
	Hence, by \Cref{propositionDistributionalExistence}, the Dirichlet problem \eqref{eqDirichletProblem} also has a distributional solution with datum \(\mu = \chi_{A_{n}} \dif x\).{}
	Since the sequence \((\chi_{A_{n}})_{n \in \N}\) is nondecreasing and bounded in \(L^{1}(\Omega)\), we deduce using the Monotone convergence theorem and the Sobolev estimate~\eqref{eqEstimateSobolev}	
	that \eqref{eqDirichletProblem} has a distributional solution with datum \(\mu = \chi_{A} \dif x\), where \(A = \bigcup\limits_{n=0}^{\infty}{A_{n}}\).{}
	Since
	\[{}
	\abs{A} = \lim_{n \to \infty}{\meas{A_{n}}}
	\ge \lim_{n \to \infty}{\meas{B_{n}}}
	= \alpha,
	\]
	the set \(A\) achieves the supremum above.
	As \eqref{eqDirichletProblem} has a distributional solution with \(\mu =  \chi_{A} \dif x\), by \Cref{remarkDistributionalSolutionsVariational} such a solution must be \(\zeta_{\chi_{A}}\).{}

	\begin{Claim}
		If \(f \in L^{\infty}(\Omega)\) is a nonnegative function such that \eqref{eqDirichletProblem} has a distributional solution with \(\mu = f \dif x\), then \(f = 0\) almost everywhere in \(\Omega \setminus A\).
	\end{Claim}
	
	\begin{proof}[Proof of the Claim]
	We use the maximality of the set \(A\).
	To this end, given a nonnegative \(f \in L^{\infty}(\Omega)\) such that \eqref{eqDirichletProblem} has a distributional solution with measure \(f \dif x\), assume by contradiction that the set \(B \vcentcolon= \{f > \epsilon\} \setminus A\) has positive Lebesgue measure for some \(\epsilon > 0\).{}
	Since \(0 \le \epsilon \chi_{B} \le f\), by \Cref{propositionDistributionalExistence}  the Dirichlet problem also has a distributional solution with measure \(\epsilon \chi_{B} \dif x\) and, by linearity of the equation, with \(\chi_{B} \dif x\), and then also with \(\chi_{A \cup B} \dif x = (\chi_{A} + \chi_{B}) \dif x\).{}
	Since \(\abs{A \cup B} > \abs{A}\), we have a contradiction with the maximality of \(A\).{}
	\end{proof}

	From the Claim, we deduce that if \(w\) is a distributional solution of \eqref{eqDirichletProblem} with \(\mu = f \dif x\), for some nonnegative \(f \in L^{\infty}(\Omega)\), then
	\[{}
	0 \le {f} \le \norm{f}_{L^{\infty}(\Omega)} \chi_{A}
	\quad \text{almost everywhere in \(\Omega\).}
	\]	
	Applying the weak maximum principle, we thus have
	\[{}
	0 \le w
	\le \norm{f}_{L^{\infty}(\Omega)} \zeta_{\chi_{A}}
	\quad \text{almost everywhere in \(\Omega\).}
	\]
	Hence, the precise representatives satisfy
	\[{}
	0 \le \quasi{w}
	\le \norm{f}_{L^{\infty}(\Omega)} \quasi{\zeta_{\chi_{A}}}
	\quad \text{in \(\Omega\).}
	\]
	Since \(w\) is arbitrary, it follows that \(Z = \{\quasi{\zeta_{\chi_{A}}} = 0\}\), and this concludes the proof of the lemma.
\end{proof}

To show that the set \(A\) above can be taken at least as large as \(\Omega \setminus Z\) we proceed in the spirit of Perron's method.
To this end, we introduce a function \(w\) which dominates all subsolutions of \eqref{eqDirichletProblem} for a fixed measure \(\mu\) and such that if \eqref{eqDirichletProblem} has a distributional solution, then such a solution must be \(w\).{}
More precisely,

\begin{lemma}
	\label{propositionExistenceTool}
	For every nonnegative measure \(\mu \in \cM(\Omega)\), there exists a nonnegative function \(w \in W_{0}^{1, 1}(\Omega)\) such that \(\chi_{\Omega\setminus Z} w \in L^{1}(\Omega; V \dif x)\), \(\Delta w \in \cM(\Omega)\), 
	\[{}
	- \Delta w + V \chi_{\Omega\setminus Z} w \le \mu{}
	\quad \text{in the sense of distributions in \(\Omega\)}
	\]
	and
	\[{}
	- \Delta (\chi_{\Omega\setminus Z} w) + V \chi_{\Omega\setminus Z} w \ge \mu\ld \lfloor_{\Omega\setminus Z}
		\quad \text{in the sense of distributions in \(\Omega\),}
	\]	
	where \(\mu\ld\) is the diffuse part of \(\mu\) with respect to the \(W^{1, 2}\)~capacity.
	Moreover, for every \(u \in W_{0}^{1, 1}(\Omega) \cap L^{1}(\Omega; V \dif x)\) such that
	\[{}
	- \Delta u + V u \le \mu 
	\quad \text{in the sense of distributions in \(\Omega\),}
	\]
	we have 
	\[{}
	u \le w
	\quad \text{almost everywhere in \(\Omega\).}
	\]
\end{lemma}

For the sake of proving \Cref{corollaryBounded}, we could have restricted ourselves to measures of the form \(\mu = f \dif x\) with \(f \in L^{1}(\Omega)\), which satisfy in particular \(\mu\ld = \mu\).{}
The statement for an arbitrary measure is used in the proof of \Cref{theoremGoodMeasuresCharacterization} in \Cref{sectionProofThm1} and we also show that
\[{}
w = 0 
\quad \text{almost everywhere in \(Z\).}
\]

To prove \Cref{propositionExistenceTool} we rely on a truncation strategy where the truncation level depends on \(x \in \Omega\).
We begin with the following observation:

\begin{lemma}
	\label{lemmaExistenceTruncation}
	Let \(v \in L^{1}(\Omega; V \dif x)\) be a nonnegative function.
	For every nonnegative measure \(\mu \in \cM(\Omega)\), there exists \(u \in W_{0}^{1, 1}(\Omega)\) such that
	\[{}
	- \Delta u + V \min{\{v, u\}} = \mu{}
	\quad{}
	\text{in the sense of distributions in \(\Omega\).}
	\]
\end{lemma}

\begin{proof}[Proof of \Cref{lemmaExistenceTruncation}]
	We proceed by approximation by taking a nonnegative sequence \((\mu_{k})_{k \in \N}\) in \(L^{2}(\Omega)\) which is bounded in \(L^{1}(\Omega)\) and converges to \(\mu\) in the sense of measures in \(\Omega\)\,; an example is \(\mu_{k} = \rho_{k} * \mu\) where \((\rho_{k})_{k \in \N}\) is a sequence of mollifiers.{}
	Consider the energy functional
	\[{}
	\widetilde E_{k}(z)
	= \frac{1}{2} \int_{\Omega}{\abs{\nabla z}^{2}}
	+ \int_{\Omega}{g(\cdot , z)}
	- \int_{\Omega}{\mu_{k} z},
	\]
	where
	\[{}
	g(x, t)
	\vcentcolon= V(x) \int_{0}^{t}{T_{v(x)}(s) \dif s}
	\quad \text{for every \((x, t) \in \Omega \times \R\).}
	\]
	We take \(\widetilde E_{k}\) defined on 
	\[{}
	\mathcal V 
	\vcentcolon= \bigl\{ z \in W_{0}^{1, 2}(\Omega) : g(\cdot, z) \in L^{1}(\Omega) \bigr\}.
	\]	
	
	Existence of a solution of the Euler-Lagrange equation associated to \(\widetilde E_{k}\) follows from \cite{Brezis_Browder:1978}*{Theorem~1} by Brezis and Browder, based on a truncation of \(g\).{}
	Here we prove directly that the minimizer satisfies the equation.
 	We first claim that in our case \(\mathcal V\) is a vector subspace of \(W_{0}^{1, 2}(\Omega)\).{}
	To this end, observe that \(g\) is even, nondecreasing in \([0, +\infty)\), and satisfies the \(\Delta_{2}\) condition
	\[{}
	0 \le g(\cdot, 2t)
	\le C g(\cdot, t)
	\quad \text{in \(\Omega\)}
	\]
	for every \(t \in \R\) and some constant \(C > 0\).{}
	Thus, for every \(t_{1}, t_{2} \in \R\),
	\[{}
	0 \le g(\cdot, t_{1} + t_{2})
	\le C \bigl( g(\cdot, t_{1}) + g(\cdot, t_{2}) \bigr)
	\quad \text{in \(\Omega\).}
	\]
	These properties of \(g\) imply that the condition \(g(\cdot, z) \in L^{1}(\Omega)\) is stable under linear combinations of \(z \in L^{1}(\Omega)\), and then \(\mathcal V\) is a vector subspace of \(W_{0}^{1, 2}(\Omega)\) as claimed.{}
	
	Since \(g\) is nonnegative, \(\widetilde E_{k}\) is bounded from below in \(\mathcal V\).{}
	Moreover, by the Rellich-Kondrashov compactness theorem and Fatou's lemma, any minimizing sequence of \(\widetilde E_{k}\) in \(\mathcal V\) has a subsequence that converges weakly in \(W_{0}^{1, 2}(\Omega)\) to a minimizer \(u_{k} \in \mathcal V\).{}
	We now observe that
	\begin{equation*}
		W_{0}^{1, 2}(\Omega) \cap L^{\infty}(\Omega){}
		\subset \mathcal V,
	\end{equation*}
	which follows from the assumption \(v \in L^{1}(\Omega; V \dif x)\) and the fact that 
	\[{}
	0 \le g(\cdot, t) \le Vv \, \abs{t}
	\quad \text{for every \(t \in \R\).}
	\]
	Since \(\mathcal V\) is a vector space that contains \(W_{0}^{1, 2}(\Omega) \cap L^{\infty}(\Omega)\), the minimizer \(u_{k}\) satisfies the Euler-Lagrange equation
	\[{}
	\int_{\Omega}{\bigl( \nabla u_{k} \cdot \nabla z + V T_{v}(u_{k}) z \bigr)}
	= \int_{\Omega}{\mu_{k} z}
	\quad \text{for every \(z \in W_{0}^{1, 2}(\Omega) \cap L^{\infty}(\Omega)\).}
	\]
	Since \(\mu_{k}\) is nonnegative, one deduces that \(u_{k}\) is also nonnegative.
	Hence, \(T_{v}(u_{k}) = \min{\{v, u_{k}\}}\) and
	\begin{equation}
		\label{eq-1594}
	- \Delta u_{k} + V \min{\{v, u_{k}\}} = \mu_{k}
	\quad{}
	\text{in the sense of distributions in \(\Omega\).}
	\end{equation}
	
	We next observe that
	\[{}
	0 \le V \min{\{v, u_{k}\}}
	\le Vv
	\quad \text{for every \(k \in \N\).}
	\]
	From equation \eqref{eq-1594} and the assumption on \(v\), the sequence \((\Delta u_{k})_{k \in \N}\) is then bounded in \(L^{1}(\Omega)\).{}
	By Sobolev imbedding of solutions of the Dirichlet problem, we can extract a subsequence from \((u_{k})_{k \in \N}\) which converges in \(L^{1}(\Omega)\) to some function \(u \in W_{0}^{1, 1}(\Omega)\).{}
	We then have the conclusion using the Dominated convergence theorem.
\end{proof}

\begin{proof}[Proof of \Cref{propositionExistenceTool}]
	Let \(v \vcentcolon= \zeta_{\chi_{A}}\), where \(A\) is the set given by \Cref{propositionZCharacteristic}.	
	For each \(k \in \N\), let \(w_{k} \in W_{0}^{1, 1}(\Omega)\) be such that
	\begin{equation}
	\label{eqEquationTruncation}
	-\Delta w_{k} + V \min{\{k v, w_{k}\}} = \mu{}
	\quad{}
	\text{in the sense of distributions in \(\Omega\).}
	\end{equation}
	The existence of \(w_{k}\) follows from \Cref{lemmaExistenceTruncation} applied to the nonnegative function \(kv\).{}
	Since the function \(t \in \R \mapsto \min{\{kv, t\}}\) is nondecreasing, 
	the weak maximum principle applies; see e.g.~\cite{Brezis_Marcus_Ponce:2007}*{Corollary~4.B.2}.{}
	The sequence \((w_{k})_{k \in \N}\) is then nonnegative and non-increasing, whence converges pointwise and in \(L^{1}(\Omega)\) to some function \(w\).{}
	By construction of \(v\), the set \(\{v = 0\}\) equals \(Z\), except for a negligible set, so that the sequence \((\min{\{k v, w_{k}\}})_{k \in \N}\) converges almost everywhere to \(\chi_{\Omega \setminus Z}\,  w\).{}
	By the absorption estimate, 
	\[{}
	\bignorm{V \min{\{k v, w_{k}\}}}_{L^{1}(\Omega)}
	\le \norm{\mu}_{\cM(\Omega)},
	\]
	the sequence \((\min{\{k v, w_{k}\}})_{k \in \N}\) is bounded in \(L^{1}(\Omega; V \dif x)\).	
	Hence, by Fatou's lemma we have \(\chi_{\Omega \setminus Z}\, w \in L^{1}(\Omega; V \dif x)\) and
	\begin{equation}
	\label{eqTheoremFatou}
		-\Delta w + V \chi_{\Omega \setminus Z}\,  w \le \mu
	\quad{}
	\text{in the sense of distributions in \(\Omega\).}
	\end{equation}
	By the boundedness of the sequence \((\Delta w_{k})_{k \in \N}\) in \(\cM(\Omega)\), we also have \(w \in W_{0}^{1, 1}(\Omega)\) and \(\Delta w \in \cM(\Omega)\).{}

	We now suppose that we are given some function \(u \in W_{0}^{1, 1}(\Omega)\) such that
	\[{}
	- \Delta u + V u \le \mu 
	\quad \text{in the sense of distributions in \(\Omega\).}
	\] 
	Then, \(u\) is a subsolution of equation \eqref{eqEquationTruncation} and,
	by the weak maximum principle, we have \(u \le w_{k}\) almost everywhere in \(\Omega\).{}
	As \(k \to \infty\), we get \(u \le w\) almost everywhere in \(\Omega\).{}
	
	We are left with the proof of 
	\begin{equation}
	\label{eqToolDistributional}
	- \Delta (\chi_{\Omega\setminus Z} w) + V \chi_{\Omega\setminus Z} w \ge \mu\ld \lfloor_{\Omega\setminus Z}
	\quad \text{in the sense of distributions in \(\Omega\).}
	\end{equation}
	To this end, we begin by writing, for every \(a, b \in \R\),{}
	\begin{equation}
		\label{eqKato1}
	\min{\{a, b\}}
	= \frac{a + b - (a - b)^{+} - (b - a)^{+}}{2},
	\end{equation}
	which we shall apply with \(a = \ell v\) and \(b = w_{k}\), where \(\ell \in \N\).
	By Kato's inequality~\cites{Brezis_Ponce:2004,DalMaso_Murat_Orsina_Prignet:1999}, \(\Delta (\ell v - w_{k})^{+}\) and \(\Delta (w_{k} - \ell v)^{+}\) are locally finite measures in \(\Omega\) that satisfy
	\begin{equation}
	\label{eqKato2}
	\bigl[\Delta (\ell v - w_{k})^{+}\bigr]\ld
	 \ge \chi_{\{\ell\quasi{v} > \quasi{w_{k}}\}} \bigl[ \Delta(\ell v - w_{k}) \bigr]\ld{}
	\end{equation}
	and
	\begin{equation}
	\label{eqKato3}
	\bigl[\Delta (w_{k} - \ell v)^{+}\bigr]\ld
	\ge \chi_{\{\ell\quasi{v} < \quasi{w_{k}}\}} \bigl[ \Delta(w_{k} - \ell v) \bigr]\ld
	\end{equation}
	in the sense of measures in \(\Omega\).{}
	Here we recall that, since \(\Delta w \in \cM(\Omega)\), the precise representative \(\quasi{w}\) is defined quasi-everywhere in \(\Omega\), i.e.~except on a subset of \(W^{1, 2}\)~capacity zero; see \cite{Ponce:2016}*{Proposition~8.9}.	
	It thus follows from \eqref{eqKato1} to \eqref{eqKato3} that \(\Delta \min{\{\ell v, w_{k}\}}\) is a locally finite measure in \(\Omega\) such that
	\begin{multline*}
	\bigl(\Delta \min{\{\ell v, w_{k}\}}\bigr)\ld{}
	\le \chi_{\{\ell\quasi{v} < \quasi{w_{k}}\}} (\ell \Delta v)\ld{}
	+ \chi_{\{\ell\quasi{v} > \quasi{w_{k}}\}} (\Delta w_{k})\ld{}\\
	+  \chi_{\{\ell\quasi{v} = \quasi{w_{k}}\}} \frac{(\ell \Delta v)\ld + (\Delta w_{k})\ld}{2}.
	\end{multline*}
	Observe that since \(v\) is a distributional solution with datum \(\chi_{A} \dif x\) (and not just a duality solution) we have
	\[{}
	(\ell\Delta v)\ld = \ell\Delta v = \ell Vv - \ell\chi_{A} \le \ell Vv.
	\]
	We also have
	\[{}
	(\Delta w_{k})\ld{}
	= V \min{\{k v, w_{k}\}} - \mu\ld.
	\]
	Thus, 
	\begin{multline*}
	\bigl(\Delta \min{\{\ell v, w_{k}\}}\bigr)\ld{}
	\le \chi_{\{\ell\quasi{v}{<} \quasi{w_{k}}\}} \ell V v
	+ \chi_{\{\ell\quasi{v} \, > \quasi{w_{k}}\}} (V \min{\{k v, w_{k}\}} - \mu\ld)\\
	+  \chi_{\{\ell\quasi{v} = \quasi{w_{k}}\}} \frac{\ell V v + V \min{\{k v, w_{k}\}}}{2}.
	\end{multline*}	
	For \(\ell \le k\), we have
	\[{}
	\chi_{\{\ell\quasi{v} < \quasi{w_{k}}\}} \ell V v
	+ \chi_{\{\ell\quasi{v} \, > \quasi{w_{k}}\}} V \min{\{k v, w_{k}\}} 
	+  \chi_{\{\ell\quasi{v} = \quasi{w_{k}}\}} \frac{\ell V v + V \min{\{k v, w_{k}\}}}{2}
	= V \min{\{\ell v, w_{k}\}}
	\]
	almost everywhere in \(\Omega\).{}
	Hence,
	\[{}
	\bigl(\Delta \min{\{\ell v, w_{k}\}}\bigr)\ld{}
	\le V \min{\{\ell v, w_{k}\}} - \chi_{\{\ell\quasi{v} \, > \quasi{w_{k}}\}}\mu\ld.
	\]
	Since \(w_{k} \le w_{0}\) and \(\mu\ld\) is nonnegative, we then have
	\begin{equation}
		\label{eq1587}
	\bigl(\Delta \min{\{\ell v, w_{k}\}}\bigr)\ld{}
	\le V \min{\{\ell v, w_{k}\}} - \chi_{\{\ell\quasi{v} \, > \quasi{w_{0}}\}} \mu\ld.
	\end{equation}
	Since the function \( \min{\{\ell v, w_{k}\}}\) is nonnegative, by the inverse maximum principle we also have
	\begin{equation}
		\label{eq1593}
	\bigl(\Delta \min{\{\ell v, w_{k}\}}\bigr)\lc 
	\le 0. 
	\end{equation}
	Combining \eqref{eq1587} and \eqref{eq1593}, for every \(\ell \le k\) we get
	\[{}
	\Delta \min{\{\ell v, w_{k}\}}
	\le 
	V \min{\{\ell v, w_{k}\}} - \chi_{\{\ell\quasi{v} \, > \quasi{w_{0}}\}} \mu\ld
	\]
	in the sense of measures and then also in the sense of distributions in \(\Omega\).
	Letting \(k \to \infty\) and next \(\ell \to \infty\), we deduce that
	\[{}
	\Delta (\chi_{\Omega \setminus Z} w)
	\le 
	V \chi_{\Omega \setminus Z} w - \chi_{\{\quasi{w_{0}} < \infty\}\setminus \{\quasi{v} = 0\}} \, \mu\ld
	\quad \text{in the sense of distributions in \(\Omega\).}
	\]
	Since \(\Delta w_{0} \in \cM(\Omega)\), the set \(\Omega \setminus \{\quasi{w_{0}} < \infty\}\) has \(W^{1, 2}\)~capacity zero.
	Moreover, by the choice of \(v\) we have \(\{\quasi{v} = 0\} = Z\).{}
	We thus get
	\[{}
	\chi_{\{\quasi{w_{0}} < \infty\}\setminus \{\quasi{v} = 0\}} \, \mu\ld 
	= \mu\ld\lfloor_{\Omega \setminus Z}
	\]
	and \eqref{eqToolDistributional} follows.{}
\end{proof}

\begin{remark}
	\label{remarkBoundary}
	\Cref{propositionExistenceTool} does not provide enough information to conclude that \(\chi_{\Omega \setminus Z} w \in W_{0}^{1, 1}(\Omega)\).{}
	To encode the zero boundary datum of \(\chi_{\Omega \setminus Z} w\), one can rely instead on test functions in the larger class \(C_{0}^{\infty}(\overline\Omega)\), which is enough to apply the weak maximum principle.
	On the one hand, since \(w \in W_{0}^{1, 1}(\Omega)\), by \cite{Ponce:2016}*{Proposition~6.5} the property
	\[{}
	- \Delta w + V \chi_{\Omega\setminus Z} w \le \mu{}
	\quad \text{in the sense of distributions in \(\Omega\)}
	\]
	is equivalent to
	\[
	- \Delta w + V \chi_{\Omega\setminus Z} w \le \mu{}
	\quad \text{in the sense of \((C_{0}^{\infty}(\overline\Omega))'\).}
	\]
	On the other hand, since \(w \in W_{0}^{1, 1}(\Omega)\) and \(\Delta w \in \cM(\Omega)\), by \cite{Ponce:2016}*{Proposition~20.1} we also have 
	\[{}
	\int_{\{x \in \Omega\, :\, d(x, \partial\Omega) < \epsilon\}}{w}
	\le C \epsilon^{2} \norm{\Delta w}_{\cM(\Omega)}
	\quad \text{for every \(\epsilon > 0\).}
	\]
	In particular, \(\chi_{\Omega \setminus Z} w\) satisfies the following vanishing mean property on the boundary:
	\[{}
	\lim_{\epsilon \to 0}{\frac{1}{\epsilon}\int_{\{x \in \Omega : d(x, \partial\Omega) < \epsilon\}}{\chi_{\Omega \setminus Z} w}}
	= 0,
	\]
	which combined with an inequality of the type
		\[{}
	- \Delta (\chi_{\Omega\setminus Z} w) + V \chi_{\Omega\setminus Z} w \ge \mu\ld \lfloor_{\Omega\setminus Z}
		\quad \text{in the sense of distributions in \(\Omega\)}
	\]
	entitles us to recover test functions in \(C_{0}^{\infty}(\overline\Omega)\) as an application of \cite{Ponce:2016}*{Proposition~20.2}:
	\[{}
	- \Delta (\chi_{\Omega\setminus Z} w) + V \chi_{\Omega\setminus Z} w \ge \mu\ld \lfloor_{\Omega\setminus Z}
		\quad \text{in the sense of\/ \((C_{0}^{\infty}(\overline\Omega))'\).}
	\]
\end{remark}

\begin{proof}[Proof of \Cref{corollaryBounded}]
	Let \(w\) be the function given by \Cref{propositionExistenceTool} with \(\mu = \chi_{\Omega\setminus Z} \dif x\).{}
	In this case, since \(\mu\) is absolutely continuous with respect to the Lebesgue measure, \(\mu\ld = \mu\) and then \(\mu\ld\lfloor_{\Omega\setminus Z}{} = \mu\).
	Using the notation \(\widetilde{w} \vcentcolon= \chi_{\Omega \setminus Z} w\), we thus have
	\[{}
	- \Delta w + V \chi_{\Omega\setminus Z} w \le \chi_{\Omega\setminus Z}{}
	\quad\text{and}\quad{}
	- \Delta \widetilde{w} + V \chi_{\Omega\setminus Z} \widetilde{w} \ge \chi_{\Omega\setminus Z}
	\]
	in the sense of distributions in \(\Omega\).{}
	By \Cref{remarkBoundary}, both inequalities hold in the sense of\/ \((C_{0}^{\infty}(\overline\Omega))'\).
	Since the potential \(V \chi_{\Omega\setminus Z}\) is nonnegative, it thus follows from the weak maximum principle that \(\widetilde{w} \ge w\) almost everywhere in \(\Omega\).{}
	By the nonnegativity of \(w\), the reverse inequality also holds.
	Hence, \(\widetilde{w} = w \in W_{0}^{1, 1}(\Omega)\) and, as \(\widetilde{w} = 0\) on \(Z\),
	\[{}
	-\Delta w + V w  = -\Delta w + V \chi_{\Omega\setminus Z} w = \chi_{\Omega\setminus Z}
	\quad \text{in the sense of distributions in \(\Omega\).}
	\]
	From \Cref{remarkDistributionalSolutionsVariational}, we thus have \(w = \zeta_{\chi_{\Omega \setminus Z}}\).
\end{proof}

\begin{remark}
	\label{remarkExistenceBounded}
	As a consequence of \Cref{corollaryBounded}, one has existence of a solution of \eqref{eqEquationSupersolution} for every nonnegative function \(f \in L^{\infty}(\Omega)\) such that \(f = 0\) almost everywhere in \(Z\).{}
	Indeed, observe that a solution of \eqref{eqEquationSupersolution} with datum \(\norm{f}_{L^{\infty}(\Omega)} \chi_{\Omega \setminus Z}\) exists.
	Since
	\[{}
	0 \le f \le \norm{f}_{L^{\infty}(\Omega)} \chi_{\Omega \setminus Z}
	\quad \text{almost everywhere in \(\Omega\),}
	\]
	it then suffices to apply the method of sub- and supersolutions (\Cref{propositionDistributionalExistence}).
\end{remark}

%%%%%%%%%%%%%%%%%%%%%%%%%%%%%%%%%%%%%%%%%%%%%%%%%%%%%%%%%%%%%%%%%%%%%%
%%%%%%%%%%%%%%%%%%%%%%%%%%%%%%%%%%%%%%%%%%%%%%%%%%%%%%%%%%%%%%%%%%%%%%
%%%%%%%%%%%%%%%%%%%%%%%%%%%%%%%%%%%%%%%%%%%%%%%%%%%%%%%%%%%%%%%%%%%%%%

\section{Orthogonality principle}
\label{sectionOrthogonality}

We establish in this section an orthogonality relation between the sets \(Z\) and \(\Omega\setminus Z\), which is used in the proof of \Cref{theoremGoodMeasuresCharacterization}\,:{}

\begin{proposition}
	\label{propositionOrthogonality}
	The universal zero-set \(Z\) satisfies
	\[{}
	\int_{\Omega}{\zeta_{\chi_{\Omega \setminus Z}}\chi_{Z}}
	= \int_{\Omega}{\zeta_{\chi_{Z}}\chi_{\Omega\setminus Z}}
	= 0.
	\]
	More precisely, we have
	\begin{alignat*}{2}
	\quasi{\zeta_{\chi_{\Omega \setminus Z}}}(x) & = 0
	&& \quad \text{for every \(x \in Z\),}\\
	\quasi{\zeta_{\chi_{Z}}}(x) & = 0
	&& \quad \text{for every \(x \in \Omega\setminus Z\).}
	\end{alignat*}
\end{proposition}

As a consequence, the function \(\zeta_{\chi_{Z}}\) satisfies \eqref{eqEquationSupersolution} with \(f = \chi_{Z}\) if and only if \(Z\) is negligible, since one must have \(\quasi{\zeta_{\chi_{Z}}} = 0\) in \((\Omega\setminus Z) \cup Z = \Omega\).
The proof of \Cref{propositionOrthogonality} relies on the existence of a solution of \eqref{eqEquationSupersolution} with \(f = \chi_{\Omega\setminus Z}\) that we proved in the previous section.{}
We also need to know that every point of \(\Omega \setminus Z\) is a density point of this set, which means that \(\Omega \setminus Z\) is open with respect to the density topology~\cite{Goffman_Waterman:1961}.
This is a general property of Sobolev-open sets (\Cref{propositionQuasiOpenDensityPoint}), but here we rely solely on the definition of \(Z\)\,:

\begin{lemma}
	\label{propositionDensity}
	For every \(x \in \Omega \setminus Z\), we have
	\[{}
	\lim_{r \to 0}{\frac{\meas{B_{r}(x) \setminus Z}}{\meas{B_{r}(x)}}} = 1.
	\]
\end{lemma}

\begin{proof}[Proof of \Cref{propositionDensity}]
	Given \(x \in \Omega \setminus Z\), let \(w \in W_{0}^{1, 2}(\Omega) \cap L^{\infty}(\Omega)\) be a solution of \eqref{eqEquationSupersolution}
	for some nonnegative \(f \in L^{\infty}(\Omega)\) with \(\quasi{w}(x) > 0\).
	Since \(\quasi{w} = 0\) in \(Z\), by the Lebesgue differentiation theorem we have \(w = 0\) almost everywhere in \(Z\).{}
	Thus,
	\begin{equation}
		\label{eqOrthogonality4}
	\fint_{B_{r}(x)}{w}
	= \frac{1}{\meas{B_{r}(x)}} \int_{B_{r}(x) \setminus Z}{w}.
	\end{equation}
	We now denote \(c \vcentcolon= \quasi{w}(x)\) and choose \(r_{1} > 0\) such that
	\begin{equation}
		\label{eqOrthogonality5}
	c - \epsilon \le \fint_{B_{r}(x)}{w}
	\quad \text{for every \(0 < r \le r_{1}\).}
	\end{equation}

	Observe that \(\quasi{w}\), being the difference between a continuous and a superharmonic function, is upper semicontinuous in \(\Omega\).{}
	Thus,
	\[
	\limsup_{y \to x}{\quasi{w}(y)}
	\le \quasi{w}(x) = c.
	\]
	We then take \(r_{2} > 0\) such that
	\[{}
	\quasi{w}(y)
	\le c + \epsilon
	\quad \text{for every \(y \in B_{r_{2}}(x)\).}
	\]	
	In particular, 
	\(w	\le c + \epsilon\) almost everywhere in \(B_{r_{2}}(x)\),
	which implies that
	\begin{equation}
		\label{eqOrthogonality6}
	\frac{1}{\meas{B_{r}(x)}} \int_{B_{r}(x) \setminus Z}{w}
	\le (c + \epsilon) \frac{\meas{B_{r}(x) \setminus Z}}{\meas{B_{r}(x)}}
	\quad 	\text{for every \(0 < r \le r_{2}\).}
	\end{equation}

	Combining \eqref{eqOrthogonality4} to \eqref{eqOrthogonality6},  we get
	\[{}
	c - \epsilon 
	\le (c + \epsilon) \frac{\meas{B_{r}(x) \setminus Z}}{\meas{B_{r}(x)}}
	\quad \text{for every \(0 < r \le \min{\{r_{1}, r_{2}\}}\).}
	\]
	Therefore,	as \(r \to 0\),
	\[{}
	\frac{c - \epsilon}{c + \epsilon} 
	\le \liminf_{r \to 0}{\frac{\meas{B_{r}(x) \setminus Z}}{\meas{B_{r}(x)}}}
	\le \limsup_{r \to 0}{\frac{\meas{B_{r}(x) \setminus Z}}{\meas{B_{r}(x)}}}
	\le 1.
	\]
	Since \(c > 0\), the conclusion follows as \(\epsilon \to 0\).
\end{proof}

\begin{proof}[Proof of \Cref{propositionOrthogonality}]
	By \Cref{corollaryBounded}, the function \(\zeta_{\chi_{\Omega \setminus Z}}\) satisfies \eqref{eqEquationSupersolution} with \(f = \chi_{\Omega \setminus Z}\) and, in particular,
	\begin{equation}
		\label{eqOrthogonality1}
	\quasi{\zeta_{\chi_{\Omega \setminus Z}}}(x) = 0 
	\quad \text{for every \(x \in Z\).}
	\end{equation}
	To establish the integral orthogonality relation, we recall that \(\zeta_{\chi_{\Omega\setminus Z}}\) is also a duality solution.
	Using the test function \(\chi_{Z} \in L^{\infty}(\Omega)\) in the duality formulation, we then have by the Lebesgue differentiation theorem,
	\begin{equation}
		\label{eqOrthogonality2}
		\int_{\Omega}{\zeta_{\chi_{\Omega \setminus Z}}\chi_{Z}}
		= \int_{\Omega}{\quasi{\zeta_{\chi_{Z}}}\chi_{\Omega\setminus Z}}
		= \int_{\Omega}{\zeta_{\chi_{Z}}\chi_{\Omega\setminus Z}}~.
	\end{equation}
	By \eqref{eqOrthogonality1} and the Lebesgue differentiation theorem, \(\zeta_{\chi_{\Omega \setminus Z}} = 0\) almost everywhere in \(Z\).{}
	Hence,
	the integral in the left-hand side of \eqref{eqOrthogonality2} vanishes.
	This establishes the orthogonality identity and then, since \(\zeta_{\chi_{Z}}\) is nonnegative,
	\begin{equation}
		\label{eqOrthogonality3}
	{\zeta_{\chi_{Z}}} = 0 
	\quad \text{almost everywhere in \(\Omega\setminus Z\).}
	\end{equation}
	
	We now claim that
	\[{}
	\quasi{\zeta_{\chi_{Z}}}(x) = 0
	\quad \text{for every \(x \in \Omega \setminus Z\).}
	\]
	Indeed, by nonnegativity of \(\zeta_{\chi_{Z}}\) and \eqref{eqOrthogonality3}, for every ball \(B_{r}(x) \subset \Omega\) we have
	\begin{equation*}
	0 
	\le \fint_{B_{r}(x)}{\zeta_{\chi_{Z}}}
	= \frac{1}{\meas{B_{r}(x)}} \int_{B_{r}(x) \cap Z}{\zeta_{\chi_{Z}}}
	\le \norm{\zeta_{\chi_{Z}}}_{L^{\infty}(\Omega)} \frac{\meas{B_{r}(x) \cap Z}}{\meas{B_{r}(x)}}.
	\end{equation*}
	By \Cref{propositionDensity}, the right-hand side converges to zero as \(r \to 0\) when \(x \in \Omega \setminus Z\) and we conclude that \(\quasi{\zeta_{\chi_{Z}}}(x) = 0\).
\end{proof}

	From the orthogonality principle, we deduce \emph{a posteriori} that one can take \(A = \Omega \setminus Z\) in \Cref{propositionZCharacteristic}:

	\begin{corollary}
	\label{remarkSetA}
	The universal zero-set satisfies
	\[{}
	Z = \bigl\{ x \in \Omega : \quasi{\zeta_{\chi_{\Omega \setminus Z}}}(x) = 0 \bigr\}.
	\] 
	\end{corollary}

\begin{proof}
	We recall that \(A\) is defined in the proof of \Cref{propositionZCharacteristic} as a maximizer among all Borel sets \(B \subset \Omega\) such that \eqref{eqDirichletProblem} has a distributional solution with \(\mu = \chi_{B} \dif x\).{}
	Since by \Cref{corollaryBounded} a solution with \(B = \Omega \setminus Z\) exists, we may assume from the beginning that 
	\[{}
	\Omega \setminus Z \subset A.
	\]
	It thus suffices to verify that \(A \cap Z\) is negligible with respect to the Lebesgue measure.
	To this end, we first observe that by \Cref{propositionDistributionalExistence} there exists a distributional solution of \eqref{eqDirichletProblem} with datum \(\mu = \chi_{A \cap Z} \dif x\), which by \Cref{remarkDistributionalSolutionsVariational} can be identified with \(\zeta_{\chi_{A \cap Z}}\).{}
	On the other hand, by comparison between variational solutions, 
	\[{}
	0 
	\le \zeta_{\chi_{A \cap Z}}
	\le \zeta_{\chi_{Z}}
	\quad \text{almost everywhere in \(\Omega\).}
	\]
	From the orthogonality principle, and in particular \eqref{eqOrthogonality3}, we thus have \(\zeta_{\chi_{A \cap Z}} = 0\) almost everywhere in \(\Omega \setminus Z\).{}
	But being a distributional solution, the same property holds in \(Z\).{}
	Hence, \(\zeta_{\chi_{A \cap Z}} = 0\) almost everywhere in the entire domain \(\Omega\) and then, from the distributional formulation,
	\[{}
	\int_{A \cap Z}{\varphi}
	= \int_{\Omega}{\chi_{A \cap Z} \, \varphi}
	= 0
	\quad \text{for every \(\varphi \in C_{c}^{\infty}(\Omega)\).}
	\]
	Therefore, \(A \cap Z\) is negligible.
\end{proof}

%%%%%%%%%%%%%%%%%%%%%%%%%%%%%%%%%%%%%%%%%%%%%%%%%%%%%%%%%%%%%%%%%%%%%%
%%%%%%%%%%%%%%%%%%%%%%%%%%%%%%%%%%%%%%%%%%%%%%%%%%%%%%%%%%%%%%%%%%%%%%
%%%%%%%%%%%%%%%%%%%%%%%%%%%%%%%%%%%%%%%%%%%%%%%%%%%%%%%%%%%%%%%%%%%%%%

\section{Comparison principle}
\label{sectionComparison}

We investigate a comparison principle which establishes that every solution of the Dirichlet problem \eqref{eqDirichletProblem} with positive measure can always be bounded from below by a nontrivial solution involving some nonnegative \(L^{\infty}\) datum.
In the proof of \Cref{theoremGoodMeasuresCharacterization}, it implies that the assumption \(\mu(Z) = 0\) is necessary for the existence of distributional solutions.
While the naive strategy based on truncation gives a bounded supersolution \(T_{k}(u)\) underneath \(u\), such an approach is unsatisfactory since \(\Delta T_{k}(u)\) typically yields a singular measure on the level set \(\{u = k\}\).

Our main result in this direction is the following

\begin{proposition}
	\label{propositionTestFunctionPositive}
	There exists a bounded continuous nondecreasing function \(H : [0, +\infty) \to [0, +\infty)\), with \(H(t) > 0\) for \(t > 0\), such that, for every Borel function \(V : \Omega \to [0, +\infty]\), if \(u \in L^{1}(\Omega)\) is a duality solution of the Dirichlet problem \eqref{eqDirichletProblem} involving a nonnegative measure \(\mu \in \cM(\Omega)\), then
	\[{}
	u \ge \zeta_{H(u)}
	\quad \text{almost everywhere in \(\Omega\).}
	\]
\end{proposition}

Observe that \(\zeta_{H(u)}\) is well defined since \(H(u)\) is bounded.
We emphasize that \(H\) is independent of \(V\), and from the proof one can take \(H(t) \sim t^{\alpha}\) near \(t = 0\) for any given \(\alpha > 1\)\,; see \eqref{eqHChoice} below.
The comparison principle above also applies to distributional solutions, as they are also duality solutions, but the important fact that \(\zeta_{H(u)}\) is also a distributional solution with datum \(H(u)\) requires some justification; see \Cref{propositionExistenceBoundedSupersolutions} below.

To prove \Cref{propositionTestFunctionPositive}, we rely on a straightforward variant of Kato's inequality for \(\zeta_{h}\) in the spirit of \cite{Brezis_Marcus_Ponce:2007}*{Proposition~4.B.5}, which formally is
\[{}
 - \Delta\zeta_{h}^{+} + V \zeta_{h}^{+}
 \le \chi_{\{\zeta_{h} > 0\}} h,
\]
that also takes into account the boundary behavior of \(\zeta_{h}\) by allowing \(\zeta_{1}\) as test function.

\begin{lemma}
	\label{lemmaKatoL1}
	For every \(h \in L^{\infty}(\Omega)\), we have
\[{}
\int_{\Omega}{\zeta_{h}^{+}}
\le \int_{\{{\zeta_{h}} > 0\}}{ h \zeta_{1}}.
\]
\end{lemma}

\begin{proof}[Proof of \Cref{lemmaKatoL1}]
	Since \(\zeta_{h}\) and \(\zeta_{1}\) satisfy an Euler-Lagrange equation involving test functions in \(W_{0}^{1, 2}(\Omega) \cap L^{2}(\Omega; V \dif x)\), the proof is implemented by suitable choices of test functions depending on \(\zeta_{h}\) and \(\zeta_{1}\) themselves.
	For example, the equation satisfied by \(\zeta_{1}\) with test function \(J(\zeta_{h})\) gives
\begin{equation}
	\label{eqKatoFunctional1}
\int_{\Omega}{\bigl(J'(\zeta_{h}) \nabla\zeta_{1} \cdot \nabla \zeta_{h} + V\zeta_{1} J(\zeta_{h})\bigr)}
= \int_{\Omega}{J(\zeta_{h})},
\end{equation}
where \(J : \R \to \R\) is a smooth function such that \(J(0) = 0\).
	Using now the test function \(J'(\zeta_{h})\zeta_{1}\) in the equation satisfied by \(\zeta_{h}\), we also have 
	\[{}
	\int_{\Omega}{\bigl(J''(\zeta_{h}) \abs{\nabla \zeta_{h}}^{2} \zeta_{1} + J'(\zeta_{h}) \nabla \zeta_{h} \cdot \nabla \zeta_{1} + V\zeta_{h} J'(\zeta_{h})\zeta_{1}\bigr)} 
	= \int_{\Omega}{h \, J'(\zeta_{h}) \zeta_{1} }. 
	\]
	Assuming that \(J'' \ge 0\), by nonnegativity of \(\zeta_{1}\) we get
	\begin{equation}
	\label{eqKatoFunctional2}
	\int_{\Omega}{\bigl(J'(\zeta_{h}) \nabla \zeta_{h} \cdot \nabla \zeta_{1} + V\zeta_{h} J'(\zeta_{h}) \zeta_{1}\bigr)} 
	\le \int_{\Omega}{h \, J'(\zeta_{h}) \zeta_{1} }. 
	\end{equation}
	Subtracting \eqref{eqKatoFunctional2} from \eqref{eqKatoFunctional1},
	\[{}
	\int_{\Omega}{V \zeta_{1} [J(\zeta_{h}) - \zeta_{h} J'(\zeta_{h})]}  
	\ge \int_{\Omega}{J(\zeta_{h})}  - \int_{\Omega}{h \, J'(\zeta_{h})\zeta_{1}}.
	\]
	We now take \(J\) convex such that \(J(t) = 0\) for \(t \le 0\) and \(0 \le J(t) \le t\) for \(t \ge 0\).{}
	In particular, for every \(t \in \R\) we have \(J(t) \le J'(t) t\).{}
	Since \(V\) and \(\zeta_{1}\) are nonnegative, the integrand in the left-hand side is nonpositive and we deduce that
	\[{}
	\int_{\Omega}{J(\zeta_{h})} 
	\le \int_{\Omega}{h \, J'(\zeta_{h}) \zeta_{1}}.{}
	\]
	To conclude, we apply this inequality to a sequence \((J_{k})_{k \in \N}\) of convex functions as above that converges pointwise to the function \(t \in \R \mapsto t^{+}\) and such that \((J_{k}')_{k \in \N}\) converges pointwise to \(\chi_{(0, +\infty)}\).{}
	As \(k \to \infty\), we have the conclusion.
\end{proof}

\begin{proof}[Proof of \Cref{propositionTestFunctionPositive}]
	We first assume that \(\mu\) is a measure of the form \(\mu = f \dif x\) with a nonnegative \(f \in L^{\infty}(\Omega)\).{}
	We have in this case that \(u = \zeta_{f}\) by uniqueness of duality solutions.
	For every \(\epsilon > 0\), we claim that 
	\begin{equation}
		\label{eqEstimateComparison}
	C u \ge \epsilon \, \zeta_{\chi_{\{u > \epsilon\}}}
	\quad \text{almost everywhere in \(\Omega\)}
	\end{equation}
	for \(C \vcentcolon= \norm{\theta}_{L^{\infty}(\Omega)}\), where \(\theta\) is the classical solution of 
	\eqref{eqDirichletTheta}.
	
Using the notation \(z_{\epsilon} \vcentcolon= \zeta_{\chi_{\{u > \epsilon\}}}\)\,, we have \(\epsilon z_{\epsilon} - C u = \zeta_{h}\)\,, where \(h = \epsilon \chi_{\{u > \epsilon\}} - Cf\). 
Thus, by \Cref{lemmaKatoL1},
\[{}
\int_{\Omega}{(\epsilon z_{\epsilon} - Cu)^{+}} 
\le \int_{\{\epsilon z_{\epsilon} > C u\}}{(\epsilon \chi_{\{u > \epsilon\}} - C f) \zeta_{1}}.
\]
Since \(f\) and \(\zeta_{1}\) are nonnegative,
\begin{equation}
	\label{eqComparisonIntegral}
\int_{\Omega}{(\epsilon z_{\epsilon} - Cu)^{+}}
\le \int_{\{\epsilon z_{\epsilon} > C u\}}{\epsilon \chi_{\{u > \epsilon\}} \zeta_{1}}
= \epsilon \int_{\{\epsilon z_{\epsilon}/C > u > \epsilon\}} \zeta_{1}
\le \epsilon \int_{\{z_{\epsilon} > C\}} \zeta_{1}. 
\end{equation}
The estimate
\[{}
0 \le z_{\epsilon} \le \zeta_{1} \le \theta
\quad \text{almost everywhere in \(\Omega\)}
\]
holds for every \(\epsilon > 0\) and is independent of \(V\).
In particular, with the choice \(C = \norm{\theta}_{L^{\infty}(\Omega)}\), the set \(\{z_{\epsilon} > C\}\) is negligible with respect to the Lebesgue measure.
We then deduce from \eqref{eqComparisonIntegral} that
\[{}
\int_{\Omega}{(\epsilon z_{\epsilon} - Cu)^{+}} 
\le 0
\]
and this implies \eqref{eqEstimateComparison}.

To obtain \(H\), it now suffices to apply \eqref{eqEstimateComparison} using an averaging argument.
For this purpose, let \(\rho : (0, +\infty) \to \R\) be a summable nonnegative function such that
\(\int_{0}^{\infty}{\rho} = 1\).
Multiplying both sides of \eqref{eqEstimateComparison} by \(\rho(\epsilon)\) and integrating with respect to \(\epsilon\) over \((0, +\infty)\), we get
\[{}
C u(x) \ge \int_{0}^{\infty}{\epsilon \rho(\epsilon) \zeta_{\chi_{\{u > \epsilon\}}}(x) \dif \epsilon}
\quad \text{for almost every \(x \in \Omega\).}
\]
By linearity of the equation, one identifies the right-hand side as \(\zeta_{\widetilde H(u)}(x)\), where
\[
\widetilde H(t) 
\vcentcolon= \int_{0}^{t}{\epsilon \rho(\epsilon) \dif\epsilon},
\]
so that the proposition holds with \(H(t) = \widetilde H(t)/C\).
Given \(\alpha > 1\), an explicit admissible choice of \(\rho\) is 
\[{}
\rho(\epsilon) = 
\begin{cases}
	(\alpha - 1)\epsilon^{\alpha - 2}
	& \text{for \(\epsilon < 1\),}\\
	0
	& \text{for \(\epsilon \ge 1\).}
\end{cases}
\]
In this case, 
\begin{equation}
\label{eqHChoice}
H(t) 
= \frac{\alpha - 1}{C \alpha} \min{\{t^{\alpha}, 1\}}
\quad \text{for every \(t \ge 0\)}.
\end{equation}

We have assumed so far that \(\mu = f \dif x\) with \(f\) bounded.{}
For an arbitrary nonnegative measure \(\mu \in \cM(\Omega)\), we apply the estimate to the function \(u_{k} \vcentcolon= \zeta_{\rho_{k} * \mu}\)\,, where \((\rho_{k})_{k \in \N}\) is a suitable sequence of mollifiers; see \Cref{remarkDualityApproximation}.
The sequence \((u_{k})_{k \in \N}\) converges to \(u\) in \(L^{1}(\Omega)\) and, for every \(k \in \N\), we have \(u_{k} \ge \zeta_{H(u_{k})}\) almost everywhere in \(\Omega\).{}
The conclusion thus follows as \(k \to \infty\).
\end{proof}

We now complement the comparison principle for distributional solutions \(u\) by showing that \(\zeta_{H(u)}\) is also a distributional solution with datum \(H(u)\).
More precisely, using the stability of duality solutions under truncation of the potential \(V\) and the independence of \(H\) with respect to \(V\), we prove

\begin{proposition}
	\label{propositionExistenceBoundedSupersolutions}
	If \(u\) is the distributional solution of \eqref{eqDirichletProblem} with nonnegative datum \(\mu \in \cM(\Omega)\), then \(\zeta_{H(u)}\) satisfies \eqref{eqEquationSupersolution} with datum \(f = H(u) \in L^{\infty}(\Omega)\), where \(H\) is the bounded continuous function given by \Cref{propositionTestFunctionPositive}.
\end{proposition}

\begin{proof}
	We use the notations of \Cref{lemmaDualitySolutionsDistributionalBis,lemmaVariationalTruncationPointwiseConvergence} for \(u_{k}\) and \(\zeta_{f, k}\)\,, respectively.
	Since \(T_{k}(V)\) is bounded, the function \(w_{k} \vcentcolon= \zeta_{H(u_{k}), k}\) satisfies
	\begin{equation}
		\label{eq2069}
	- \Delta w_{k} + T_{k}(V) w_{k}
	= H(u_{k})
		\quad \text{in the sense of distributions in \(\Omega\).}
	\end{equation}
	By nonnegativity of  \(\mu\), the sequence \((u_{k})_{k \in \N}\) is non-increasing and so is \((w_{k})_{k \in \N}\).{}
	As each \(w_{k}\) is also nonnegative, the sequence \((w_{k})_{k \in \N}\) converges in \(L^{1}(\Omega)\) to some function \(w\).
	Since \(H\) is independent of the potential \(V\), by the comparison principle (\Cref{propositionTestFunctionPositive}) and the nonnegativity of \(w_{k}\) we also have 
	\[{}
	0 \le w_{k} \le u_{k}
	\quad \text{almost everywhere in \(\Omega\).}
	\]
	By \Cref{lemmaDualitySolutionsDistributionalBis}, the sequence \((T_{k}(V)u_{k})_{k \in \N}\) converges to \(Vu\) in \(L^{1}(\Omega)\). 
	Thus, by the Dominated convergence theorem, the sequence \((T_{k}(V)w_{k})_{k \in \N}\) converges to \(Vw\) in \(L^{1}(\Omega)\).{}
	As \(k \to \infty\) in \eqref{eq2069}, we then deduce that \(w\) satisfies \eqref{eqEquationSupersolution} with \(f = H(u)\).{}
	To conclude, observe that since \((w_{k})_{k \in \N}\) is bounded in \(L^{\infty}(\Omega)\) and \((\Delta w_{k})_{k \in \N}\) is bounded in \(L^{1}(\Omega)\), by interpolation the sequence \((w_{k})_{k \in \N}\) is bounded in \(W_{0}^{1, 2}(\Omega)\).{}
	By the closure property in Sobolev spaces, we then have \(w \in W_{0}^{1, 2}(\Omega)\) and \(w = \zeta_{H(u)}\).
\end{proof}

%%%%%%%%%%%%%%%%%%%%%%%%%%%%%%%%%%%%%%%%%%%%%%%%%%%%%%%%%%%%%%%%%%%%%%
%%%%%%%%%%%%%%%%%%%%%%%%%%%%%%%%%%%%%%%%%%%%%%%%%%%%%%%%%%%%%%%%%%%%%%
%%%%%%%%%%%%%%%%%%%%%%%%%%%%%%%%%%%%%%%%%%%%%%%%%%%%%%%%%%%%%%%%%%%%%%

\section{Proofs of \Cref{theoremGreen,theoremGoodMeasuresCharacterization}}
\label{sectionProofThm1}

\begin{proof}[Proof of \Cref{theoremGoodMeasuresCharacterization}. ``\(\Longrightarrow\)'']
	Since \(0 \le \mu\lfloor_{Z}{} \le \mu\), by \Cref{propositionDistributionalExistence} the Dirichlet problem \eqref{eqDirichletProblem} also has a distributional solution \(v\) with measure \( \mu\lfloor_{Z}{}\).{}
	As a consequence of the comparison principle from the previous section, we have \(v = 0\) almost everywhere in \(\Omega\).{}
	Indeed, by \Cref{propositionDistributionImpliesDuality}, \(v\) is also a duality solution and
	\begin{equation}
	\label{eq1681}
	\int_{\Omega}{v f}
	= \int_{\Omega}{\quasi{\zeta_{f}} \dif\mu\lfloor_{Z}}
	\quad \text{for every \(f \in L^{\infty}(\Omega)\).}
	\end{equation}	
	By \Cref{propositionExistenceBoundedSupersolutions}, the function \(\zeta_{H(v)}\) satisfies \eqref{eqEquationSupersolution} with bounded datum \(f = H(v)\) and then, by definition of \(Z\), we have \(\quasi{\zeta_{H(v)}} = 0\) in \(Z\).	
	Thus taking \(f = H(v)\) in \eqref{eq1681}, we get
	\[{}
	\int_{\Omega}{v H(v)}
	= \int_{\Omega}{\quasi{\zeta_{{H(v)}}} \dif\mu\lfloor_{Z}}
	= 0.
	\]
	By positivity of \(H\) on \((0, +\infty)\) we deduce that \(v = 0\) almost everywhere in \(\Omega\).
	Since \(v\) solves an equation with \(\mu\lfloor_{Z}\) in the sense of distributions, for every \(\varphi \in C_{c}^{\infty}(\Omega)\) we have
	\[{}
	\int_{\Omega}{\varphi \dif\mu\lfloor_{Z}}
	= \int_{\Omega}{v \, (- \Delta\varphi + V\varphi)}
	= 0.
	\]
	Hence, \(\mu\lfloor_{Z}{} = 0\) and then \(\mu(Z) = 0\).{}
\end{proof}

\begin{proof}[Proof of \Cref{theoremGoodMeasuresCharacterization}. ``\(\Longleftarrow\)'']
	Let \(w\) be the function provided by \Cref{propositionExistenceTool}:
	\(w\) dominates all distributional subsolutions of \eqref{eqDirichletProblem} and, by \Cref{remarkBoundary}, also satisfies
	\[{}
	 - \Delta w + V \chi_{\Omega \setminus Z} w 
	 \le \mu{}
	\quad \text{in the sense of \((C_{0}^{\infty}(\overline\Omega))'\).}
	\]
	Denoting \(\widetilde{w} \vcentcolon= \chi_{\Omega \setminus Z} w\), we claim that 
	\begin{equation}
	\label{eq1736}
	- \Delta \widetilde{w} 
	+ V\chi_{\Omega \setminus Z} \widetilde{w} \ge \mu{}
	\quad \text{in the sense of \((C_{0}^{\infty}(\overline\Omega))'\).}
	\end{equation}
	Once such a property is established, the weak maximum principle implies that \(w \le \widetilde w\) almost everywhere in \(\Omega\).{}
	By nonnegativity of \(w\), we also have \(w \ge \widetilde w\).{}
	Hence, equality holds and we deduce that 
	\[{}
	w = 0
	\quad \text{almost everywhere in \(Z\)}
	\]
	and
	\[{}
	 - \Delta w + V w 
	= \mu{}
	\quad \text{in the sense of distributions in \(\Omega\).}
	\]

	It thus suffices to prove \eqref{eq1736}.
	We perform this task by analyzing separately the diffuse and concentrated parts of \(\Delta\widetilde{w}\).{}
	Concerning the \emph{diffuse} part, we first observe that the assumption \(\mu(Z) = 0\) and the nonnegativity of \(\mu\) imply that \(\mu\ld\lfloor_{\Omega\setminus Z}{} = \mu\ld\).{}
	Thus, by \Cref{propositionExistenceTool},
	\[{}
	- \Delta \widetilde{w} 
	+ V \chi_{\Omega \setminus Z} \widetilde{w}
	\ge \mu\ld
	\quad \text{in the sense of distributions in \(\Omega\),}
	\]
	hence also in the sense of measures in \(\Omega\).{}
	Then, by comparison between the diffuse parts from both sides,
	\begin{equation}
		\label{eqProofTheorem1Diffuse}
	(- \Delta \widetilde{w})\ld 
	+ V \chi_{\Omega \setminus Z} \widetilde{w} \ge \mu\ld.{}
	\end{equation}

	Concerning the \emph{concentrated} part, 
	we first prove that
	\begin{equation}
		\label{eqProofTheorem1}
			u \le  \chi_{\Omega \setminus Z} \, w
			= \widetilde{w} 
	\quad 
	\text{almost everywhere in \(\Omega\),}
	\end{equation}
	where \(u \in L^{1}(\Omega)\) is the duality solution of \eqref{eqDirichletProblem} associated to \(\mu\).{}
	By \Cref{lemmaDualitySolutionsDistributional}, we have \(u \in W_{0}^{1, 1}(\Omega) \cap L^{1}(\Omega; V \dif x)\) and 
	\[{}
	- \Delta u + Vu 
	\le \mu \quad \text{in the sense of distributions in \(\Omega\).}
	\]
	Thus, by \Cref{propositionExistenceTool}, 
	\begin{equation}
		\label{eq2187}
		u \le w
		\quad \text{almost everywhere in \(\Omega\).}
	\end{equation}
	To prove that 
	\begin{equation}
		\label{eq2180}
	u = 0
	\quad \text{almost everywhere in \(Z\),}
	\end{equation}
	we use \(\chi_{Z}\) as test function in the duality formulation:
	\[{}
	\int_{Z}{u}
	= \int_{\Omega}{u \chi_{Z}}
	= \int_{\Omega}{\quasi{\zeta_{\chi_{Z}}} \dif\mu}.
	\]
	By the orthogonality principle (\Cref{propositionOrthogonality}), we have \(\{\quasi{\zeta_{\chi_{Z}}} > 0\} \subset Z\).
	Since \(\mu = 0\) on \(Z\), we get \(\int_{Z}{u} = 0\)
	which, by nonnegativity of \(u\), implies \eqref{eq2180}. 
	As a consequence of \eqref{eq2187}, \eqref{eq2180} and the nonnegativity of \(w\), \eqref{eqProofTheorem1} follows.
	Next, from the inverse maximum principle and \eqref{eqProofTheorem1}, we get
	\[{}
	(-\Delta \widetilde{w})\lc \ge (-\Delta u)\lc.
	\]
	We recall that, by \Cref{propositionSolutionDualityasDistribution}, \(u\) satisfies
	\[{}
	-\Delta u + Vu = \mu\lfloor_{\Omega \setminus S}{} - \lambda
	\quad{}
	\text{in the sense of distributions in \(\Omega\),}
	\]
	where the measure \(\lambda\) is diffuse, that is, \(\lambda\lc = 0\).
	Since \(S \subset Z\) and \(\mu = 0\) on \(Z\), we have \(\mu\lfloor_{\Omega \setminus S}{} = \mu\).	
	Thus,
	\begin{equation}
	\label{eqTheoremConcentrated}	
	(- \Delta \widetilde{w})\lc \ge (-\Delta u)\lc 
	= \mu\lc.
	\end{equation}
	Since \(\Delta \widetilde{w} = (\Delta \widetilde{w})\ld + (\Delta \widetilde{w})\lc\), a combination of \eqref{eqProofTheorem1Diffuse} and \eqref{eqTheoremConcentrated} gives \eqref{eq1736}, but only in the sense of distributions in \(\Omega\).
	As explained in \Cref{remarkBoundary}, the vanishing average property of \(\widetilde{w}\) then implies \eqref{eq1736}, which completes the proof.
\end{proof}

\begin{proof}[Proof of \Cref{theoremGreen}]
	By \Cref{theoremGoodMeasuresCharacterization}, the Dirichlet problem \eqref{eqDirichletProblem} does not have a distributional solution with \(\mu = \delta_{x}\) and \(x \in Z\).{}
	When \(x \not\in Z\), again by \Cref{theoremGoodMeasuresCharacterization} a distributional solution exists and, by \Cref{propositionDistributionImpliesDuality} and the uniqueness of the duality solution, it must coincide with the duality solution \(G_{x}\).{}
	In this case, if \(w \in W_{0}^{1, 2}(\Omega) \cap L^{\infty}(\Omega)\) satisfies \eqref{eqEquationSupersolution} with \(f \in L^{\infty}(\Omega)\), then by \Cref{remarkDistributionalSolutionsVariational} we have \(w = \zeta_{f}\).{}
	The representation formula~\eqref{eqRepresentationGreen} satisfied by \(\zeta_{f}\) then becomes
	\[{}
	\quasi{w}(x)
	= \quasi{\zeta_{f}}(x)
	= \int_{\Omega}{G_{x} f}
	\quad \text{for every \(x \in \Omega\).}
	\qedhere
	\]
\end{proof}

%%%%%%%%%%%%%%%%%%%%%%%%%%%%%%%%%%%%%%%%%%%%%%%%%%%%%%%%%%%%%%%%%%%%%%
%%%%%%%%%%%%%%%%%%%%%%%%%%%%%%%%%%%%%%%%%%%%%%%%%%%%%%%%%%%%%%%%%%%%%%
%%%%%%%%%%%%%%%%%%%%%%%%%%%%%%%%%%%%%%%%%%%%%%%%%%%%%%%%%%%%%%%%%%%%%%

\section{Green's functions and decomposition of $\Omega \setminus S$}
\label{sectionGreen}

In this section, we utilize Green's function \(G_{x}\) in the duality sense for \(x \in \Omega \setminus S\) to identify the various components of \(\Omega \setminus S\).
The topological properties of these subsets will be investigated in the next two sections.
Observe that, when \(x \in S\), we have 
\[{}
\int_{\Omega}{G_{x} f}
= \quasi{\zeta_{f}}(x)
= 0
\quad
\text{for every \(f \in L^{\infty}(\Omega)\),}
\]
 whence 
 \[{}
 G_{x} = 0
 \quad \text{almost everywhere in \(\Omega\).}
 \]
For \(x \in \Omega \setminus S\), the picture is radically different as we know from \Cref{propositionSolutionDualityasDistribution} that \(G_{x}\) satisfies the equation
\begin{equation}
	\label{eqGreenDistributions}
- \Delta G_{x} + VG_{x}
= \delta_{x} - \lambda{}
\quad \text{in the sense of distributions in \(\Omega\)}
\end{equation}
for some nonnegative diffuse measure \(\lambda \in \cM(\Omega)\) carried by \(S\), where for simplicity we omit the possible dependence of \(\lambda\) on \(x\).
In particular, \(G_{x}\) is a nontrivial locally bounded subharmonic function in \(\Omega \setminus \{x\}\).{}
Hence, the Lebesgue set of \(G_{x}\) is \(\Omega \setminus \{x\}\) and the precise representative \(\quasi{G_{x}}\) is upper semicontinuous in this set.
We can interpret \(x\) as the point where \(G_{x}\) diverges to \(+\infty\).

\begin{definition}
	\label{definitionSetF}
For every \(x \in \Omega \setminus S\), the \emph{superlevel set \(U_{x}\)} is defined by 
\[
	U_{x}
	= \bigl\{y \in \Omega : y = x \ \text{or}\ \quasi{G_{x}}(y) > 0 \bigr\}.
\]
\end{definition}

By the Lebesgue differentiation theorem and the fact that \(G_{x} \not\equiv 0\) for \(x \in \Omega \setminus S\), each superlevel set \(U_{x}\) has positive Lebesgue measure.
We also observe that
\begin{equation}
\label{eqComplementSuperlevel}
\Omega \setminus U_{x}
= \{\quasi{G_{x}} = 0\}.
\end{equation}
We now prove that these superlevel sets yield equivalence classes in \(\Omega \setminus S\)\,:

\begin{proposition}
	\label{propositionEquivalenceClasses}
	For every \(x, y \in \Omega \setminus S\), we have that 
	\[{}
	\text{either}
	\quad
	U_{x} = U_{y}
	\quad \text{or} \quad
	U_{x} \cap U_{y} = \emptyset.{}
	\]
\end{proposition}

Since each \(U_{x}\) has positive Lebesgue measure, for \(x\) running over \(\Omega \setminus S\) one then gets a decomposition of \(\Omega \setminus S\) as a finite or countably infinite disjoint union of sets \(U_{x}\).
The components \(D_{j}\) that arise in \Cref{theoremMaximumPrincipleFull} are the superlevel sets that are contained in \(\Omega \setminus Z\).

We begin by showing that each point of \(\{\quasi{G_{x}} > 0\}\) is a density point of this set:

\begin{lemma}
	\label{lemmaGreenDensity}
	Let \(x \in \Omega \setminus S\).
	For every \(z \in \{\quasi{G_{x}} > 0\}\), we have
	\[{}
	\lim_{r \to 0}{\frac{\bigl|B_{r}(z) \cap \{\quasi{G_{x}} > 0\}\bigr|}{\meas{B_{r}(z)}}}
	= 1.
	\]
\end{lemma}

\begin{proof}[Proof of \Cref{lemmaGreenDensity}]
	Let \(c = \quasi{G_{x}}(z) > 0\).{}
	Given \(\epsilon > 0\), one proceeds as in the proof of \Cref{propositionDensity} using the upper semicontinuity of \(G_{x}\) to find some \(\eta > 0\) such that, for every \(0 < r \le \eta\),
	\[{}
	c - \epsilon{}
	\le \frac{1}{\meas{B_{r}(z)}}\int_{B_{r}(z) \cap \{\quasi{G_{x}} > 0\}}{G_{x}}
	\le (c + \epsilon) \frac{\bigl|B_{r}(z) \cap \{\quasi{G_{x}} > 0\}\bigr|}{\meas{B_{r}(z)}},
	\]
	and then
	\[{}
	\frac{c - \epsilon}{c + \epsilon}
	\le \frac{\bigl|B_{r}(z) \cap \{\quasi{G_{x}} > 0\}\bigr|}{\meas{B_{r}(z)}} 
	\le 1.
	\]
	The conclusion follows letting \(r \to 0\) and then \(\epsilon \to 0\).
\end{proof}

We now prove an orthogonality relation among the superlevel sets \(U_{x}\)\,:

\begin{lemma}
	\label{lemmaGreenSets}
	Let \(x, y \in \Omega\setminus S\) with \(x \ne y\).{}
	If\/ \(\quasi{G_{x}}(y) = 0\), then \(U_{x} \cap U_{y} = \emptyset\).
\end{lemma}

To prove this property we need the symmetry of the Green's function~\cite{Malusa_Orsina:1996}*{Theorem~7.4}:
	For every \(x, y \in \Omega\) with \(x \ne y\),{}
	\[{}
	\quasi{G_{x}}(y)
	= \quasi{G_{y}}(x).
	\]

\begin{proof}[Proof of \Cref{lemmaGreenSets}]
	Let \(y \in \Omega \setminus \{x\}\) with \(\quasi{G_{x}}(y) = 0\).{}
	We first show that
	\begin{equation}
	\label{eqGreenNegligible}
	\{G_{x} > 0\} \cap \{G_{y} > 0\} 
	\quad \text{is negligible.}
	\end{equation}
	To this end, by the comparison principle (\Cref{propositionTestFunctionPositive}) and the representation formula~\eqref{eqRepresentationGreen} we have
	\[{}
	\quasi{G_{x}}(y) 
	\ge \quasi{\zeta_{H(G_{x})}}(y)
	= \int_{\Omega}{G_{y} \, H(G_{x})}.
	\]
	Since the left-hand side vanishes by assumption and the integrand is nonnegative, we have \(G_{y} H(G_{x}) = 0\) almost everywhere in \(\Omega\), and \eqref{eqGreenNegligible} thus holds by positivity of \(H\) on \((0, +\infty)\).	
	
	It follows from \eqref{eqGreenNegligible} and the Lebesgue differentiation theorem that \eqref{eqGreenNegligible} is also satisfied by the precise representatives and then, for every \(z \in \Omega\),
	\[{}
	\frac{\bigl|B_{r}(z) \cap \{\quasi{G_{x}} > 0\}\bigr|}{\meas{B_{r}(z)}}
	+ \frac{\bigl|B_{r}(z) \cap \{\quasi{G_{y}} > 0\}\bigr|}{\meas{B_{r}(z)}}
	\le 1.
	\]
	As \(r \to 0\), we deduce using \Cref{lemmaGreenDensity} that the first quotient converges to \(1\) for \(z \in \{\quasi{G_{x}} > 0\}\), while the second one also converges to \(1\) for \(z \in \{\quasi{G_{y}} > 0\}\).{}
	Therefore, no point in \(\Omega\) can belong simultaneously to both sets, and so their intersection must be empty:
	\[{}
	\{\quasi{G_{x}} > 0\} \cap \{\quasi{G_{y}} > 0\} 
	= \emptyset.
	\]
	Since \(\quasi{G_{y}}(x) = \quasi{G_{x}}(y) = 0\), we also have 
	\[{}
	x \not\in \{\quasi{G_{y}} > 0\}
	\quad \text{and} \quad 
	y \not\in \{\quasi{G_{x}} > 0\}.
	\]
	Therefore, \(U_{x} \cap U_{y} = \emptyset\).
\end{proof}

\begin{proof}[Proof of \Cref{propositionEquivalenceClasses}]
	Assume that \(U_{x} \cap U_{y} \neq \emptyset\) and \(x \ne y\).{}
	We wish to show the equality \(U_{x} = U_{y}\), which, by \eqref{eqComplementSuperlevel}, is equivalent to 
	\begin{equation}
	\label{eqGreenEqualityZeroSet}
	\{\quasi{G_{x}} = 0\} = \{\quasi{G_{y}} = 0\}.	
	\end{equation}
	
	Let us prove the inclusion ``\(\subset\)'' in \eqref{eqGreenEqualityZeroSet}.{}
	To this end, take \(z \in \Omega\) such that \(\quasi{G_{x}}(z) = 0\).{}
	Then, by \Cref{lemmaGreenSets},
	\begin{equation}
		\label{eqUDisjoint}
	U_{x} \cap U_{z} = \emptyset.
	\end{equation}
	As another application of \Cref{lemmaGreenSets}, the assumption \(U_{x} \cap U_{y} \neq \emptyset\) implies that \(\quasi{G_{x}}(y) > 0\), and then \(y \in U_{x}\) by the definition of \(U_{x}\).{}
	In view of \eqref{eqUDisjoint}, we thus have \(y \not\in U_{z}\).
	By symmetry of the Green's function and the definition of \(U_{z}\)\,, we deduce that
	\(\quasi{G_{y}}(z) = \quasi{G_{z}}(y) = 0\). 
	Therefore,
	\[{}
	\{\quasi{G_{x}} = 0\} \subset \{\quasi{G_{y}} = 0\}.
	\]	
	We can now interchange the roles of \(x\) and \(y\) to get the reverse inclusion ``\(\supset\)'' and \eqref{eqGreenEqualityZeroSet} then follows.
\end{proof}

\section{Sobolev-openness of $U_{x}$}
\label{sectionQuasiTopology}

We provide in this section additional properties of the superlevel sets \(U_{x}\) related to the Sobolev-topology induced by the definition below:

\begin{definition}
	\label{definitionQuasiOpen}
	A set \(O \subset \Omega\) is \emph{Sobolev-open} whenever there exists a nonnegative function \(\xi \in W_{0}^{1, 2}(\Omega)\) such that every point in \(\Omega\) is a Lebesgue point of \(\xi\) and 
	\[{}
	O = \{\quasi{\xi} > 0\}.
	\]
\end{definition}

Replacing \(\xi\) in this definition by the truncated function \(T_{1}(\xi)\), one can assume to start with that \(\xi \in  W_{0}^{1, 2}(\Omega) \cap L^{\infty}(\Omega)\).
One verifies that a set \(O \subset \Omega\) is Sobolev-open if and only if \(\Omega \setminus O\) is Sobolev-closed, as defined in the Introduction.
The family of Sobolev-open subsets of \(\Omega\) is stable under finite intersections and countably infinite unions.

As a consequence of the Lebesgue differentiation theorem, if \(O\) is Sobolev-open and non-empty, then \(O\) has positive Lebesgue measure.
We now prove that every point of a Sobolev-open set is a density point:

\begin{proposition}
	\label{propositionQuasiOpenDensityPoint}
	Let \(O \subset \R^{N}\) be a Sobolev-open set.
	For every \(x \in O\), we have
	\[{}
	\lim_{r \to 0}{\frac{\abs{B_{r}(x) \cap O}}{\abs{B_{r}(x)}}}
	= 1.
	\]
\end{proposition}

\begin{proof}
For \(x \in O\) and \(\xi \in W_{0}^{1, 2}(\Omega)\) as in the definition of a Sobolev-open set we have \(\quasi{\xi}(x) > 0\).{}
Since \(\xi = 0\) almost everywhere in \(\Omega \setminus O\),
\[{}
\frac{\abs{B_{r}(x) \setminus O}}{\abs{B_{r}(x)}} \, \quasi{\xi}(x)
= \frac{1}{\abs{B_{r}(x)}} \int_{B_{r}(x) \setminus O}{\abs{\xi - \quasi{\xi}(x)}}
\le \fint_{B_{r}(x)}{\abs{\xi - \quasi{\xi}(x)}}.
\]
As \(r \to 0\), the quantity in the right-hand side converges to \(0\) and then
\[{}
\lim_{r \to 0}{\frac{\abs{B_{r}(x) \setminus O}}{\abs{B_{r}(x)}}}
= 0.
\qedhere
\]
\end{proof}

While every open set in the usual Euclidean topology is Sobolev-open, the converse is not true:

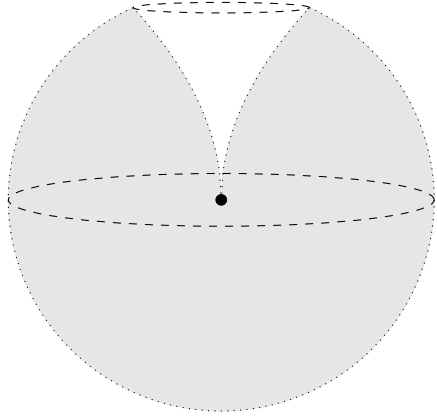
\begin{figure}
\centering
\begin{tikzpicture}[scale=0.7]
	%~ \clip (-4.05,-4.05) rectangle (4.05,4.05);
	\clip (-4.05,-4.05) rectangle (4.05,4);
	\def\fx{sqrt(8 * \x)}
	\filldraw [gray!20] (0,0) -- (1.657,3.641) arc (65.52:-245.52:4) -- cycle;
	\draw [dotted] (1.657,3.641) arc (65.52:-245.52:4);
	\filldraw[smooth,domain=0:1.657,samples=50,gray!20] plot (\x,{\fx}) -- cycle;
	\draw[smooth,dotted,domain=0:1.657,samples=50] plot (\x,{\fx});
	\filldraw[smooth,domain=0:1.657,samples=50,gray!20] plot (-\x,{\fx}) -- cycle;
	\draw[smooth,dotted,domain=0:1.657,samples=50] plot (-\x,{\fx});
	\draw[dashed] (0,0) ellipse (4 and 0.5);
	\draw[dashed] (0,3.641) ellipse (1.657 and 0.1);
	\filldraw (0,0) circle (0.1);
	%~ \draw (0,-2)node{(a)};
\end{tikzpicture}
%~ \includegraphics[width= 0.4 \textwidth]{quasi-open}
%~ \hspace{5ex}
%~ \includegraphics[width= 0.4 \textwidth]{quasi-disconnected}
\caption{Sobolev-open set which is not open.}
\label{figureQuasiOpen}
\end{figure}

\begin{example}
	\label{exampleQuasiOpen}
	Let \(N \ge 3\).{}
	For any given \(0 < \alpha < 1\), the set
	\[{}
	O 
	= \bigl\{x = (x', x_{N}) \in \R^{N-1} \times \R : |x| < 1 \ \text{and}\ x_{N} < |x'|^{\alpha} \bigr\} \cup \{0\}
	\]
	illustrated in \Cref{figureQuasiOpen} is Sobolev-open but not open in \(B_{1}(0)\).
	The assumption \(\alpha < 1\) ensures that \(0\) is a density point of \(O\).{}
	To verify that \(O\) is Sobolev-open, consider the function \(\xi : B_{1}(0) \to \R\) defined for \(x' \ne 0\) by 
	\begin{equation}
		\label{eqExampleQuasiOpen}
	\xi(x)
	= \min{\Bigl\{ \varphi\Bigl(\frac{x_{N}}{|x'|^{\alpha}}\Bigr), 1 - \abs{x}^{2} \Bigr\}},
	\end{equation}
	where \(\varphi : \R \to \R\) is a smooth function such that \(\varphi(t) = 0\) for \(t \ge 1\), \(\varphi(t) = 1\) for \(t \le 1/2\), and \(\varphi(t) > 0\) otherwise.
	Observe that \(\xi\) has a continuous extension to \(B_{1}(0) \setminus \{0\}\), every point in \(B_{1}(0)\) is a Lebesgue point of \(\xi\), and \(\quasi{\xi}(0) = 1\).{}
	The latter is due to the fact that \(\xi(x) = 1 - |x|^2\) on \(\{x_N < |x’|^\alpha/2\}\) and the origin is a density point of this set.	
	Moveover, 
	\[{}
	x \in O
	\quad \text{if and only if} \quad
	\quasi{\xi}(x) > 0.
	\]

	To verify that \(\xi \in W_{0}^{1, 2}(B_{1}(0))\) it suffices to check that \(v(x) = \varphi\bigl({x_{N}}/{|x'|^{\alpha}}\bigr)\) belongs to \(W^{1, 2}(B_{1}(0))\).
	Observe that
	\[{}
	\abs{\nabla v(x)}
	\le C_{1} \Bigabs{\varphi'\Bigl(\frac{x_{N}}{|x'|^{\alpha}}\Bigr)} \, \biggl(\frac{1}{\abs{x'}^{\alpha}} + \frac{\abs{x_{N}}}{\abs{x'}^{\alpha + 1}} \biggr).
	\]
	As \(\varphi'= 0\) outside the interval \((1/2, 1)\) and \(0 < \alpha < 1\), we have
	\[{}
	\abs{\nabla v(x)}
	\le C_{2} \, \chi_{\{{1}/{2} \le {x_{N}}/{|x'|^{\alpha}} \le 1\}}  \biggl(\frac{1}{x_{N}} + \frac{1}{x_{N}^{{1}/{\alpha}}} \biggr){}
	\le 2C_{2} \, \chi_{\{|x'|^{\alpha} \le 2{x_{N}}\}} \frac{1}{x_{N}^{{1}/{\alpha}}}.
	\]
	Thus, by Fubini's theorem,
	\[
	\int_{B_{1}(0)}{\abs{\nabla v}^{2}}
	\le C_{3} \int_{0}^{1}{ \frac{x_{N}^{(N-1)/\alpha}}{x_{N}^{{2}/{\alpha}} } \dif x_{N}}
	= C_{3} \int_{0}^{1}{x_{N}^{(N - 3)/{\alpha} } \dif x_{N}}
	\]
	and the integral in the right-hand side is finite for \(\alpha > 0\).{}
	This implies that \(\xi \in W_{0}^{1, 2}(B_{1}(0))\) and \(O\) is Sobolev-open.
\end{example}

The superlevel sets \(U_{x}\) defined in the previous section are Sobolev-open:

\begin{proposition}
	\label{propositionQuasiOpenFx}
	For every \(x \in \Omega \setminus S\), the set
	 \(U_{x}\) is Sobolev-open and contained in \(\Omega \setminus S\).
\end{proposition}

\begin{proof}
	Since \(\zeta_{\chi_{U_{x}}}\) belongs to \(W_{0}^{1, 2}(\Omega)\) and its Lebesgue set coincides with \(\Omega\), it suffices to prove that 
	\begin{equation}
	\label{eqSetUx}
	U_{x} = \{\quasi{\zeta_{\chi_{U_{x}}}} > 0\}.
	\end{equation}
	To this end, recall that for every \(y \in \Omega\) the representation formula~\eqref{eqRepresentationGreen} gives
	\[{}
	\quasi{\zeta_{\chi_{U_{x}}}}(y)
	= \int_{\Omega}{G_{y} \chi_{U_{x}}}
	= \int_{U_{x}}{G_{y}}.
	\]
	If \(y \in U_{x}\), then by \Cref{propositionEquivalenceClasses} we have \(U_{x} = U_{y}\) and the integral in the right-hand side is thus positive.
	When \(y \not\in U_{x}\), by \Cref{propositionEquivalenceClasses} we have \(U_{x} \cap U_{y} = \emptyset\) and the integral equals zero.
	We conclude that \eqref{eqSetUx} holds and, in particular, \(U_{x} \subset \Omega \setminus S\) by \eqref{eqSetS2}.
\end{proof}

We recall that
\[{}
S \subset Z \subset \Omega,
\]
and then we can decompose \(\Omega\) as
\begin{equation}
	\label{eqDecompositionOmega}
\Omega 
=
S \cup (Z \setminus S) \cup (\Omega \setminus Z).
\end{equation}
The following property implies that \(Z \setminus S\) and \(\Omega \setminus Z\) can be further decomposed as a disjoint union of sets \(U_{x}\).{}
In particular, \(Z \setminus S\) is also a Sobolev-open set.

\begin{proposition}
	\label{propositionSZ}
	For every \(x \in \Omega \setminus S\), we have that 
	\[{}
	\text{either}
	\quad U_{x} \subset Z \setminus S
	\quad \text{or} \quad 
	U_{x} \subset \Omega \setminus Z.
	\]
\end{proposition}

\begin{proof}
	Let \(x \in \Omega \setminus Z\).
	By the representation formula~\eqref{eqRepresentationGreen} and the orthogonality principle (\Cref{propositionOrthogonality}) we have
	\[{}
	\int_{Z}{G_{x}}
	= \quasi{\zeta_{\chi_{Z}}}(x)
	= 0.
	\]
	Thus, \(G_{x} = 0\) almost everywhere in \(Z\).{}
	Similarly, for \(y \in Z\),{}
	\[{}
	\int_{\Omega \setminus Z}{G_{y}}
	= \quasi{\zeta_{\chi_{\Omega \setminus Z}}}(y)
	= 0.
	\]
	Thus, \(G_{y} = 0\) almost everywhere in \(\Omega \setminus Z\).{}
	It then follows for every \(x \in \Omega \setminus Z\) and \(y \in Z \setminus S\) that \(U_{x} \cap U_{y}\) is a Sobolev-open negligible set.
	Hence, \(U_{x} \cap U_{y} = \emptyset\).{}
	In particular, as the superlevel sets are contained in  \(\Omega \setminus S\) (\Cref{propositionQuasiOpenFx}),
	\[{}
	U_{y} \subset (\Omega \setminus S) \setminus \{x\}
	\quad \text{and} \quad{}
	U_{x} \subset (\Omega \setminus S) \setminus \{y\}.
	\]
	Since both inclusions hold for every \(x \in \Omega \setminus Z\) and \(y \in Z \setminus S\),{}
	we then get
	\[{}
	U_{y} \subset Z \setminus S
	\quad \text{and} \quad{}
	U_{x} \subset \Omega \setminus Z.
	\qedhere
	\]
\end{proof}

\begin{figure}
\centering
\begin{tikzpicture}[scale=0.7]
	\draw (0,0) ellipse (6 and 3);
	\draw (2,0.25) ellipse (3 and 1.5);
	\draw (-2,0.25) ellipse (3 and 1.5);
	\draw (2.8,0.25)node{\small Sobolev-open};
	\draw (-2.8,0.25)node{\small fine-open};
	\draw (0,-2)node{\small quasi-open};
\end{tikzpicture}
\caption{Relation among the classes of quasi-, fine- and Sobolev-open sets.}
\label{figureRelationOpenSets}
\end{figure}
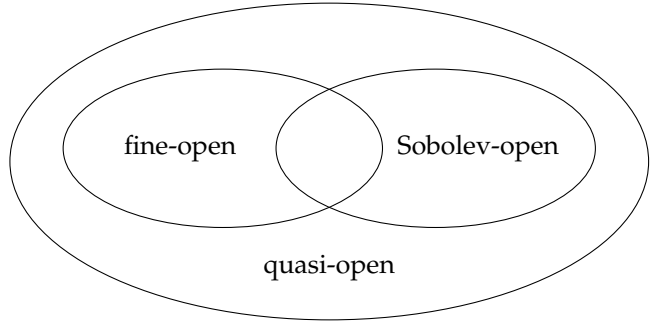

\begin{remark}
	\label{remarkQuasiOpen}
	There are in the literature several other definitions of open sets related to classical concepts of Potential theory, like regular point and capacity.
	For example, fine- and quasi-open sets are of particular interest and we refer the reader to Malý and Ziemer's book~\cite{Maly_Ziemer:1997} for their definitions.
	It is known that every fine-open set is quasi-open; see \cite{Maly_Ziemer:1997}*{Theorem~2.144}.
	In our case, as Sobolev-open sets are of the form \(\{\quasi{\xi} > 0\}\) for some \(\xi \in W_{0}^{1, 2}(\Omega)\) and \(\quasi{\xi}\) is quasicontinuous~\cite{Maly_Ziemer:1997}*{Lemma~2.152}, every Sobolev-open set is also quasi-open.
		
	The classes of fine- and Sobolev-open sets are nevertheless different and one is not contained in the other, which we summarize in \Cref{figureRelationOpenSets}.
	Indeed, any singleton \(\{a\}\) in dimension \(N \ge 2\) has \(W^{1, 2}\)~capacity zero and thus is fine-open (and also quasi-open), but never Sobolev-open.
	In dimension \(N = 3\), the Sobolev-open set \(O\) defined in \Cref{exampleQuasiOpen} is not fine-open as the origin is a regular point of \(\R^{3} \setminus O\)\,: 
	This observation goes back to Lebesgue and is due to the algebraic behavior of the boundary in the neighborhood of the origin; see \cite{Kellogg:1929}*{Chapter~XI, Section~19}.
\end{remark}

\section{Sobolev-connectedness of $U_{x}$}
\label{sectionSobolevConnected}

One can define Sobolev-connected sets in analogy with their classical topological counterpart:

\begin{definition}
	\label{definitionQuasiConnected}
	A set \(D \subset \Omega\) is \emph{Sobolev-connected} whenever, for every disjoint Sobolev-open sets \(A, B \subset \Omega\) such that \(D \subset A \cup B\), one has \(D \subset A\) or \(D \subset B\).
\end{definition}

Since there are more Sobolev-open sets than open sets, any Sobolev-connected set is connected in the usual Euclidean sense.
The converse is false; see \Cref{figureQuasiConnected}~(a) that is related to \Cref{exampleQuasiOpen}.
Using the Intermediate value theorem for Sobolev functions from \cite{VanSchaftingen_Willem:2008}, one verifies that an \emph{open} set is Sobolev-connected if and only if it is connected; see the proof of \Cref{theoremMaximumPrinciplePreciseImproved} below.
Alternatively, one can rely on the fact that such a property is also true for density-connected sets in the density-topology, see~\cite{Goffman_Waterman:1961}, and every Sobolev-open set is density-open by \Cref{propositionQuasiOpenDensityPoint} above.

\begin{example}
	\label{exampleQuasiConnected}
	Let \(N \ge 3\).{}
	For any given \(1 < \alpha < N-1\), the set
	\[{}
	D 
	= \bigl\{x = (x', x_{N}) \in \R^{N-1} \times \R : |x| < 1 \ \text{and}\ \abs{x_{N}} > |x'|^{\alpha} \bigr\} \cup \{0\}
	\]
	is Sobolev-open and Sobolev-connected in \(B_{1}(0)\), but not open for the Euclidean topology; see \Cref{figureQuasiConnected}~(b).
	One proceeds as in \Cref{exampleQuasiOpen} by taking \(\xi : B_{1}(0) \to \R\) defined by \eqref{eqExampleQuasiOpen},
	where the smooth function \(\varphi : \R \to \R\) is now such that \(\varphi(t) = 0\) for \(\abs{t} \le 1/2\), \(\varphi(t) = 1\) for \(\abs{t} \ge 1\), and \(\varphi(t) > 0\) otherwise.
	This new choice of \(\varphi\) ensures that the origin is a Lebesgue point of \(\xi\) for \(\alpha > 1\).
	Moreover, for \(x \in B_{1}(0)\) the function \(v(x) = \varphi\bigl({x_{N}}/{|x'|^{\alpha}}\bigr)\) satisfies
	\[{}
	\abs{\nabla v(x)}
	\le C_{3} \, \chi_{\{{1}/{2} \le {\abs{x_{N}}}/{|x'|^{\alpha}} \le 1\}} \biggl(\frac{1}{\abs{x_{N}}} + \frac{1}{\abs{x_{N}}^{{1}/{\alpha}}} \biggr){}
	\le 2C_{3} \, \chi_{\{|x'|^{\alpha} \le 2 \abs{x_{N}}\}} \frac{1}{\abs{x_{N}}}.
	\]
	Thus, by Fubini's theorem,
	\[
	\int_{B_{1}(0)}{\abs{\nabla v}^{2}}
	\le C_{4} \int_{0}^{1}{ \frac{x_{N}^{(N-1)/\alpha}}{x_{N}^{2} } \dif x_{N}}
	= C_{4} \int_{0}^{1}{x_{N}^{-2 + (N-1)/\alpha } \dif x_{N}}.
	\]
	The right-hand side is finite for \(\alpha < N-1\) and then \(\xi \in W_{0}^{1, 2}(B_{1}(0))\).	
	
	To prove that \(D\) is Sobolev-connected, take disjoint Sobolev-open sets \(A, B \subset \Omega\) such that \(D \subset A \cup B\) and assume that \(0 \in A\). 
	Since \(0\) is a density point of \(D\) but not a density point of \(D_{+} \vcentcolon= D \cap \{x_{N} > 0\}\) nor of \(D_{-} \vcentcolon= D \cap \{x_{N} < 0\}\), we have that \(A\) must intersect both \(D_{+}\) and \(D_{-}\).{}
	Since both sets are open and connected, they are Sobolev-connected and we deduce that \(D_{+}\) and \(D_{-}\) are contained in \(A\).{}
	Therefore, 
	\[{}
	D = \{0\}\cup D_{+} \cup D_{-}  \subset A,
	\]
	which implies that \(D\) is Sobolev-connected.
\end{example}

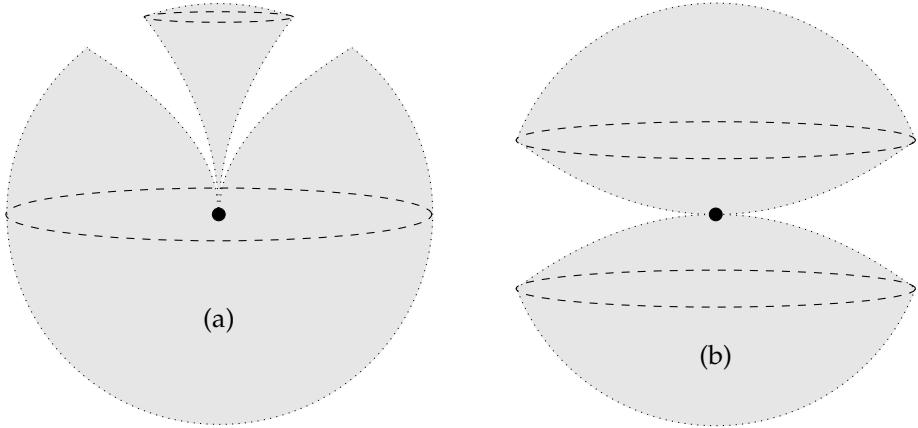
\begin{figure}
\centering{}
\begin{tikzpicture}[scale=0.7]
	\clip (-4.05,-4.05) rectangle (4.05,4.05);
	\def\fx{sqrt(4 * \x)}
	\def\gx{sqrt(10 * \x)}
	\filldraw [gray!20] (0,0) -- (2.5,3.16) arc (51.67:-231.67:4) -- cycle;
	\draw [dotted] (2.5,3.16) arc (51.67:-231.67:4);
	\filldraw [gray!20] (1.4,3.74) arc (69.47:110.53:4) -- cycle;
	\draw [dotted] (1.4,3.74) arc (69.47:110.53:4);
	\filldraw[smooth,domain=0:2.5,samples=50,gray!20] plot (\x,{\fx}) -- cycle;
	\draw[smooth,dotted,domain=0:2.5,samples=50] plot (\x,{\fx});
	\filldraw[smooth,domain=0:2.5,samples=50,gray!20] plot (-\x,{\fx}) -- cycle;
	\draw[smooth,dotted,domain=0:2.5,samples=50] plot (-\x,{\fx});
	\filldraw[smooth,domain=0:1.4,samples=50,gray!20] plot (\x,{\gx}) -- (0,3.74) -- cycle;
	\draw[smooth,dotted,domain=0:1.4,samples=50] plot (\x,{\gx});
	\filldraw[smooth,domain=0:1.4,samples=50,gray!20] plot (-\x,{\gx}) -- (0,3.74) -- cycle;
	\draw[smooth,dotted,domain=0:1.4,samples=50] plot (-\x,{\gx});
	\draw[dashed] (0,0) ellipse (4 and 0.5);
	\draw[dashed] (0,3.74) ellipse (1.4 and 0.1);
	\filldraw (0,0) circle (0.12);
	\draw (0,-2)node{\small (a)};
	%~ \draw (0,-2.5)node{\small (b)};
\end{tikzpicture}
\qquad
\begin{tikzpicture}[scale=0.7]
	\clip (-4.05,-4.05) rectangle (4.05,4.05);
	\def\fx{(\x * \x) / 10}
	\filldraw [gray!20] (3.750,1.406) arc (20.55:159.45:4) -- cycle;
	\draw [dotted] (3.750,1.406) arc (20.55:159.45:4);
	\filldraw [gray!20] (3.750,-1.406) arc (-20.55:-159.45:4) -- cycle;
	\draw [dotted] (3.750,-1.406) arc (-20.55:-159.45:4);
	\filldraw[smooth,domain=-3.75:3.75,samples=50,gray!20] plot (\x,{\fx}) -- cycle;
	\draw[smooth,dotted,domain=-3.75:3.75,samples=50] plot (\x,{\fx});
	\filldraw[smooth,domain=-3.75:3.75,samples=50,gray!20] plot (\x,{-\fx}) -- cycle;
	\draw[smooth,dotted,domain=-3.75:3.75,samples=50] plot (\x,{-\fx});
	\draw[dashed] (0,1.406) ellipse (3.75 and 0.35);
	\draw[dashed] (0,-1.406) ellipse (3.75 and 0.35);
	\filldraw (0,0) circle (0.12);
	\draw (0,-2.7)node{\small (b)};
\end{tikzpicture}
%~ \includegraphics[width= 0.4 \textwidth]{quasi-connected}
\caption{(a) Path-connected set which is not Sobolev-connected; (b) Sobolev-connected set.}
\label{figureQuasiConnected}
\end{figure}

We now show that a Sobolev-connected subset of \(\Omega \setminus S\) cannot intersect two different superlevel sets \(U_{x}\)\,:

\begin{proposition}
	\label{propositionConnected}
	If \(D \subset \Omega \setminus S\) is Sobolev-connected, then \(D \subset U_{x}\) for any \(x \in D\).
\end{proposition}

\begin{proof}
	Since \(\Omega \setminus S\) is a finite or countably infinite disjoint union of the Sobolev-open sets \(U_{y}\) (\Cref{propositionEquivalenceClasses,propositionQuasiOpenFx}), the set
	\(\Omega \setminus (S \cup U_{x})\) is Sobolev-open.
	By Sobolev-connectedness of \(D\) and the inclusion
	\[{}
	D \subset U_{x} \cup \bigl(\Omega \setminus (S \cup U_{x})\bigr),
	\]
	it follows that \(D\) is contained in one of the Sobolev-open sets in the right-hand side.
	For \(x \in D\), we then must have \(D \subset U_{x}\).
\end{proof}

A deeper property concerns the Sobolev-connectedness of all sets \(U_{x}\)\,:

\begin{proposition}
	\label{propositionConnectedF}
	For every \(x \in \Omega \setminus S\), the set \(U_{x}\) is Sobolev-connected.
\end{proposition}

To prove \Cref{propositionConnectedF} we need a counterpart of Poincaré's balayage method for the Schrödinger operator on a non-empty Sobolev-open set \(O \subset \Omega\).
We use the following notation
\[{}
\mathcal{W}(O, \Omega)
= \bigl\{ v \in W_{0}^{1, 2}(\Omega) : {v} = 0 \ \text{almost everywhere in \(\Omega \setminus O\)}  \bigr\}.
\]
Observe that \(\mathcal{W}(O, \Omega)\) contains any function \(\xi\) that verifies the Sobolev-openness of \(O\). 
As a vector space, \(\mathcal{W}(O, \Omega)\) is then nontrivial and also complete with respect to the \(W^{1, 2}\)~norm.

\begin{lemma}
	\label{lemmaBalayage}
	Given a non-empty Sobolev-open set \(O \subset \Omega \setminus S\) and a nonnegative function \(h \in L^{2}(\Omega)\) such that \(h = 0\) almost everywhere in \(\Omega \setminus O\),{}
	let \(u\) be the minimizer of
	\[{}
	E(v)
	= 
	\frac{1}{2} \int_{\Omega}{(\abs{\nabla v}^{2} + Vv^{2})} - \int_{\Omega}{h v}
	\quad \text{in \(\mathcal{W}(O, \Omega) \cap L^{2}(\Omega; V \dif x)\).}
	\]
	Then, there exists a nonnegative locally finite diffuse Borel measure  \(\tau \in L^{1}(\Omega) + (W_{0}^{1, 2}(\Omega))'\) such that 
	\[{}
	\int_{\Omega}{u f}
	= \int_{\Omega}{\zeta_{f} h} - \int_{\Omega}{\quasi{\zeta_{f}} \dif\tau}
	\quad \text{for every \(f \in L^{\infty}(\Omega)\),}
	\]
	with
	\[{}
	\tau(O) = 0
	\]
	and
	\[{}
	\tau(T) = 0
	\quad \text{for every Sobolev-open set \(T \subset \Omega \setminus O\).}
	\] 
\end{lemma}

The measure \(\tau\) can be interpreted as the density of charges in \(\Omega\) needed to obtain a zero potential outside \(O\), starting from a given potential \(u\) that satisfies the equation \(-\Delta u + Vu = h\) in \(O\) and vanishes on the Sobolev-boundary of \(O\).
The properties \(\tau(O) = \tau(T) = 0\) encode the concentration of \(\tau\) on the Sobolev-boundary of \(O\).

\begin{proof}[Proof of \Cref{lemmaBalayage}]
	The minimizer \(u\) exists and satisfies the Euler-Lagrange equation
	\begin{equation}
	\label{eqBalayageWeak}
	\int_{\Omega}{(\nabla u \cdot \nabla v + V uv)} 
	= \int_{\Omega}{h v}
	\quad \text{for every \(v \in \mathcal{W}(O, \Omega) \cap L^{2}(\Omega; V \dif x)\).}
	\end{equation}
	Using various choices of test functions in \eqref{eqBalayageWeak}, one shows that \(u\) is nonnegative, \(u \in L^{1}(\Omega; V \dif x)\) and 
	\begin{equation}
	\label{eqVariationalQuasiOpen}
	- \Delta u + V u 
	\le h
	\quad 
	\text{in the sense of distributions in \(\Omega\).} 
	\end{equation}
	Indeed, taking \(v = \min{\{u, 0\}}\) in \eqref{eqBalayageWeak}, one sees that \(u \ge 0\) almost everywhere in \(\Omega\).{}
	Taking \(v = T_{1}(ku)\) with \(k \in \N\) and letting \(k \to \infty\), one deduces that
	\[{}
	\norm{Vu}_{L^{1}(\Omega)}
	\le \norm{h}_{L^{1}(\Omega)}.
	\] 
	Finally, to show \eqref{eqVariationalQuasiOpen}, one chooses \(v = T_{1}(ku) \varphi \) for any nonnegative \(\varphi \in C_{c}^{\infty}(\Omega)\).{}
	Dropping the nonnegative term \(\nabla u \cdot \nabla T_{1}(ku) \, \varphi\), as \(k \to \infty\) one gets
	\[{}
	\int_{\Omega}{(\nabla u \cdot \nabla \varphi \, \chi_{\{u > 0\}} + V u\varphi)} 
	\le \int_{\Omega}{h \varphi},
	\]
	from which \eqref{eqVariationalQuasiOpen} follows since \(\nabla u = 0\) almost everywhere on \(\{u = 0\}\).{}
	
	By \eqref{eqVariationalQuasiOpen} and a classical property of positive distributions, there exists a nonnegative locally finite Borel measure \(\tau\) in \(\Omega\) such that
	\begin{equation}
	\label{eqPoincareDistributions}
	- \Delta u + V u 
	= h - \tau{}
	\quad 
	\text{in the sense of distributions in \(\Omega\).} 
	\end{equation}
	Since \(u \in W_{0}^{1, 2}(\Omega)\), \(Vu \in L^{1}(\Omega)\) and \(h \in L^{2}(\Omega)\), we have that \(\tau\) is diffuse with respect to the \(W^{1, 2}\)~capacity \cite{Grun-Rehomme:1977} and belongs to \(L^{1}(\Omega) + (W_{0}^{1, 2}(\Omega))'\). 
	The latter property is a general fact satisfied by diffuse measures that has been established in~\cite{Boccardo_Gallouet_Orsina:1996}.{}
	We now prove that
	\begin{equation}
	\label{eqTestFunctionDuality}
	\int_{\Omega}{(\nabla u \cdot \nabla z + Vu z)} 
	= \int_{\Omega}{z h} - \int_{\Omega}{\quasi{z} \dif\tau}
	\quad 
	\text{for every \(z \in W_{0}^{1, 2}(\Omega) \cap L^{\infty}(\Omega)\)}.
	\end{equation}	
	To this end, we write in functional form the action on any \(\varphi \in C_{c}^{\infty}(\Omega)\) in \eqref{eqPoincareDistributions} as
	\[{}
	\int_{\Omega}{(\nabla u \cdot \nabla \varphi + Vu \varphi)} 
	= \int_{\Omega}{\varphi h} - \int_{\Omega}{\varphi \dif\tau}
	= \int_{\Omega}{\varphi h} - \tau[\varphi].
	\]
	Since \(\tau \in L^{1}(\Omega) + (W_{0}^{1, 2}(\Omega))'\), by an approximation of \(z\in W_{0}^{1, 2}(\Omega) \cap L^{\infty}(\Omega)\) with functions in \(C_{c}^{\infty}(\Omega)\) we get
	\[
	\int_{\Omega}{(\nabla u \cdot \nabla z + Vu z)} 
	= \int_{\Omega}{z h} - \tau[z].
	\]
	The identification of \(\tau[z]\) as integration with respect to \(\tau\) then gives \eqref{eqTestFunctionDuality}.
	In particular, for every \(f \in L^{\infty}(\Omega)\) we can apply \eqref{eqTestFunctionDuality} with \(z = \zeta_{f}\).{}
	Using \(u\) as test function in the Euler-Lagrange equation~\eqref{eqEulerLagrange} satisfied by \(\zeta_{f}\) we deduce that
	\[{}
	\int_{\Omega}{u f}
	= \int_{\Omega}{(\nabla u \cdot \nabla \zeta_{f} + V u \zeta_{f})} 
	= \int_{\Omega}{\zeta_{f} h} - \int_{\Omega}{\quasi{\zeta_{f}} \dif\tau}.
	\]
	
	We now prove that
	\(
	\tau(T) = 0
	\)
	holds for every Sobolev-open set \(T \subset \Omega \setminus O\).{}
	Replacing the function \(\xi\) coming from the definition of Sobolev-openness of \(T\) by the truncated function \(T_{1}(\xi)\), we may assume from the beginning that \(\xi \in W_{0}^{1, 2}(\Omega) \cap L^{\infty}(\Omega)\).
	We are thus entitled to take in \eqref{eqTestFunctionDuality} the test function \(z = \xi\).{}
	Since \(\xi = 0\) almost everywhere in \(O\) and \(u = h = 0\) almost everywhere in \(\Omega \setminus O\), we get
	\[{}
	\int_{\Omega}{\quasi{\xi} \dif\tau} 
	= \int_{\Omega}{\xi h} - \int_{\Omega}{(\nabla u \cdot \nabla \xi + Vu \xi)} 
	=  0.
	\]
	As \(\quasi{\xi} > 0\) in \(T\), we conclude that \(\tau(T) = 0\) .	
	
	We are left to prove that \(\tau(O) = 0\).{}
	For this purpose, we now take \(v \in \mathcal{W}(O, \Omega) \cap L^{2}(\Omega; V \dif x) \cap L^{\infty}(\Omega)\), which is an admissible test function for both \eqref{eqBalayageWeak} and \eqref{eqTestFunctionDuality}.{}
	Comparison between both identities gives
	\begin{equation}
	\label{eqLambda}
	\int_{\Omega}{\quasi{v} \dif\tau}
	= 0.
	\end{equation}
	As the function \(\xi\) coming from the definition of the Sobolev-openness of \(O\) belongs to \(\mathcal{W}(O, \Omega)\) and the torsion function \(\zeta_{1}\) belongs to \(W_{0}^{1, 2}(\Omega) \cap L^{2}(\Omega; V \dif x) \cap L^{\infty}(\Omega)\), we may apply \eqref{eqLambda} with \(v \vcentcolon= \min{\{\xi, \zeta_{1}\}}\).{}
	Observe that every point in \(\Omega\) is a Lebesgue point of \(v\) and 
	\begin{equation}
		\label{eq2894}
	\quasi{v} 
	= \min{\{\quasi{\xi}, \quasi{\zeta_{1}}\}},
	\end{equation}
	which is a consequence of the facts that \(\min{\{a, b\}} = a - (a - b)^{+}\) for every \(a, b \in \R\) and composition with Lipschitz functions preserves Lebesgue points.
	Moreover, the assumption \(O \subset \Omega \setminus S\) implies that \(\quasi{\zeta_{1}}(x) > 0\) for every \(x \in O\) and then by \eqref{eq2894} and the choice of \(\xi\) we have \(\quasi{v} > 0\) in \(O\).{}
	As \(\tau\) is nonnegative, we deduce from \eqref{eqLambda} that \(\tau(O) = 0\).{}
\end{proof}

\begin{proof}[Proof of \Cref{propositionConnectedF}]
	Assume that \(U_{x} \subset A \cup B\), where \(A, B \subset \Omega\) are disjoint Sobolev-open sets, and \(A \cap U_{x} \neq \emptyset\).{}
	Since \(A \cap U_{x}\) is Sobolev-open, it has positive Lebesgue measure.
	We then let \(h \vcentcolon= \chi_{A \cap U_{x}}\)\,.{}
	As \(\zeta_{h}\) is a duality solution of \eqref{eqDirichletProblem} with datum \(\mu = h \dif x\), by the representation formula \eqref{eqRepresentationGreen} we have
	\begin{equation}
	\label{eq2491}
	\quasi{\zeta_{h}}(y)
	= \int_{\Omega}{G_{y} h}
	= \int_{A \cap U_{x}}{G_{y}}
	\quad \text{for every \(y \in \Omega\).}
	\end{equation}
	We then observe that
	\begin{equation}
		\label{eq2498}
	\quasi{\zeta_{h}} > 0
	\quad \text{in \(U_{x}\).}
	\end{equation}
	Indeed, since \(U_{y} = U_{x}\) for \(y \in U_{x}\) (by \Cref{propositionEquivalenceClasses}) and \(A \cap U_{x}\) has positive Lebesgue measure, from \eqref{eq2491} we get
	\[{}
	\quasi{\zeta_{h}}(y)
	= \int_{A \cap U_{y}}{G_{y}}
	> 0
	\quad \text{for every \(y \in U_{x}\).}
	\]
	
	In view of \eqref{eq2498}, the proof of \(U_{x} \subset A\) will be complete once we show that
	\begin{equation}
		\label{eq2499}
	\quasi{\zeta_{h}} = 0
	\quad \text{in \(B\).}
	\end{equation}	
	The heart of the matter lies in the following
	
	\begin{Claim}
		\(\zeta_{h} = 0\) almost everywhere in \(B\).
	\end{Claim}
	
	\begin{proof}[Proof of the Claim]
		It suffices to prove that \(\zeta_{h} = u\), where \(u\) is the function given by Poincaré's balayage method with \(h = \chi_{A \cap U_{x}}\) as above and \(O = A \cap U_{x}\).{}
		Indeed, we recall that \(u = 0\) almost everywhere in \(\Omega \setminus O\) and, by the choice of \(O\), we have  
		\[{}
		B \subset \Omega \setminus A \subset \Omega \setminus O.{}
		\]
		By \Cref{lemmaBalayage}, the function \(u\) satisfies
		\begin{equation}
		\label{eqpropositionConnectedF}
		\int_{\Omega}{u f}
		= \int_{\Omega}{\zeta_{f} h} - \int_{\Omega}{\quasi{\zeta_{f}} \dif\tau}
		\quad \text{for every \(f \in L^{\infty}(\Omega)\),}
		\end{equation}
		where \(\tau\) is a nonnegative measure in \(\Omega\).{}
		Let us first show that \(\tau\) is carried by the Sobolev-closed set \(S\), that is,
		\begin{equation}
		\label{eqBalayageZeroMeasure}
		\tau = 0 
		\quad \text{in \(\Omega \setminus S\).}
		\end{equation}
		Since \(\tau\) is nonnegative and, as a consequence of \Cref{propositionEquivalenceClasses}, \(\Omega \setminus S\) can be covered by at most countably many sets \(U_{y}\)\,, it suffices to prove that \(\tau(U_{y}) = 0\) for every \(y \in \Omega \setminus S\).{}
		
		When \(y \not\in U_{x}\), an application of \Cref{propositionEquivalenceClasses} gives
		\[{}
		U_{y} \subset \Omega \setminus U_{x} \subset \Omega \setminus O
		\]
		and one applies \Cref{lemmaBalayage} with \(T = U_{y}\).
		We are left to prove that \(\tau(U_{x}) = 0\).
		To this end, we observe that \(U_{x} \subset O \cup B\).{}
		Thus, by monotonicity and additivity of \(\tau\),{}
		\[{}
		0 \le \tau(U_{x})
		 \le \tau(O \cup B)
		 = \tau(O) + \tau(B).
		\]
		By \Cref{lemmaBalayage} we have \(\tau(O) = 0\).{}
		Since \(B \subset \Omega \setminus O\) is Sobolev-open, once again by \Cref{lemmaBalayage} we have \(\tau(B) = 0\). 
		Thus, \(\tau(U_{x}) = 0\) and \eqref{eqBalayageZeroMeasure} is satisfied.

		Since \(\quasi{\zeta_{f}} = 0\) in \(S\), we thus have
		\[{}
		\int_{\Omega}{\quasi{\zeta_{f}} \dif\tau}
		= \int_{\Omega \setminus S}{\quasi{\zeta_{f}} \dif\tau}
		= 0
		\quad \text{for every \(f \in L^{\infty}(\Omega)\).}
		\]
		Inserting this identity in \eqref{eqpropositionConnectedF}, we conclude that \(u\) is the duality solution of \eqref{eqDirichletProblem} with datum \(\mu = h \dif x\) and then, by uniqueness, \(u = \zeta_{h}\).{}
	\end{proof}
	
	We now proceed with the proof of \eqref{eq2499}.
	By the Claim, for every ball \(B_{r}(x) \subset \Omega\) we have
	\[{}
	0 
	\le \fint_{B_{r}(x)}{\zeta_{h}}
	= \frac{1}{\abs{B_{r}(x)}} \int_{B_{r}(x) \setminus B}{\zeta_{h}}
	\le \frac{\abs{B_{r}(x) \setminus B}}{\abs{B_{r}(x)}} \, \norm{\zeta_{h}}_{L^{\infty}(\Omega)}.
	\]
	Since \(B\) is Sobolev-open, every \(x \in B\) is a density point of \(B\) by \Cref{propositionQuasiOpenDensityPoint}.
	In this case, the right-hand side converges to zero as \(r \to 0\) and we conclude that \(\quasi{\zeta_{h}}(x) = 0\), which is \eqref{eq2499}.{}
	\end{proof}	

%%%%%%%%%%%%%%%%%%%%%%%%%%%%%%%%%%%%%%%%%%%%%%%%%%%%%%%%%%%%%%%%%%%%%%
%%%%%%%%%%%%%%%%%%%%%%%%%%%%%%%%%%%%%%%%%%%%%%%%%%%%%%%%%%%%%%%%%%%%%%
%%%%%%%%%%%%%%%%%%%%%%%%%%%%%%%%%%%%%%%%%%%%%%%%%%%%%%%%%%%%%%%%%%%%%%

\section{Proofs of \Cref{theoremMaximumPrincipleFull} and \Cref{theoremMaximumPrinciplePrecise}}
\label{sectionProofs}

\begin{proof}[Proof of \Cref{theoremMaximumPrincipleFull}]
Each superlevel set \(U_{x}\) from \Cref{definitionSetF} is Sobolev-open (\Cref{propositionQuasiOpenFx}), Sobolev-connected (\Cref{propositionConnectedF}) and \(U_{x} \subset \Omega \setminus Z\) whenever \(x \in \Omega \setminus Z\) (\Cref{propositionSZ}).
Since \(U_{x}\) is non-empty and Sobolev-open, it has positive Lebesgue measure.
Thus, by \Cref{propositionEquivalenceClasses}, the set \(\Omega \setminus Z\) is a finite or countably infinite disjoint union of  components \((D_{j})_{j \in J}\) of the form \(D_{j} = U_{x_{j}}\) for some \(x_{j} \in \Omega \setminus Z\).{}

Uniqueness of the decomposition is based on a standard topological argument.
Indeed, let \((\widetilde D_{i})_{i \in I}\)\/ be another finite or infinite countable decomposition of \(\Omega \setminus Z\) in terms of disjoint Sobolev-connected-open sets.
If \(D_{k} \cap \widetilde D_{l} \neq \emptyset\), then as
\[{}
D_{k} \subset \widetilde D_{l} \cup \bigcup_{i \in I \setminus \{l\}}{\widetilde D_{i}}
\]
and the sets in the right-hand side are disjoint and Sobolev-open,
it follows from the definition of Sobolev-connectedness that \(D_{k} \subset \widetilde D_{l}\)\,.{}
Interchanging the roles of \(D_{k}\) and \(\widetilde D_{l}\)\,, the reverse inclusion also holds.
Hence, both families \((D_{j})_{j \in J}\) and \((\widetilde D_{i})_{i \in I}\) coincide up to a bijection between indices.

It remains to prove that a function \(w \in W_{0}^{1, 2}(\Omega) \cap L^{\infty}(\Omega)\) satisfying \eqref{eqEquationSupersolution} with nonnegative \(f \in L^{\infty}(\Omega)\) is either positive or zero in each component \(D_{j}\).
To this end, take \(x \in D_{j}\).{}
By \Cref{propositionEquivalenceClasses}, we have \(U_{x} = D_{j}\) and then \(G_{x} = 0\) almost everywhere in \(\Omega \setminus U_{x} = \Omega \setminus D_{j}\).{}
By Green's representation formula in \Cref{theoremGreen} we thus have
\begin{equation}
	\label{eqRepresentationComponent}
\quasi{w}(x)
= \int_{\Omega}{G_{x} f}
= \int_{D_{j}}{G_{x} f}
\quad \text{for every \(x \in D_{j}\)}.
\end{equation}
If \(\quasi{w}(x) = 0\) for some \(x \in D_{j}\)\,, then by positivity of \(G_{x}\) in \(U_{x} = D_{j}\) and nonnegativity of \(f\), we must have 
\begin{equation}
	\label{eq3059}
	f = 0
	\quad \text{almost everywhere in \(D_{j}\).}
\end{equation}
By \eqref{eqRepresentationComponent} applied at a point \(y \in D_{j}\) and \eqref{eq3059} we conclude that
\[{}
\quasi{w}(y)
= \int_{D_{j}}{G_{y} f}
= 0
\quad \text{for every \(y \in D_{j}\)}.
\qedhere
\]
\end{proof}

The representation formula \eqref{eqRepresentationComponent} in terms of each Sobolev-connected component \(D_{j}\) makes more transparent the fact that the strong-maximum-principle alternative in \(D_{j}\) is independent of the behavior of the solution in the other components.
Observe that for any given subset of indices \(L \subset J\), by \Cref{theoremGoodMeasuresCharacterization}, see also \Cref{remarkExistenceBounded}, there exists a solution \(w \in W_{0}^{1, 2}(\Omega) \cap L^{\infty}(\Omega)\) of \eqref{eqEquationSupersolution} with \(f = \chi_{B}\) and \(B = \bigcup\limits_{l \in L}{D_{l}}\).
We then deduce from \eqref{eqRepresentationComponent} in this case that \(\quasi{w} > 0\) in \(D_{j}\) if and only if \(j \in L\).

The proof of \Cref{theoremMaximumPrincipleFull} adapts automatically to duality solutions after replacing the universal zero-set \(Z\) by  the zero-set of the torsion function,
\[{}
S = \{\quasi{\zeta_{1}} = 0\}.
\]
As the Sobolev-connected components of \(\Omega \setminus S\) are obtained from all distinct superlevel sets \(U_{x}\) with \(x \in \Omega \setminus S\) (and not only \(x \in \Omega \setminus Z\) as in \Cref{theoremMaximumPrincipleFull}), from \Cref{propositionSZ} they are formed by the collection \((D_{j})_{j \in J}\) of Sobolev-connected components of \(\Omega \setminus Z\) and the Sobolev-connected components of the Sobolev-open set \(Z \setminus S\).{}
In this respect, there can be more components when \(Z \ne S\), but they can never get larger by replacing \(\Omega \setminus Z\) with \(\Omega \setminus S\).
We may then summarize the counterpart of \Cref{theoremMaximumPrincipleFull} for duality solutions as follows:

\begin{theorem}
	The Sobolev-open set \(\Omega \setminus S\) can be uniquely decomposed as a finite or countably infinite family \((D_{j})_{j \in \widetilde J}\)\/ of Sobolev-connected-open sets that contains \((D_{j})_{j \in J}\).{}
	In addition, every duality solution \(\zeta_{f}\) of \eqref{eqDirichletProblem} with \(\mu = f \dif x\) and nonnegative \(f \in L^{\infty}(\Omega)\) satisfies, in each component \(D_{j}\) with \(j \in \widetilde J\),{}
	\begin{center}
		either 
		\quad
		 \(\quasi{\zeta_{f}} > 0\) \ in \(D_{j}\) \quad{}
		or 
		\quad{}
		\(\quasi{\zeta_{f}} \equiv 0\) \ in \(D_{j}\).
	\end{center}	
\end{theorem}

We now present a stronger version of \Cref{theoremMaximumPrinciplePrecise}, where the Hausdorff-measure assumption is made upon \(S\) and gives a sufficient condition for equality with the universal zero-set \(Z\)\,:

\begin{proposition}
	\label{theoremMaximumPrinciplePreciseImproved}
	If\/ \(\cH^{N-1}(S) = 0\) and \(Z \neq \Omega\), then \(S = Z\) and the Sobolev-open set \(\Omega \setminus Z\) is Sobolev-connected. 
	Hence, every solution \(w \in W_{0}^{1, 2}(\Omega) \cap L^{\infty}(\Omega)\) of \eqref{eqEquationSupersolution} for some nonnegative \(f \in L^{\infty}(\Omega)\) with \(\int_{\Omega}{f} > 0\)  satisfies
	\[{}
		\quasi{w}(x) = 0
		\quad \text{if and only if} \quad
		x \in Z.
	\]
\end{proposition}

The assumption \(\cH^{N-1}(S) = 0\) can be verified with the help of \Cref{propositionSingularSetSize}.
To check that \(Z \ne \Omega\), it is enough to know there is a distributional solution of the Dirichlet problem~\eqref{eqDirichletProblem} for some finite nonnegative measure \(\mu \ne 0\), since by \Cref{theoremGoodMeasuresCharacterization} one must have \(\mu(Z) = 0\).

\begin{example}
\label{exampleCapacityS}
If there exists \(v \in W_{0}^{1, 2}(\Omega) \cap L^{2}(\Omega; V \dif x)\) such that 
\[{}
\quasi{v} > 0
\quad \text{quasi-everywhere in \(\Omega\),}
\]
then \(\capt_{W^{1, 2}}{(S)} = 0\).{}
As the measure \(\lambda\) in \Cref{propositionSolutionDualityasDistribution} is diffuse and carried by \(S\), we then have \(\lambda = 0\).{}
Since \(S\) is negligible for the Lebesgue measure, the torsion function \(\zeta_{1}\) satisfies \eqref{eqEquationSupersolution} with \(f \equiv 1\), whence \(Z \ne \Omega\).{}
Thus, by \Cref{theoremMaximumPrinciplePreciseImproved}, we have \(S = Z\) and \(\Omega \setminus Z\) is Sobolev-connected.
\end{example}

\begin{proof}[Proof of \Cref{theoremMaximumPrinciplePreciseImproved}]
	Since \(\Omega\) is connected and \(\cH^{N-1}(S) = 0\), the set \(\Omega \setminus S\) is Sobolev-connected.
	Indeed, let \(A, B \subset \Omega\) be disjoint Sobolev-open sets such that
	\[{}
	\Omega \setminus S \subset A \cup B
	\]
	and assume by contradiction that the sets \(\widetilde{A} \vcentcolon= A \setminus S\) and \(\widetilde{B} \vcentcolon= B \setminus S\) are both non-empty.
	Observe that \(\widetilde{A}\) and \(\widetilde{B}\) are also Sobolev-open and
	\[{}
	\Omega \setminus S = \widetilde{A} \cup \widetilde{B}.
	\]	
	Let \(\xi_{1}, \xi_{2} \in W_{0}^{1, 2}(\Omega)\) be nonnegative functions such that their Lebesgue sets coincide with \(\Omega\) and \(\widetilde{A} = \{\quasi{\xi_{1}} > 0\}\) and \(\widetilde{B} = \{\quasi{\xi_{2}} > 0\}\).{}
	The function \(\quasi{\xi_{1}} - \quasi{\xi_{2}}\) is positive on \(\widetilde{A}\), negative on \(\widetilde{B}\), and vanishes on \(S\).{}
	By the Intermediate value theorem for Sobolev functions~\cite{VanSchaftingen_Willem:2008}*{Proposition~2.11}, we have \(\cH^{N-1}(S) > 0\), which is a contradiction.
	We conclude that \(\Omega \setminus S\) is Sobolev-connected.
	
	It thus follows from \Cref{propositionConnected} that \(\Omega \setminus S \subset U_{x}\) for any \(x \in \Omega \setminus S\). 
	Since \(U_{x} \subset \Omega \setminus S\), equality holds and we can write
	\[{}
	\Omega = S \cup U_{x}.
	\]
	When \(Z \ne \Omega\), we can take \(x \in \Omega \setminus Z\) and deduce from \Cref{propositionSZ} and the decomposition \eqref{eqDecompositionOmega} that 
	\[{}
	U_{x} = \Omega \setminus Z 
	\quad \text{and} \quad{}
	Z \setminus S = \emptyset.{}
	\]
	Therefore, \(S = Z\).
	Since \(\Omega \setminus Z = \Omega \setminus S\) contains only one Sobolev-connected component, the conclusion follows from \Cref{theoremMaximumPrincipleFull}.
\end{proof}

%%%%%%%%%%%%%%%%%%%%%%%%%%%%%%%%%%%%%%%%%%%%%%%%%%%%%%%%%%%%%%%%%%%%%%
%%%%%%%%%%%%%%%%%%%%%%%%%%%%%%%%%%%%%%%%%%%%%%%%%%%%%%%%%%%%%%%%%%%%%%
%%%%%%%%%%%%%%%%%%%%%%%%%%%%%%%%%%%%%%%%%%%%%%%%%%%%%%%%%%%%%%%%%%%%%%

\section{Strong maximum principle for distributional solutions involving measures}
\label{sectionCounterpart}

We prove in this last section a counterpart of \Cref{theoremMaximumPrincipleFull} for distributional solutions of the Dirichlet problem \eqref{eqDirichletProblem}:

\begin{theorem}
	\label{theoremMaximumPrincipleFullMeasure}
	Let \((D_{j})_{j \in \N}\) be the Sobolev-connected components of \(\Omega\setminus Z\).
	If \(u\) is a distributional solution of~\eqref{eqDirichletProblem} for some nonnegative measure \(\mu \in \cM(\Omega)\) and if there exists a Lebesgue point \(x \in D_{j}\) such that \(\quasi{u}(x) = 0\), then
	\begin{center}
	\(\quasi{u} = 0\) \ in \(D_{j}\)
	\quad \text{and} \quad{}
	\(\mu(D_{j}) = 0\).
	\end{center}
\end{theorem}

\begin{proof}
	We first observe that
	\begin{equation}
		\label{eqVanishAE}
		u = 0 \quad \text{almost everywhere in \(D_{j}\).}
	\end{equation}
	To this end, we apply the comparison principle to deduce that \(u \ge \zeta_{H(u)}\) almost everywhere in \(\Omega\), where \(\zeta_{H(u)}\) satisfies \eqref{eqEquationSupersolution} with \(f = H(u)\).
	Since \(\zeta_{H(u)}\) is nonnegative and \(\quasi{u}(x) = 0\), we get \(\quasi{\zeta_{H(u)}}(x) = 0\).{}
	From \eqref{eq3059}, we thus have \(H(u) = 0\) almost everywhere in \(D_{j}\) and then \eqref{eqVanishAE} is satisfied by positivity of \(H\) on \((0, +\infty)\).
	Now, at any Lebesgue point \(y \in D_{j}\), by \eqref{eqVanishAE} and the fact that \(y\) is a density point of \(D_{j}\), we then have \(\quasi{u}(y) = 0\).{}
\end{proof}

To prove that \(\mu(D_{j}) = 0\), we first need a weak form of Green's representation formula for general duality solutions of \eqref{eqDirichletProblem}:

\begin{lemma}
	\label{lemmaRepresentationMeasure}
	If \(u\) is a duality solution of \eqref{eqDirichletProblem} with nonnegative datum \(\mu \in \cM(\Omega)\), then for almost every \(x \in \Omega\) we have \(\mu(\{x\}) = 0\) and
	\[{}
	\quasi{u}(x)
	= \int_{\Omega}{\quasi{G_{x}} \dif\mu}.
	\]
\end{lemma}

\begin{proof}[Proof of \Cref{lemmaRepresentationMeasure}]
	Let \((\rho_{k})_{k \in \N}\) be a sequence of mollifiers and let \(u_{k}\) be the duality solution associated to \(\rho_{k} * \mu\).{}
	Passing to a subsequence if necessary, by \Cref{remarkDualityApproximation} we have that \((\quasi{u_{k}})_{k \in \N}\) converges to \(\quasi{u}\) in \(L^{1}(\Omega)\) and everywhere in \(\Omega \setminus E_{1}\) for some negligible set \(E_{1}\).
	By the representation formula \eqref{eqRepresentationGreen} for bounded data,
	\[{}
	\quasi{u_{k}}(x)
	= \int_{\Omega}{G_{x} \, \rho_{k} * \mu}
	\quad \text{for every \(x \in \Omega\).}
	\]
	When \(x \in S\), we have \(\quasi{u_{k}}(x) = 0\) and \(G_{x} = 0\) almost everywhere in \(\Omega\).{}
	Thus, for every \(x \in S \setminus E_{1}\), it follows that \(\quasi{u}(x) = 0\) and the representation formula is satisfied almost everywhere in \(S\).{}
	
	When \(x \not\in S\), we first apply Fubini's theorem,
	\begin{equation}
	\label{eqGreenApproximation}
	\quasi{u_{k}}(x)
	= \int_{\Omega}{\widecheck{\rho_{k}} * G_{x} \dif \mu}.
	\end{equation}
	Taking \(\rho_{k}\) of the form \(\rho_{k}(z) = \frac{1}{r_{k}^{N}} \rho(\frac{z}{r_{k}})\), where \(\rho \in C_{c}^{\infty}(\Omega)\) and \((r_{k})_{k \in \N}\) converges to zero, we have
	\begin{equation*}
	(\widecheck{\rho_{k}} * G_{x})(y) \to \quasi{G_{x}}(y)
	\quad \text{for every \(y \in \Omega \setminus \{x\}\),}
	\end{equation*}
	since every point in \(\Omega \setminus \{x\}\) is a Lebesgue point of \(G_{x}\).
	To apply the Dominated convergence theorem in \eqref{eqGreenApproximation}, we first observe that, for every \(x \in \Omega\),
	\begin{equation}
		\label{eqGreenFundamentalSolution}
		0 \le G_{x}(a) \le F(x - a)
		\quad \text{for almost every \(a \in \Omega\),}
	\end{equation}
	where \(F\) is the fundamental solution of the Laplacian:
	\[{}
	F(z) =
	\begin{cases}
		\displaystyle \frac{1}{\gamma_{N}} \frac{1}{\abs{z}^{N - 2}}
		\quad \text{if \(N \ge 3\),}\\[2ex]
		\displaystyle \frac{1}{2\pi} \log{\frac{d}{\abs{z}}}
		\quad \text{if \(N = 2\),}
	\end{cases}
	\]
	and we take \(d > \diam{(\Omega)}\)  in dimension two to make sure that \(F(x - y) > 0\) for every \(x, y \in \overline{\Omega}\).
	The second inequality in \eqref{eqGreenFundamentalSolution} follows from the weak maximum principle since \(G_{x} \in W_{0}^{1, 1}(\Omega)\) satisfies \eqref{eqGreenDistributions} for some nonnegative measure \(\lambda\) and \(F(\cdot - a)\) is a positive function on \(\overline\Omega\) that satisfies the Poisson equation with \(\delta_{a}\)\,; see \cite{Ponce:2016}*{Example~6.2}. 
	
	Assuming that \(\rho\) is radial, by \eqref{eqGreenFundamentalSolution} and superharmonicity of \(F(x - \cdot)\), we then have
	\begin{equation}
		\label{eq3237}
	0 \le (\widecheck{\rho_{k}} * G_{x})(y) \le F(x - y)
	\quad \text{for every \(y \in \Omega\).}
	\end{equation}
	Let \(E_{2} \subset \Omega\) be a negligible set such that \((F * \mu)(x) < \infty\) for every \(x \in \Omega \setminus E_{2}\).
	For such a point \(x\), \(\mu(\{x\}) = 0\), the function \(F(x - \cdot)\) is summable with respect to \(\mu\) and, by \eqref{eq3237}, we can apply the Dominated convergence theorem to get
	\[{}
	\lim_{k \to \infty}{\int_{\Omega}{\widecheck{\rho_{k}} * G_{x} \dif \mu}}
	= \int_{\Omega}{\quasi{G_{x}} \dif\mu}
	\quad \text{for every \(x \in \Omega \setminus (S \cup E_{2})\).}
	\]
	We deduce the representation formula for every \(x \in \Omega \setminus (S \cup E_{1} \cup E_{2})\) as \(k \to \infty\) in \eqref{eqGreenApproximation}.
\end{proof}

\begin{proof}[Proof of \Cref{theoremMaximumPrincipleFullMeasure} completed]
	Take a Lebesgue point \(y \in D_{j}\) such that \(\quasi{u}(y) = 0\), \(\mu(\{y\}) = 0\) and the representation formula in \Cref{lemmaRepresentationMeasure} holds.
	Then,
	\[{}
	\int_{\Omega}{\quasi{G_{y}} \dif\mu}
	= \quasi{u}(y)
	= 0.
	\]
	This implies that \(\mu(\{\quasi{G_{y}} > 0\}) = 0\).{}
	Since by \Cref{propositionEquivalenceClasses} we have 
	\[{}
	D_{j} = U_{y} = \{y\} \cup \{\quasi{G_{y}} > 0\},
	\]
	we conclude using the additivity of \(\mu\) that
	\[{}
	\mu(D_{j})
	= \mu(\{y\}) + \mu(\{\quasi{G_{y}} > 0\}) 
	= 0.
	\qedhere
	\]
\end{proof}

Although all distributional solutions with bounded datum vanish on \(Z\), the same need not be true in the case of measures or even \(L^{1}\)~functions:{}

\begin{example}
	For \(N \ge 3\), take \(V(x) = 1/\abs{x - a}^{2}\) for some fixed \(a \in \Omega\).{}
	Any nontrivial superharmonic function \(\psi \in C_{0}^{\infty}(\overline{\Omega})\) satisfies the Schrödinger equation
	\[{}
	- \Delta\psi + V\psi{}
	=: f
	\quad \text{in the sense of distributions in \(\Omega\),}
	\]
	where the function \(f\) is nonnegative and belongs to \(L^{p}(\Omega)\) for every \(1 \le p < N/2\).{}
	However,
	\[{}
	Z = \{a\}
	\quad \text{and} \quad \psi(a) > 0.
	\]
\end{example}

From our proof of \Cref{theoremGoodMeasuresCharacterization}, see \eqref{eq2180}, we know that all distributional solutions vanish almost everywhere in \(Z\).{}
We now show the stronger property that this actually holds except for a set of \(W^{1, 2}\)~capacity zero:

\begin{proposition}
	If \(u\) is a distributional solution of \eqref{eqDirichletProblem} for some nonnegative measure \(\mu \in \cM(\Omega)\), then
	\[{}
	\quasi{u} = 0
	\quad \text{quasi-everywhere in \(Z\).}
	\]
\end{proposition}

\begin{proof}
	We first assume that the Newtonian potential \(F * \mu\) is bounded.
	In particular, \(\mu\) is diffuse with respect to the \(W^{1, 2}\)~capacity and also belongs to \((W_{0}^{1, 2}(\Omega))'\).{}
	In this case, \(u \in W_{0}^{1, 2}(\Omega) \cap L^{\infty}(\Omega)\) and every element in this space is an admissible test function.
	
	Given a sequence of mollifiers \((\rho_{k})_{k \in \N}\), let \(u_{k}\) be the distributional solution of \eqref{eqDirichletProblem} with datum \(\chi_{\Omega \setminus Z}(\rho_{k} * \mu)\), which exists by \Cref{theoremGoodMeasuresCharacterization}; see also \Cref{remarkExistenceBounded}.{}
	We claim that \((u_{k})_{k \in \N}\) converges to \(u\) in \(W_{0}^{1, 2}(\Omega)\).{}
	To this end, we apply \(u_{k} - u\) as test function in the equation satisfied by \(u_{k} - u\) to get
	\begin{equation}
	\label{eq3211}
	\int_{\Omega}{\abs{\nabla(u_{k} - u)}^{2}} + \int_{\Omega}{V (u_{k} - u)^{2}}
	= \int_{\Omega \setminus Z}{(u_{k} - u) \rho_{k} * \mu} - \int_{\Omega}{(\quasi{u_{k}} - \quasi{u}) \dif\mu}.
	\end{equation}
	We write the action of \(\mu\) as an element of the dual \((W_{0}^{1, 2}(\Omega))'\) in the form
	\[{}
	\int_{\Omega}{(\quasi{u_{k}} - \quasi{u}) \dif\mu}
	= \mu[u_{k} - u].
	\]
	Since the sequence \((u_{k})_{k \in \N}\) is bounded in \(W_{0}^{1, 2}(\Omega)\) and converges to \(u\) in \(L^{1}(\Omega)\), we have weak convergence in \(W_{0}^{1, 2}(\Omega)\) and then
	\[{}
	\lim_{k \to \infty}{\int_{\Omega}{(\quasi{u_{k}} - \quasi{u}) \dif\mu}}
	= 0.
	\]
	We next recall that \(u_{k} = u = 0\) almost everywhere in \(Z\).{}
	Thus, using Fubini's theorem,
	\[{}
	\int_{\Omega \setminus Z}{(u_{k} - u) \rho_{k} * \mu}
	= \int_{\Omega}{(u_{k} - u) \rho_{k} * \mu}
	= \int_{\Omega}{\widecheck{\rho_{k}} * (u_{k} - u)  \dif\mu}.
	\]
	The sequence \((\widecheck{\rho_{k}} * (u_{k} - u))_{k \in \N}\) converges weakly to zero in \(W^{1, 2}(\R^{N})\), and then one has as before,
	\[{}
	\lim_{k \to \infty}{\int_{\Omega \setminus Z}{(u_{k} - u) \rho_{k} * \mu}}
	= 0.
	\]
	As \(k \to \infty\) in \eqref{eq3211},  we get
	\[{}
	\lim_{k \to \infty}{\int_{\Omega}{\abs{\nabla(u_{k} - u)}^{2}}} 
	= 0,
	\]
	which implies the claim.
	Now passing to a subsequence \((u_{k_{j}})_{j \in \N}\), one deduces that 
	\[{}
	\quasi{u_{k_{j}}} \to \quasi{u}
	\quad \text{quasi-everywhere in \(\Omega\).}
	\]
	Since every \(u_{k_{j}}\) satisfies an equation in the sense of distributions with bounded datum, we have
	\[{}
	\quasi{u_{k_{j}}} = 0
	\quad \text{in \(Z\).}
	\]
	The conclusion thus follows when \(F * \mu\) is bounded.
	
	In the case of a general nonnegative measure \(\mu\), it suffices to prove that the truncated function \(T_{1}(u)\) satisfies the conclusion.
	Observe that \(T_{1}(u) \in W_{0}^{1, 1}(\Omega) \cap L^{1}(\Omega; V \dif x)\) and
	\[{}
	- \Delta T_{1}(u) + V T_{1}(u)
	 = \widetilde{\mu}
	 \quad \text{in the sense of distributions in \(\Omega\),}
	\]
	for some nonnegative diffuse measure \(\widetilde\mu \in \cM(\Omega)\)\,; see \cites{DalMaso_Murat_Orsina_Prignet:1999,Brezis_Ponce:2004,Brezis_Ponce:2008}.{}
	By a classical property in Potential theory~\cite{Helms:2009}*{Theorem~3.6.3}, this measure can be strongly approximated in \(\cM(\Omega)\) by a nondecreasing sequence of measures \((\nu_{k})_{k \in \N}\) with bounded Newtonian potential, for which the conclusion holds from the first part of the proof.
	This implies the theorem as the distributional solutions \(v_{k}\) of \eqref{eqDirichletProblem} with data \(\nu_{k}\) converge strongly to \(T_{1}(u)\) in \(W_{0}^{1, 2}(\Omega)\) and, for each \(k \in \N\), they satisfy \(\quasi{v_{k}} = 0\) quasi-everywhere in \(\Omega\). 
\end{proof}

\section*{Acknowledgements}

%~ The authors would like to thank Baptiste Devyver, Vitaly Moroz and Yehuda Pinchover for stimulating discussions and Nicolas Wilmet for the careful reading of the paper.
The second author (ACP) was supported by the Fonds de la Recherche scientifique--FNRS under research grant  J.0020.18. 
He warmly thanks the Dipartimento di Matematica of the ``Sapienza'' Universit\`a di Roma and the Math Department of the Technion (Haifa) for the invitations. 
He also acknowledges the hospitality of the Academia Belgica in Rome.

%%%%%%%%%%%%%%%%%%%%%%%%%%%%%%%%%%%%%%%%%%%%%%%%%%%%%%%%%%%%%%%%%%%%%%
%%%%%%%%%%%%%%%%%%%%%%%%%%%%%%%%%%%%%%%%%%%%%%%%%%%%%%%%%%%%%%%%%%%%%%
%%%%%%%%%%%%%%%%%%%%%%%%%%%%%%%%%%%%%%%%%%%%%%%%%%%%%%%%%%%%%%%%%%%%%%

\end{document}